\documentclass{amsart}
\usepackage{amssymb}
\usepackage{amsmath}
\usepackage{amsthm}

\usepackage{float}

\usepackage[english]{babel}
\usepackage{authblk}
\usepackage{xfrac}
\usepackage{enumitem}
\usepackage{graphicx}
\usepackage{dsfont}
\usepackage{xcolor}
\usepackage{pst-node}
\usepackage{tikz-cd}
\usepackage{titling}
\newtheorem{thm}{Theorem}[section]
\newtheorem{lem}[thm]{Lemma}
\newtheorem{prop}[thm]{Proposition}
\newtheorem{cor}[thm]{Corollary}
\newtheorem{defn}[thm]{Definition}

\newtheorem{exmp}[thm]{Example}

\newtheorem{rem}[thm]{Remark}

\newcommand{\R}{\mathbb{R}}
\DeclareMathOperator{\Ld}{\mathcal{L}^d}

\newcommand{\diag}{\mathop{\mathrm{diag}}}\newcommand{\esssup}{\mathop{\mathrm{ess\,sup}}}\newcommand{\bfh}{\mathbf h}\newcommand{\bfi}{\mathbf i}\newcommand{\bfj}{\mathbf j}\newcommand{\bfk}{\mathbf k}\newcommand{\bfm}{\mathbf m}\newcommand{\bfn}{\mathbf n}\newcommand{\bfp}{\mathbf p}\newcommand{\bfx}{\mathbf x}\newcommand{\bfy}{\mathbf y}\newcommand{\bfz}{\mathbf z}

\newcommand{\bftheta}{{\boldsymbol\theta}}

\DeclareMathSymbol{\shortminus}{\mathbin}{AMSa}{"39}

\providecommand{\keywords}[1]{\textit{GLT matrix-sequences, asymptotic spectral analysis} #1}

\begin{document}
\title{The $*$-algebra of unbounded GLT: construction and theoretical foundations}
\author{Andrea Adriani $^{(1)}$, Alec Jacopo Almo Schiavoni-Piazza $^{(2)}$}
\date{}

\maketitle

\section*{Abstract}

In the present paper, we are concerned with the study of matrix-sequences arising from the discretization of PDEs and FDEs on domains $\Omega \subset \R^d$ with finite measure. When $\Omega$ is either a hypercube or a bounded domain, the theory of Generalized Locally Toeplitz (GLT) sequences and of reduced GLT sequences cover the spectral analysis of the matrix-sequences derived from the approximation of the continuous problem. This work aims to extend the machinery and tools of the GLT apparatus to the case of unbounded domains with finite measure. For any unbounded domain $\Omega \subset \R^d$ with finite measure, we define a new class of sequences, which we call unbounded GLT, and study their spectral properties.

\ \\
\ \\
\medskip

\noindent
$(1)$\, Vanguard Center, Mohammed VI Polytechnic University, Morocco\\
(andrea.adriani@um6p.ma ORCID ID: 0000-0003-3390-7891);\\ 
$(2)$\, Scuola Internazionale Superiore di Studi Avanzati, Via Bonomea 265, 34136 Trieste, Italy\\ (aschiavo@sissa.it ORCID ID: 0000-0003-2652-7081)\\
\ \\
\ \\
\medskip
\noindent
\keywords{Unbounded domains, Spectral distribution of matrix-sequences, Approximating class of sequences, GLT theory, Discretization of PDEs and FDEs.}

\section{Introduction}





In recent years, there has been an increasing interest in the study of discretizations of PDEs on either unbounded or moving domains (see \cite{CCKR18,Coco20} for problems on moving domains, \cite{CS07,CCNR14} for problems on unbounded domains and \cite{ABKST26,dT14,dTEJ18,dTEJ19} for the study of various numerical aspects of the problem in \cite{CS07}).

A powerful tool for studying the discretization of differential problems is represented by the theory of Generalized Locally Toeplitz (GLT) sequences (see the two comprehensive books \cite{glt-book-1,glt-book-2} and the research papers \cite{GLT-block1D,GLT-blockdD} for a recent account of the theory). The main idea behind the GLT machinery is to associate to some suitable matrix-sequences a privileged symbol, which resonates with the hidden structure of the sequences and behaves well with respect to algebraic operations and limits.

Heuristically, a classical GLT matrix-sequence arises from the discretization of
$$\mathcal{L}u=f \quad \text{on} \quad [0,1]^d$$
for some differential operator $\mathcal{L}$. Its canonical symbol $f:[0,1]^d\times [-\pi,\pi]^d \to \mathbb{C}^{s \times t}$ encodes both the geometry of the space where we solve the PDE and the geometry of the operator $\mathcal{L}$ itself.

In the last few years, many contributions have been made to the original theory. In \cite{Barb} the author replaces the classical reference domain $[0,1]^d$ with any (regular enough) subset $\Omega \subseteq [0,1]^d$ constructing the class of reduced GLT over $\Omega$: the idea was already introduced in the first GLT proposal \cite[pp. 398-399]{glt-laa}, while the initial definition and name of reduced GLT were introduced in \cite[Section 3.1.4]{glt-Fourier}. Furthermore, in \cite{BGMS} the authors provide an extension of the original notion of approximating class of sequence (a.c.s.) to the case of rectangular sequences and start the study of sequences showing block structures; the latter have been extensively studied in \cite{AFGS26,AGS26,ASPSC25,BFFS}, where the authors proved distributional results for matrix-sequences with block structure having either rational or irrational ratios with respect to the global dimension; finally, in \cite{gacs} the authors generalizes the concept of a.c.s. to approximating sequences with different dimensions from the approximated one.

In particular, the latter work (\cite{gacs}) is the starting point of the present contribution, which aims to construct various algebras of matrix-sequences related with unbounded domains for which it is possible to infer the spectral properties. The main idea is to associate with unbounded domains of finite measure and regular boundary a suitable class of matrix-sequences. The construction is rather involved and requires many intermediate steps. First, we extend the definition of GLT algebras to general hypercubes and we define a class of natural mappings between them. Secondly, we associate an algebra of reduced GLT to any (regular enough) bounded domain, by applying a suitable restriction operator from the GLT algebra associated to an hypercube containing it. Then, we prove that these algebras still capture some geometry of the domain to which they are associated and that they are still connected by some natural mappings. Finally, we pass to the definition of the unbounded GLT class, exploiting the g.a.c.s. tool, by approximating both the domain (with a regular exhaustion) and the matrix-sequence (with sequences of reduced GLT associated to the exhaustion). 
At each step of the construction, we prove that the sets of matrix-sequences associated with each domain are indeed $*-$algebras over $\mathbb{C}$, closed under $a.c.s.$ convergence and that the symbol map is an isometry with the algebra of measurable functions over the corresponding domain.
In particular, we prove the following crucial result for sequences in the class of unbounded GLT.
\begin{thm}\label{main_thm_intro}
    Let $\Omega$ be an open domain, such that $\Ld(\Omega)<\infty$, $\Ld(\partial\Omega)=0$, and let $\mathcal{G}_{\Omega}$ denote the set of all unbounded GLT over $\Omega$. Consider $\{A_\bfn\}_\bfn,\{A'_\bfn\}_\bfn\in\mathcal{G}_{\Omega}$, with canonical symbols $f,g$, and let $\alpha,\beta\in\mathbb{C}$. Then
\begin{itemize}
    \item[(i)]\label{main_item_i} $\{A_\bfn^*\}_\bfn\in\mathcal{G}_\Omega$ and $ \{A^*_{\bfn}\}_{\bfn}\sim_{\mathrm{GLT}}^{\Omega}\,\bar{f};$
    \item[(ii)]\label{main_item_ii} $\{\alpha A_\bfn+\beta A'_\bfn\}_\bfn\in\mathcal{G}_\Omega$ and $\{\alpha A_{\bfn}+\beta A'_\bfn\}_{\bfn}\sim_{\mathrm{GLT}}^{\Omega}\,\alpha f+\beta g;$
    \item[(iii)]\label{main_item_iii} $\{A_\bfn A'_\bfn\}_\bfn\in\mathcal{G}_\Omega$ and $\{A_{\bfn} A'_\bfn\}_{\bfn}\sim_{\mathrm{GLT}}^{\Omega}\,fg$;
    \item[(iv)] Assuming that $f\neq 0$ almost everywhere, then $\{A_\bfn^{\dagger}\}_\bfn \sim_{\mathrm{GLT}}^{\Omega}\,f^{-1}$, where $A_\bfn^{\dagger}$ denotes the Moore-Penrose pseudo inverse of $A_{\bfn}$.
\end{itemize}
\end{thm}

This, in turn, leads us to our second fundamental contribution, namely, the equivalence between unbounded GLT sequences related to a domain $\Omega$ and the set of measurable functions over it.

\begin{thm}\label{thm_isom_intro}
    Let $\Omega$ be an open domain, such that $\Ld(\Omega)<+\infty$ and $\Ld(\partial\Omega)=0$. Consider the map $\Phi_\Omega:(\sfrac{\mathcal{G}_\Omega}{\sim_d},d_{a.c.s.})\to(L^0(\Omega\times[-\pi,\pi]^d),d_{m,\Omega})$ that associates to any reduced GLT its canonical symbol. Then, $\Phi_\Omega$ is a surjective isometry.
\end{thm}

With this contribution, the authors aim to look at the GLT theory from a different, more geometric, viewpoint. In particular, we expect that considering all unbounded GLT algebras as a whole object will be advantageous in the spectral study of PDEs. For example, the water waves equation (see \cite{L13} for an introduction on the topic) can be formalized as a Hamiltonian system over a moving domain (the region between the wave profile and the bottom of the sea). Using the unbounded GLT world as a background environment, we aim to build a nice framework for studying such kind of moving domain PDEs, with the spectral information associated to the discretized version of the problem that is carried over at each time step.

The paper is organized as follows: in Section \ref{sec-prel} we introduce some well known tools from spectral linear algebra, such as the concept of spectral distribution, approximating class of sequences (a.c.s.), Toeplitz and diagonal sampling matrices, ending with the class of GLT sequences and their isometric equivalence with the space of measurable functions. Section \ref{sec-hypercube} is devoted to the study of GLT algebras over different hypercubes. In Section \ref{sec-reduced-glt}, we introduce various algebras of Reduced GLT sequences. In Section \ref{sec-gacs} we recall the definition of generalized approximating class of sequences (g.a.c.s.), together with the main results from \cite{gacs}. Section \ref{sec-UGLT} represents the core of the present work. Here, we define the class of unbounded GLT sequences, we establish the uniqueness of the canonical symbol and we prove that many algebraic and topological properties hold just as for the usual GLT sequences. In Section \ref{sec:appl}, we introduce the classes of unbounded Toeplitz matrices and unbounded diagonal sampling matrices and perform a numerical test on a model problem together with some visualizations to show the validity of our construction and derivations. Finally, Section \ref{sec-final} is devoted to drawing conclusions and proposing open problems and possible directions for future research.

\section{Notation and preliminaries in asymptotic analysis}\label{sec-prel}

In this section we introduce the main notation that we use throughout the paper, focusing in particular on the concepts of spectral distribution, sequences of multi-level Toeplitz and diagonal sampling matrices and GLT theory (for a comprehensive account of all these concepts we refer the reader to \cite{glt-book-1,glt-book-2}).

\subsection{Spectral distribution of matrix-sequences}

\begin{defn}\label{def-distribution}
Let $\{A_n\}_n$ be a matrix-sequence, with $A_n$ of size $d_n$ increasing, and let $f:\Omega\subset\mathbb R^d\to\mathbb{C}^{r\times r}$ be
a measurable function defined on a set $\Omega$ with $0<\mathcal{L}^d(\Omega)<\infty$.
\begin{itemize}
    \item We say that $\{A_n\}_n$ has a (asymptotic) singular value distribution described by $f$ on $\Omega$, and we write $\{A_n\}_n\sim_\sigma (f,\Omega)$, if
    \begin{equation}\label{distribution:sv-sv}
     \lim_{n\to\infty}\frac1{d_n}\sum_{i=1}^{d_n}F(\sigma_i(A_n))=\frac{1}{\mathcal{L}^d(\Omega)}\int_{\Omega}\frac{\sum_{i=1}^{r}F(\sigma_i(f(\mathbf x)))}{r}{\rm d}\mathbf x,\qquad\forall\,F\in C_c(\mathbb R).
    \end{equation}
    \item We say that $\{A_n\}_n$ has a (asymptotic) spectral (or eigenvalue) distribution described by $f$ on $\Omega$, and we write $\{A_n\}_n\sim_\lambda (f,\Omega)$, if
    \begin{equation}\label{distribution:sv-eig}
     \lim_{n\to\infty}\frac1{d_n}\sum_{i=1}^{d_n}F(\lambda_i(A_n))=\frac{1}{\mathcal{L}^d(\Omega)}\int_{\Omega}\frac{\sum_{i=1}^{r}F(\lambda_i(f(\mathbf x)))}{r}{\rm d}\mathbf x,\qquad\forall\,F\in C_c(\mathbb C).
    \end{equation}
\end{itemize}
If $\{A_n\}_n$ has both a singular value and an eigenvalue distribution described by $f$, we write $\{A_n\}_n\sim_{\sigma,\lambda}\left(f, \Omega\right)$.
\end{defn}

\subsection{Approximating class of sequences}

\begin{defn}\label{def_acs}
Let $\{A_{n}\}_n$ be a square matrix-sequence of size $\{d_n\}_n$, such that $ d_n \nearrow \infty $, and let $\left\{\{B_{n,t}\}_n\right\}_{t}$ be a sequence of matrix-sequences of the same size $\{d_n\}_n$. We say that $\left\{\{B_{n,t}\}_n\right\}_{t}$ is an approximating class of sequences (a.c.s.) for $\{A_{n}\}_n$ 
if the following condition is met: for every $t$ there exists $n_t$ such that, for $n>n_t$,
\begin{equation*} \nonumber
A_n =  B_{n,t} + S_{n,t} + N_{n,t},
\end{equation*}
and
\begin{equation*}
    \textnormal{rank}\left(S_{n,t}\right) \leq c(t) d_n,
\end{equation*}
$$
\left\|N_{n,t}\right\|\leq \omega(t),
$$
where $n_t, c(t), \omega(t)$ depend only on $t$, and
$$
\lim_{t \to \infty} c(t) = \lim_{t \to \infty} \omega(t) =0.
$$
\end{defn}

The a.c.s. approximation techniques are a key ingredient for the computation of spectral symbols, as shown in the next theorem.

 \begin{thm}
 \label{acs}\rm(\cite{acs-laa,Tilliloc})
 Let $\{A_n\}_n$ be a square matrix-sequence of size $d_n$, with $ d_n \nearrow \infty $, and let $\left\{ \{B_{n,t} \}_n \right\}_t$ be an a.c.s. for $ \{ A_n \}_n$. Suppose that $ \{ B_{n,t}\}_n \sim_{\sigma} (f_t, \Omega) $ and $ f_t \to f $ in measure, then $ \{ A_n \}_n \sim_{\sigma} (f, \Omega) $. Furthermore, if all the matrices involved are Hermitian, $ \{B_{n,t} \}_n \sim_{\lambda} (f_t, \Omega) $ and $ f_t \to f $ in measure, then $ \{ A_n \}_n \sim_{\lambda} (f,\Omega). $
\end{thm}

If $\mathcal{M}$ is the space of matrix-sequences $\{A_n\}_n$ of increasing (fixed) size $\{d_n\}_n$, then it can be given the structure of a pseudo-metric space, which induces the a.c.s. convergence. More precisely, given $A_n\in\mathbb{C}^{d_n\times d_n}$, we define
\begin{equation*}
    p(A_n):=\min\limits_{i=1,\ldots,d_n+1}\left\{\frac{i-1}{d_n}+\sigma_i(A_n)\right\},
\end{equation*}
where $\sigma_i(A_n)$, $i=1,\dots,d_n$, are the singular values of $A_n$ ordered non-increasingly and, by convention, $\sigma_{d_n+1}(A_n)=0$. In addition, given $\{A_n\}_n,\{B_n\}_n\in\mathcal{M}$, we can define
\begin{equation*}
    \rho(\{A_n\}_n):=\limsup_{n\to+\infty}p(A_n),
\end{equation*}
and 
\begin{equation*}
    d_{a.c.s.}(\{A_n\}_n,\{B_n\}_n):=\rho(\{A_n-B_n\}_n).
\end{equation*}
In \cite{Barb17}, the author proved the following result.
\begin{thm}\label{completeness_acs_conv}
    The space $(\mathcal{M},d_{a.c.s.})$ is a complete pseudo-metric space and, given $\{A_n\}_n,\{B_n\}_n\in\mathcal{M}$, it holds
    \begin{equation*}
        d_{a.c.s.}(\{A_n\}_n,\{B_n\}_n)=0\quad\iff\quad\{A_n-B_n\}_n\sim_{\sigma} 0.
    \end{equation*}
    Moreover, given $\{A_n\}_n,\{\{B_{n,t}\}_n\}_t\in\mathcal{M}$, the following are equivalent:
    \begin{itemize}
        \item $\lim\limits_{t\to+\infty} d_{a.c.s.}(\{B_{n,t}\}_n,\{A_n\}_n)=0;$
        \item $\{\{B_{n,t}\}_n\}_t$ is an a.c.s. for $\{A_n\}_n$.
    \end{itemize}
\end{thm}
\subsection{Sparsely vanishing and sparsely unbounded matrix-sequences}
Given a matrix-sequence $\{A_n\}_n$, with $A_n$ of increasing size $d_n$, we consider its singular values and study how many of them are asymptotically small or large in modulus. The next definition encodes precisely this behavior.
\begin{defn}[s.v. and s.u. matrix sequences]
    A matrix-sequence $\{A_n\}_n$, of increasing size $\{d_n\}_n$ is said to be sparsely vanishing $\mathrm{(s.v.)}$ if, for every $\varepsilon>0$, there exists $n_\varepsilon$, such that
    \begin{equation*}
        \frac{\#\{i\in\{1,\ldots,d_n\}\,\vert\,\sigma_i(A_n)<\varepsilon\}}{d_n}\leq r(\varepsilon),
    \end{equation*}
    for every $n>n_\varepsilon$, and with $\lim\limits_{\varepsilon\to 0}r(\varepsilon)=0$.
    
    Similarly, a matrix-sequence $\{A_n\}_n$, of size $\{d_n\}_n$ is said to be sparsely unbounded $\mathrm{(s.u.)}$
    if, for every $M>0$, there exists $n_M$, such that
    \begin{equation*}
        \frac{\#\{i\in\{1,\ldots,d_n\}\,\vert\,\sigma_i(A_n)>M\}}{d_n}\leq r'(M),
    \end{equation*}
    for every $n>n_M$, and with $\lim\limits_{M\to +\infty}r'(M)=0$.
\end{defn}
\begin{rem}\label{rem_sv_su}
    Note that if a matrix-sequence $\{A_n\}_n$ is sparsely vanishing, then its pseudo-inverse $\{A_n^\dagger\}_n$ is sparsely unbounded. This is a consequence of the fact that, given a matrix $A\in\mathbb{C}^{d\times d}$, with $\mathrm{rank}(A)=r$, then the singular values of its pseudo-inverse $A^\dagger$ are the reciprocal of the non-zero singular values of $A$, together with $d-r$ additional zero singular values. Note that the converse is not true, the sequence of pseudo-inverses of a sparsely unbounded matrix-sequence may not be sparsely vanishing (e.g., consider a matrix-sequence with constant rank). 
\end{rem}
The next results (see \cite[Proposition 5.4 - Proposition 8.4]{glt-book-1}) show the relation between possessing a symbol function and being sparsely vanishing/unbounded.
\begin{prop}\label{sv_neq_zero_a_e}
    Let $\{A_n\}_n\sim_\sigma(f,\Omega)$. Then, $\{A_n\}_n$ is $\mathrm{s.u.}$ 
    
    In addition, $\{A_n\}_n$ is $\mathrm{s.v.}$ if and only if $f\neq 0$ almost everywhere. 
\end{prop}
Finally, sparsely unbounded matrix-sequences do not alter the property of being zero-distributed, when acting by multiplication. The next lemma is well-known in the literature (see \cite{glt-book-1,glt-book-2}) even though not stated explicitly. For this reason, we include also the proof in the present work.
\begin{lem}\label{preserving_zero_distribution}
    Let $\{A_n\}_n$ and $\{B_n\}_n$ be two matrix-sequences of increasing size $\{d_n\}_n$. Assume that $\{A_n\}_n\sim_\sigma 0$ and that $\{B_n\}_n$ is sparsely unbounded. Then, $\{A_nB_n\}\sim_\sigma 0$.
\end{lem}
\begin{proof}
    Since $\{A_n\}_n\sim_\sigma 0$, for every $t>0$, there exist functions $c(t)$ and $\omega(t)$ and $n_t$, such that, for every $n>n_t$, we can write
   \begin{equation*} 
A_n = S_{n,t} + N_{n,t},
    \end{equation*}
    with 
    \begin{align*}
        \mathrm{rank}(S_{n,t})&\leq c(t)d_n,\\
        \|N_{n,t}\|&\leq \omega(t),\\
        \lim\limits_{t\to+\infty}c(t)&=\lim\limits_{t\to+\infty}\omega(t)=0.
    \end{align*}
    Now, for every $t>0$, we can choose $M_t=(\omega(t))^{-1/2}$ in the definition of sparsely unbounded, and there exists $n_{M_t}$, such that, for every $n>n_{M_t}$, we can write
    \begin{equation*}
        B_n=S'_{n,t}+N'_{n,t},
    \end{equation*}
    with
    \begin{align*}
        &\mathrm{rank}(S'_{n,t})\leq r((\omega(t))^{-1/2})d_n,\\
        &\|N'_{n,t}\|\leq \omega(t)^{-1/2},\\
        &\lim\limits_{t\to+\infty}r((\omega(t))^{-1/2})=0.
    \end{align*}
    More precisely, $S'_{n,t}$ can be computed from $B_n$ by considering the s.v.d. of $B_n$ and isolating the portion of singular values of modulus larger than $\omega(t)^{-1/2}$, while the remaining part $N'_{n,t}$ has norm less than $(\omega(t))^{-1/2}$. 
    As a consequence of previous choices, for every $n>\max(n_t,n_{M_t})$, we have
    \begin{equation*}
        A_nB_n=S_{n,t}S'_{n,t}+S_{n,t}N'_{n,t} + N_{n,t}S'_{n,t}+N_{n,t}N'_{n,t},
    \end{equation*}
    with
    \begin{equation*}
        \mathrm{rank}(S_{n,t}S'_{n,t}+S_{n,t}N'_{n,t} + N_{n,t}S'_{n,t})\leq(2c(t)+r((\omega(t))^{-1/2}))d_n,
    \end{equation*}
    and 
    \begin{equation*}
        \|N_{n,t}N'_{n,t}\|\leq\|N_{n,t}\|\|N'_{n,t}\|\leq \omega(t)(\omega(t))^{-1/2}=(\omega(t))^{1/2}.
    \end{equation*}
    Noticing that
    \begin{equation*}
        \lim\limits_{t\to+\infty} \left(2c(t)+r((\omega(t))^{-1/2})\right)=\lim\limits_{t\to+\infty}(\omega(t))^{1/2}=0,
    \end{equation*}
    we conclude that $\{A_nB_n\}_n\sim_\sigma 0$.
\end{proof}

\subsection{Multi-index notation}

A multi-index $\bfn$ of size $d$, also called a $d$-index, is a row vector in $\mathbb{Z}^d$, whose components are denoted by $n_1,\dots,n_d$. The set of $d$-indices can be ordered with the standard lexicographic ordering, i.e., we write $\bfn \preceq \bfm $ if either $\bfn=\bfm$, or the first index $j$ for which $n_j \neq m_j$ satisfies $n_j < m_j$. Moreover, given two $d$-indices $\bfn$ and $\bfm$, we say that $\bfn \leq \bfm$ if $n_j \leq m_j$ for every $j=1,\dots,d$. We denote by $\bf0,\bf1,\bf2,\dots$ the multi-indices of all zeros, all ones, all twos, etc. We write $\bf m \to \infty$ whenever $\min({\bfm}) \to \infty$. All operations involving $d$-indices that have no meaning in the vector space $\R$ have to be intended in the component-wise sense. For example, $\bfn \bfm = \left( n_1 m_1, \dots, n_d m_d \right)$, $\bfn / \bfm =(n_1/m_1, \dots, n_d/m_d)$, etc.


\subsection{Diagonal sampling and Toeplitz matrix-sequences}
Let $f:[-\pi,\pi]^d \to \mathbb{C}^{r \times r}$ be a function in $L^1([-\pi,\pi]^d)$ and let $\{f_{\bfk}\}_{\bfk \in \mathbb{Z}^d}$ be the Fourier coefficients of $f$, defined as
\begin{equation*}
    f_{\bfk} = \frac{1}{(2 \pi)^d} \int_{[-\pi,\pi]^d} f(\bftheta) \rm{e}^{-i \bfk \cdot \bftheta}  \rm{d}\bftheta \in \mathbb{C}^{r \times r}, \qquad \bfk \in \mathbb{Z}^d,
\end{equation*}
where $\bfk \cdot \bftheta = k_1 \theta_1 + \dots + k_d \theta_d $ and the integrals are computed component-wise. For every $\bfn \in \mathbb{N}^d$, we define the $\bfn$-th Toeplitz matrix generated by $f$ as
\begin{equation*}
    T_{\bfn}(f) = [f_{\bfi-\bfj}]_{\bfi,\bfj=\bf1}^{\bfn}.
\end{equation*}
Now, let $a: [0,1]^d \to \mathbb{C}^{r \times r}$. For every $\bfn \in \mathbb{N}^d$, we define the $\bfn$-th diagonal sampling matrix generated by $a$ as
\begin{equation*}
    D_{\bfn}(a) =\diag_{\bfi=\bf1,\dots,\bfn} a \left( \frac{\bfi}{\bfn} \right).
\end{equation*}
Note that we have the following distributional results for the sequences $\{T_\bfn(f)\}_{\bfn}$ and $\{D_\bfn(a)\}_\bfn$.
\begin{thm}\rm{(\cite{Tillinota})}
    Let $f \in L^1[-\pi,\pi]^d$. Then
    \begin{equation*}
        \{T_\bfn(f)\}_\bfn \sim_{\sigma} (f,[-\pi,\pi]^d).
    \end{equation*}
If, additionally, $f$ is real a.e., then
\begin{equation*}
    \{T_{\bfn}(f)\}_{\bfn} \sim_{\lambda} (f,[-\pi,\pi]^d).
\end{equation*}
\end{thm}
\begin{prop}
Let $a:[0,1]^d \to \mathbb{C}^{r \times r}$ be a continuous almost everywhere function. Then
\begin{equation*}
    \{D_{\bfn}(a)\}_{\bfn} \sim_{\lambda,\sigma} (a,[0,1]^d).
\end{equation*}
\end{prop}

\subsection{The Generalized Locally Toeplitz (GLT) algebra}

A Generalized Locally Toeplitz (GLT) sequence $\{A_{\bfn}\}_\bfn$ is a matrix-sequence equipped with a measurable function $f:[0,1]^{d} \times [-\pi,\pi]^d \to \mathbb{C}$ called GLT symbol. We use the notation $\{A_{\bfn}\}_{\bfn} \sim_{\text{GLT}} f$ to indicate that $\{A_{\bfn}\}_{\bfn}$ is a GLT sequence with symbol $f$. In this subsection we report the main properties of the $*$-algebra of GLT sequences. The definition of GLT sequences can be rather cumbersome and requires the introduction of many accessory tools. For these reasons, we just state the main algebra properties of GLT sequences, which will be of interest in the rest of this paper. In particular, we state those properties that will be preserved by the class of sequences that we construct in the following sections.

For a complete and recent account of the theory of GLT theory we refer to the two comprehensive books \cite{glt-book-1, glt-book-2} and to \cite{GLT-block1D,GLT-blockdD}. For more recent developments of the theory and possible future directions of research we also refer the reader to \cite{AFGS26,AGS26,gacs,ASPSC25,BFFS,Barb,BGMS}.

\begin{enumerate}[label=\textbf{GLT \arabic*}]
\setcounter{enumi}{-1}
\item \label{GLT0} If $\{A_\bfn\}_{\bfn} \sim_{\text{GLT}} f$ then $\{A_\bfn\}_\bfn \sim_{\text{GLT}} g$ if and only if $f=g$ a.e.\\
If $f:[0,1]^d \times[-\pi,\pi]^d \to \mathbb{C}^{r \times r}$ is measurable and $\{\bfn =\bfn(n)\}_n$ is a sequence of $d$-indices such that $\bfn \to \infty$ as $n \to \infty$ then there exists $\{A_\bfn\}_\bfn$ such that $\{A_\bfn\}_\bfn \sim_{\text{GLT}} f$;
\item \label{GLT1} if $\{A_{\bfn}\}_{\bfn} \sim_{\text{GLT}} f$ then $\{A_{\bfn}\}_{\bfn}\sim_{\sigma} \left(f,[0,1]^{d}\times[-\pi,\pi]^d\right)$. If $\{A_{\bfn}\}_{\bfn} \sim_{\text{GLT}} f$ and each $A_{\bfn}$ is Hermitian then $\{A_{\bfn}\}_{\bfn}\sim_{\lambda} \left(f,[0,1]^{d}\times[-\pi,\pi]^d\right)$;
\item \label{GLT2} Suppose $\bfn=\bfn(n) \in \mathbb{N}^d$ with $\bfn \to \infty$ as $n\to \infty$. We have
\begin{itemize}
    \item $\{ T_{\bfn}(g)\}_{\bfn} \sim_{\text{GLT}} f(\bfx,\bftheta)=g(\bftheta)$ if $g \in L^{1}\left([-\pi,\pi]^d\right)$;
    \item $\{ D_{\bfn}(a)\}_{\bfn} \sim_{\text{GLT}} f(\bfx,\bftheta)=a(\bfx)$ if $a$ is a continuous almost everywhere function;
    \item $\{ Z_{\bfn}\}_{\bfn} \sim_{\text{GLT}} f(\bfx,\bftheta)=0$ if and only if $\{ Z_{\bfn}\}_{\bfn} \sim_{\sigma} 0$;
    \end{itemize}
    \item \label{GLT3}  if $\{A_{\bfn}\}_{\bfn}\sim_{\text{GLT}} f$, then
    \begin{itemize}
        \item $\{A^{*}_{\bfn}\}_{\bfn} \sim_{\text{GLT}} \bar{f}$;
        \item if, additionally, $f\neq 0$ a.e., then $\{A^{\dagger}_{\bfn}\}_{\bfn}\sim_{\text{GLT}} f^{-1}$;
    \end{itemize}
    \item \label{GLT4} if $\{A_{\bfn}\}_{\bfn}\sim_{\text{GLT}} f$ and $\{B_{\bfn}\}_{\bfn}\sim_{\text{GLT}} h$, then
    \begin{itemize}
        \item $\{\alpha A_{\bfn}+\beta B_{\bfn}\}_{\bfn}  \sim_{\text{GLT}} \alpha f+\beta h$ for every $\alpha, \beta \in \mathbb{C}$;
        \item $\{A_{\bfn} B_{\bfn}\}_{\bfn}  \sim_{\text{GLT}}  f h$;
    \end{itemize}
    
    \item \label{GLT5} $\{A_{\bfn}\}_{\bfn} \sim_{\text{GLT}} f$ if and only if there exist GLT sequences $\{B_{\bfn,t}\}_{\bfn} \sim_{\text{GLT}} f_t$ such that $\{\{B_{\bfn,t}\}_{\bfn}\}_t$ is an a.c.s. for $\{A_{\bfn}\}_{\bfn}$ and $f_t \to f$ in measure.
\end{enumerate}

\begin{rem}
    Note that the function in Definition \ref{def-distribution} describing the distribution of a matrix-sequence $\{A_\bfn\}_\bfn$ and the relative domain are highly non-unique. However, in GLT theory, the crucial point is to select matrix-sequences $\{A_\bfn\}_\bfn$, which have privileged symbols over a canonical domain. Such symbols are chosen taking into account the hidden asymptotic structure of the sequences and behave well with algebra operations. This, in turn, gives the uniqueness of the GLT symbol (up to a.e. equivalence).
\end{rem}
    Denote by $\mathcal{G}$ the algebra of all GLT sequences, which is naturally equipped with the following equivalence relation:
    \begin{equation*}
        \{A_n\}_n\sim_d\{B_n\}_n\quad\iff\quad d_{a.c.s}(\{A_n\}_n,\{B_n\}_n)=0.
    \end{equation*} 
    Moreover, denote by $\mathcal{L}^0(=\mathcal{L}^0([0,1]^d\times[-\pi,\pi]^d))$ the set of all measurable functions $f:[0,1]^d\times[-\pi,\pi]^d\to\mathbb{C}$ and define, for every $f \in \mathcal{L}^0$,
    \begin{equation*}
        p_m(f):=\inf\limits_{E\subseteq [0,1]^d\times[-\pi,\pi]^d}\left\{\frac{\mathcal{L}^{2d}(E)}{(2\pi)^d}+\esssup_{E^c}|f|\right\},
    \end{equation*}
    where $E$ is measurable, $E^c$ is the complement of $E$ in $[0,1]^d \times [-\pi,\pi]^d$, and
    \begin{equation*}
    d_m(f,g):=p_m(f-g).
    \end{equation*}
    Then, $d_m(f,g)=0$ if and only if $f=g$ almost everywhere (which we denote briefly by $f\sim_d g$), and $d_m$ is a complete pseudo-metric on $\mathcal{L}^0$, inducing the convergence in measure. We denote by $L^0$ the quotient metric space $\sfrac{\mathcal{L}^0}{\sim_d}$.
    The natural map $\Phi:(\mathcal{G},d_{a.c.s})\to(\mathcal{L}^0,d_m)$, that associates to any GLT sequence its canonical symbol, factors through the equivalence relations given by $d_{a.c.s.}$ and $d_m$. Its properties are fully studied in \cite{Barb17}, where the author proved the following crucial result.
    \begin{thm}[\cite{Barb17}]\label{equiv_GLT_meas}
        The map $\Phi:(\sfrac{\mathcal{G}}{\sim_d},d_{a.c.s.})\to(L^0,d_m)$ is a surjective isometry of metric spaces.
    \end{thm}
    Note that Theorem \ref{equiv_GLT_meas} is a substantial improvement of \eqref{GLT5}, which states that the graph of $\Phi$ is closed in the product pseudo-metric space $\mathcal{G}\times\mathcal{L}^0$.
    
    Coupled with all the aforementioned results, Theorem \ref{equiv_GLT_meas} completes the picture of all properties of interest of $\mathcal{G}$, gathered in the following corollary.
    \begin{cor}
        The canonical symbol map $\Phi:(\sfrac{\mathcal{G}}{\sim_d},d_{a.c.s.})\to(L^0,d_m)$ is an isometry of metric spaces and a morphism of $*$-algebras. 
    \end{cor}

\section{GLT algebras over different hypercubes}\label{sec-hypercube}

For a fixed multi-index $\bfn>\bf0$, we define 
\begin{equation*}
	\Theta_{\bfn}:=\biggl\{\frac{\bfx}{\bfn}\,\bigg|\,\bfx\in\mathbb{Z}^d\biggr\}
\end{equation*}

    to be the d-dimensional lattice whose elements have integer multiples of $\frac{\bf1}{\bfn}$ as coordinates. $\Theta_{\bfn}$ is naturally equipped with the lexicographic order $\preceq$. 
    
    Given the reference hypercube $Q_{{\bf0},1}:=\left(0,1\right]^d$, for every couple $(\bfy,l)\in\mathbb{Z}^{d}\times\mathbb{N}$, we define the $d$-dimensional hypercube $Q_{{\bfy},l}:=\bfy+l\cdot Q_{{\bf0},1}$, and
\begin{equation*}
	\Theta_{\bfn,\bfy,l}:=Q_{{\bfy},l}\cap\Theta_{\bfn}=\biggl\{\bfy+\frac{\bfi}{\bfn}\,\bigg|\,{\bf1\preceq\bfi\preceq}l\cdot\bfn\biggr\},
\end{equation*}
the restriction of the grid $\Theta_{\bfn}$ to $Q_{{\bfy},l}$.
Note that $\#\Theta_{\bfn,\bfy,l}=l^d\cdot N(\bfn)$.
    \begin{defn}[GLT algebras on hypercubes]\label{hypercube}
    Given a diverging sequence of $d$-dimensional multi-indices $\{\bfn\}_{\bfn\in\mathbb{N}^d}$, consider the associated classical GLT algebra on $Q_{{\bf0},1}=\left(0,1\right]^d$, denoted by $\mathcal{G}_{{\bf0},1}$. The elements of \,$\mathcal{G}_{{\bf0},1}$ are matrix-sequences $\{A_{\bfn}\}_{\bfn}$, such that $A_{\bfn}\in\mathbb{C}^{N(\bfn)\times N(\bfn)}$, with a canonical GLT symbol $f:Q_{\bf0,1}\times[-\pi,\pi]^d\to\mathbb{C}$.\\
    For $(\bfy,l)\in\mathbb{Z}^{d}\times\mathbb{N}$, define the affine map $\phi_{{\bfy},l}:Q_{\bf0,1}\to Q_{\bfy,l}$ as
	\begin{equation*}
		\phi_{{\bfy},l}(\bfx):=\bfy+l\cdot \bfx,
	\end{equation*}
	\begin{equation*}
		\phi_{{\bfy},l}^{-1}(\bfz):=\frac{1}{l}\left(\bfz-\bfy\right).
	\end{equation*}
	A GLT on $Q_{{\bfy},l}$ with symbol $f:Q_{{\bfy},l}\times\left[-\pi,\pi\right]^d\to\mathbb{C}$ is a matrix-sequence $\{A_{\bfn}\}_{\bfn}$ satisfying $\{A_{\bfn}\}_{\bfn}\sim_{\mathrm{GLT}}\tilde{f}$ in the classical sense, where $\tilde{f}(\bfx,\bftheta)=f(\phi_{{\bfy},l}(\bfx),\bftheta)$. In formulas, we write 
	\begin{equation*}
		\{A_{\bfn}\}_{\bfn}\sim_{\mathrm{GLT}}^{Q_{\bfy,l}}\,f.
	\end{equation*} 
	The set of all the GLT matrix-sequences on $Q_{{\bfy},l}$ is denoted by $\mathcal{G}_{{\bfy},l}$.
\end{defn}
\begin{rem}
	GLT axioms of the standard GLT algebra $\mathcal{G}_{{\bf0},1}$ associated with the sequence $\{l\cdot\bfn\}_{\bfn\in\mathbb{N}^d}$ induce analogous axioms for $\mathcal{G}_{{\bfy},l}$, which is therefore a matrix-sequence algebra on $Q_{{\bfy},l}$.\\  In particular, we have
	\begin{equation*}
		\{A_{\bfn}\}_{\bfn}\sim_{\mathrm{GLT}}^{Q_{{\bfy},l}}\, f \implies \{A_{\bfn}\}_{\bfn}\sim_{\sigma}\left(f,Q_{{\bfy},l}\times\left[-\pi,\pi\right]^d\right),
	\end{equation*} 
	\begin{equation*}
		\{A_{\bfn}\}_{\bfn}\sim_{\mathrm{GLT}}^{Q_{\bfy,l}}\,f\wedge A_{\bfn}A_{\bfn}^*=A_{\bfn}^*A_{\bfn} \text{ } \forall\bfn \implies
		 \{A_{\bfn}\}_{\bfn}\sim_{\lambda}\left(f,Q_{{\bfy},l}\times\left[-\pi,\pi\right]^d\right).
	\end{equation*}
\end{rem}
 In order to build maps between different GLT algebras, we want to extract minors from a matrix-sequence (associated to a bigger hypercube) in such a way that we end up with a new GLT matrix-sequence (associated to a smaller hypercube).
 
 Let us consider two couples $(\bfy_1,l_1)$ and $(\bfy_2,l_2)$ such that $Q_1=Q_{{\bfy_1},l_1}\subset Q_2=Q_{{\bfy_2},l_2}$. Since, in particular, $\bfy_1-\bfy_2\geq\bf0$, we can define the following strictly increasing function\\
\begin{equation*}
	\psi_{\bfn}(=\psi_{{\bfn, \bfy_1},l_1,{\bfy_2},l_2}):\{{\bf1},\ldots\,l_1\cdot\bfn\}\to\{{\bf1},\ldots\,l_2\cdot\bfn\},
\end{equation*}
\begin{equation*}
	\psi_{\bfn}(\bfi):=\bfi+\bfn\cdot(\bfy_1-\bfy_2).
\end{equation*}\\
Identifying the points of $\Theta_{\bfn,\bfy_1,l_1}$ and $\Theta_{\bfn,\bfy_2,l_2}$ with $\{{\bf1},\ldots\,l_1\cdot\bfn\}$ and $\{{\bf1},\ldots\,l_2\cdot\bfn\}$ respectively, the map $\psi_{\bfn}(\bfi)$ selects exactly the points of the mesh of $Q_2$ belonging to $Q_1$ and we can use it to cut rows and colums of a matrix-sequence in $\mathcal{G}_{{\bfy_2},l_2}$ to get a matrix-sequence in $\mathcal{G}_{{\bfy_1},l_1}$. Hence, let us define the family of rectangular matrices $\Pi_{\bfn}(=\Pi_{{\bfn, \bfy_1},l_1,{\bfy_2},l_2})\in\mathbb{C}^{N(l_1\cdot\bfn)\times N(l_2\cdot\bfn)}$ given by\\
\begin{equation*}
	\left(\Pi_{\bfn}\right)_{\bfi,\bfj}=\delta_{\psi_{\bfn}(\bfi),\bfj},\qquad\bfi\in\{{\bf1},\ldots\,l_1\cdot\bfn\},\quad\bfj\in\{{\bf1},\ldots\,l_2\cdot\bfn\},
\end{equation*}\\
where $\delta_{\bfi,\bfj}$ is the multivariate Kronecker delta. 

Using the matrices above, we can define the restriction and extension maps as follows\\
\begin{equation*}
	R_{\bfn}(=R_{\bfn,\bfy_2,l_2,\bfy_1,l_1}):\mathbb{C}^{N(l_2\cdot \bfn)\times N(l_2\cdot\bfn)}\to\mathbb{C}^{N(l_1\cdot \bfn)\times N(l_1\cdot\bfn)},
\end{equation*}
\begin{equation*}
	E_{\bfn}(=E_{{\bfn, \bfy_1},l_1,{\bfy_2},l_2}):\mathbb{C}^{N(l_1\cdot \bfn)\times N(l_1\cdot\bfn)}\to\mathbb{C}^{N(l_2\cdot \bfn)\times N(l_2\cdot\bfn)},
\end{equation*}
as
\begin{equation*}
	R_{\bfn}(A_{\bfn}):=\Pi_{\bfn}A_{\bfn}\Pi_{\bfn}^T,
\end{equation*}
\begin{equation*}
	E_{\bfn}(B_{\bfn}):=\Pi_{\bfn}^T B_{\bfn}\Pi_{\bfn}.
\end{equation*}
It follows directly from their definitions that $E_{\bfn}$ and $R_{\bfn}$ send normal matrices to normal matrices and Hermitian matrices to Hermitian matrices. 

In the following lemma, we prove some easy, but crucial, properties of the maps defined above.
\begin{lem}\label{lemma_1}
	 Let $Q_1=Q_{{\bfy_1},l_1}\subset Q_2=Q_{{\bfy_2},l_2}$ be two hypercubes. For any multi-index $\bfn$, we have 
	\begin{equation}\label{lemma_1_eq_1}
	 	\Pi_{\bfn}\Pi_{\bfn}^T = I_{\bfn}\in\mathbb{C}^{N(l_1\cdot \bfn)\times N(l_1\cdot\bfn)},
	\end{equation}
	\begin{equation}\label{lemma_1_eq_2}
		\Pi_{\bfn}^T \Pi_{\bfn} = D_{l_2\bfn}(\mathds{1}_{Q_1})\in\mathbb{C}^{N(l_2\cdot \bfn)\times N(l_2\cdot\bfn)},
	\end{equation}
	where $I_{\bfn}$ is the identity matrix of size $N(l_1\bfn)$ and $D_{l_2\bfn}(\mathds{1}_{Q_1})$ is the $l_2\bfn$-th matrix of the diagonal matrix-sequence $\{D_{l_2\bfn}(\mathds{1}_{Q_1})\}_\bfn\in\mathcal{G}_{{\bfy_2},l_2}$, generated by the characteristic function of $Q_1$, namely
	\begin{equation*}
		D_{l_2\bfn}(\mathds{1}_{Q_1})=\diag\biggl(\mathds{1}_{Q_1}\biggl(\bfy_2+\frac{\bfi}{\bfn}\biggr)\biggr)_{{\bfi}={\bf1},\ldots,l_2\cdot\bfn}.
	\end{equation*}
	As a consequence, for square matrices $A_{\bfn}$ of size $N(l_2\bfn)$ and $B_{\bfn}$ of size $N(l_1\bfn)$, we have:
	\begin{equation*}
		R_{\bfn}(E_{\bfn}(B_{\bfn}))=B_{\bfn},
	\end{equation*}
	\begin{equation*}
		E_{\bfn}(R_{\bfn}(A_{\bfn}))=D_{l_2\bfn}(\mathds{1}_{Q_1})A_{\bfn}D_{l_2\bfn}(\mathds{1}_{Q_1}).
	\end{equation*}

\end{lem}


\begin{proof}
	Fix $Q_1$, $Q_2$ and the index $\bfn$. By direct computation, we have
	\begin{equation*}
		(\Pi_{\bfn}\Pi_{\bfn}^T)_{\bfi,\bfj}=\sum_{\bfk=\bf1}^{l_2\cdot\bfn}(\Pi_{\bfn})_{\bfi,\bfk}(\Pi_{\bfn}^T)_{\bfk,\bfj}=\sum_{\bfk=\bf1}^{l_2\cdot\bfn}\delta_{\psi_{\bfn}(\bfi),\bfk}\delta_{\bfk,\psi_{\bfn}(\bfj)}=\delta_{\psi_{\bfn}(\bfi),\psi_{\bfn}(\bfj)}=\delta_{\bfi,\bfj},
	\end{equation*}
	where the last equality holds since $\psi_{\bfn}$ is injective. Similarly,
	\begin{equation*}
		(\Pi_{\bfn}^T \Pi_{\bfn})_{\bfi,\bfj}=\sum_{\bfk=\bf1}^{l_2\cdot\bfn}(\Pi_{\bfn}^T)_{\bfi,\bfk}(\Pi_{\bfn})_{\bfk,\bfj}=\sum_{\bfk=\bf1}^{l_2\cdot\bfn}\delta_{\bfi,\psi_{\bfn}(\bfk)}\delta_{\psi_{\bfn}(\bfk),\bfj}=\delta_{\bfi,\bfj}\mathds{1}_{Q_1}(\bfy_2+\frac{\bfi}{\bfn}),
	\end{equation*}
	where, as above, last equality holds since the last sum gives $\delta_{\bfi,\bfj}$, but only if $\bfi,\bfj\in \mathrm{Im}(\psi_{\bfn})$. This corresponds precisely to the fact that the points $\bfy_2+\frac{\bfi}{\bfn}$ and $\bfy_2+\frac{\bfj}{\bfn}$ belong to $Q_1$. 
    
    Finally, for any matrices $A_\bfn$ and $B_\bfn$ of suitable size, we get
	\begin{equation*}
		R_{\bfn}(E_{\bfn}(B_{\bfn}))=R_{\bfn}(\Pi_{\bfn}^T B_{\bfn}\Pi_{\bfn})=\Pi_{\bfn}\Pi_{\bfn}^T B_{\bfn}\Pi_{\bfn}\Pi_{\bfn}^T=B_{\bfn},
	\end{equation*}
	\begin{equation*}
		E_{\bfn}(R_{\bfn}(A_{\bfn}))=E_{\bfn}(\Pi_{\bfn}A_{\bfn}\Pi_{\bfn}^T)=\Pi_{\bfn}^T \Pi_{\bfn}A_{\bfn}\Pi_{\bfn}^T\Pi_{\bfn} =D_{l_2\bfn}(\mathds{1}_{Q_1})A_{\bfn}D_{l_2\bfn}(\mathds{1}_{Q_1}).
	\end{equation*}
	
\end{proof}
\begin{cor}\label{product_hypercubes}
     Let $Q_1=Q_{{\bfy_1},l_1}\subset Q_2=Q_{{\bfy_2},l_2}$ be two hypercubes and let $\bfn$ be a multi-index. Then, if $B_\bfn,B'_\bfn$ are square matrices of size $N(l_1\bfn)$, it holds
     \begin{equation*}
         E_\bfn(B_\bfn B'_\bfn)=E_\bfn(B_\bfn)E_\bfn(B'_\bfn).
     \end{equation*}
     Conversely, if $A_\bfn,A'_\bfn\in\mathcal{C}^{N(l_2\cdot\bfn)\times N(l_2\cdot\bfn)}$ are in the image of the extension operator $E_\bfn$, then
     \begin{equation*}
         R_\bfn(A_\bfn A'_\bfn)=R_\bfn(A_\bfn)R_\bfn(A'_\bfn).
     \end{equation*}
     As a consequence, the extension operator $E_\bfn$ is a $*$-algebra morphism, whilst the restriction operator $R_\bfn$ is a $*$-algebra morphism, when restricted to the image of the extension operator. 
\end{cor}
\begin{proof}
    The proof is a direct consequence of Lemma \ref{lemma_1}. Indeed, first consider $B_\bfn,B'_\bfn\in\mathcal{C}^{N(l_1\cdot\bfn)\times N(l_1\cdot\bfn)}$. Then, we have
    \begin{equation*}
        E_\bfn(B_\bfn B'_\bfn)=\Pi_{\bfn}^T B_{\bfn}B'_{\bfn}\Pi_{\bfn}=\Pi_{\bfn}^T B_{\bfn}\Pi_{\bfn}\Pi_{\bfn}^TB'_{\bfn}\Pi_{\bfn}=E_\bfn(B_\bfn )E_\bfn(B'_\bfn).
    \end{equation*}
    Conversely, consider $A_\bfn,A'_\bfn\in\mathcal{C}^{N(l_2\cdot\bfn)\times N(l_2\cdot\bfn)}$ such that 
    \begin{equation*}
        A_\bfn=E_\bfn(B_\bfn),\qquad A'_\bfn=E_\bfn(B'_\bfn),
    \end{equation*}
    for some $B_\bfn,B'_\bfn\in\mathcal{C}^{N(l_1\cdot\bfn)\times N(l_1\cdot\bfn)}$. Then, by what we just proved, it follows
    \begin{align*}
        R_\bfn(A_\bfn A'_\bfn)&=R_\bfn(E_\bfn(B_\bfn)E_\bfn(B'_\bfn))\\
            &=R_\bfn(E_\bfn(B_\bfn B'_\bfn))\\
            &=B_\bfn B'_\bfn = R_\bfn(A_\bfn)R_\bfn(A_\bfn),
    \end{align*}
    as desired.
\end{proof}
\begin{rem}
Once that we have fixed $Q_1=Q_{{\bfy_1},l_1}\subset Q_2=Q_{{\bfy_2},l_2}$, we can consider the sequences $\{R_{\bfn}\}_{\bfn}$ and $\{E_{\bfn}\}_{\bfn}$ as operators acting on matrix-sequences of size $\{l_2^d\cdot N(\bfn)\}_{\bfn}$ and $\{l_1^d\cdot N(\bfn)\}_{\bfn}$, respectively. 

As in the classic reduced GLT approach (see \cite{Barb}), we would like to define an algebra of matrix-sequences equipped with a symbol on any bounded regular domain $\Omega\subset\mathbb{R}^d$, as the image of a restriction-type operator acting on GLT matrix-sequences over a hypercube containing $\Omega$. Of course, for a given bounded domain, we can find infinitely many hypercubes $Q_{\bfy,l}$, such that $\Omega\subset Q_{\bfy,l}$. Hence, we start by proving that $\mathcal{R}:=\{R_{\bf n}\}_\bfn$ and $\mathcal{E}:=\{E_{\bf n}\}_\bfn$ send GLT sequences to GLT sequences, so that our definition will be independent of the choice of $Q_{\bfy,l}$.
\end{rem}
\begin{lem}\label{acs_cts_operators}
    Let $Q_1(=Q_{{\bfy_1},l_1})\subset Q_2(=Q_{{\bfy_2},l_2})$ be two hypercubes. Denote by $\mathcal{G}_1=\mathcal{G}_{{\bfy_1},l_1}$ and $\mathcal{G}_2=\mathcal{G}_{{\bfy_2},l_2}$ the corresponding GLT algebras and by $\mathcal{M}_1$ and $\mathcal{M}_2$ the algebras of all matrix-sequences of size $\{l_1^d\cdot N(\bfn)\}_\bfn$ and $\{l_2^d\cdot N(\bfn)\}_\bfn$, respectively. Then, the operators $\mathcal{R}=\{R_{\bfn}\}_\bfn:\mathcal{M}_2\to\mathcal{M}_1$ and $\mathcal{E}=\{E_{\bfn}\}_\bfn:\mathcal{M}_1\to\mathcal{M}_2$ are continuous with respect to a.c.s. convergence.
\end{lem}
\begin{rem}\label{acs_same_structure}
    In the following sections, we make use of many other restriction and extension operators, built with the same structure as those between two hypercubes. Therefore, all those operators will be continuous with respect to a.c.s. convergence, with the same proof as the one of Lemma \ref{acs_cts_operators}, as long as the sequences of dimensions have the same rate of divergence (i.e. their ratio has a positive finite limit). 
\end{rem}
\begin{proof}[Proof of Lemma \ref{acs_cts_operators}]
    Fix $\bfn$ and fix $A_\bfn\in\mathbb{C}^{N(l_2\bfn)\times N(l_2\bfn)}$. Then, the matrix $R_\bfn(A_\bfn)$ is a minor of the larger matrix $A_\bfn$. For this reason,
    \begin{equation*}
        \mathrm{rank}(R_\bfn(A_\bfn))\leq\mathrm{rank}(A_\bfn),\qquad \|R_\bfn(A_\bfn)\|\leq\|A_\bfn\|.
    \end{equation*}
    Thus, $\mathcal{R}$ is linear and it preserves rank corrections and norm corrections, without increasing their magnitude. Moreover, the dimensions of the matrix-sequences in $\mathcal{M}_1$ and $\mathcal{M}_2$ have the same rate of divergence, namely
    \begin{equation*}
        \lim\limits_{\bfn\to\infty}\frac{N(l_1\bfn)}{N(l_2\bfn)}=\lim\limits_{\bfn\to\infty}\frac{l_1^d}{l_2^d}=\frac{l_1^d}{l_2^d}.
    \end{equation*}
    Therefore, $\mathcal{R}$ is continuous with respect to the a.c.s. convergence (Note that the pointwise information for $\bfn$ fixed was enough to prove that $\mathcal{R}$ preserved the norm correction of the whole sequence; on the other hand, for the rank correction one has to be careful that the estimate is in terms of the dimension of the matrices of the sequence, and this changes under the action of the operator).
    Similarly, fix $\bfn$ and fix $B_\bfn\in\mathbb{C}^{N(l_1\bfn)\times N(l_1\bfn)}$.  Then (see \cite[Lemma 4.3]{Barb} for further details), there exists a permutation matrix $P_{\bfn}\in\mathbb{C}^{N(l_2\bfn)\times N(l_2\bfn)}$, depending only on $\bfn$ and on the sets $Q_1,Q_2$, such that
    \begin{equation*}
        P_{\bfn}^TE_\bfn(B_\bfn) P_{\bfn}=\left( 
    \begin{array}{c|c} 
      B_\bfn & 0 \\ 
      \hline 
      0 & 0 
    \end{array} 
    \right).
    \end{equation*}
    As a consequence, each extension operator $E_\bfn$ is norm-preserving and rank-preserving, and the continuity of $\mathcal{E}$ follows as for $\mathcal{R}$.
\end{proof}
\begin{thm}\label{restriction_of_standard_GLT}
	Let $Q_1(=Q_{{\bfy_1},l_1})\subset Q_2(=Q_{{\bfy_2},l_2})$ be two hypercubes.
    If $\mathcal{R}=\{R_{\bfn}\}_\bfn:\mathcal{M}_2\to\mathcal{M}_1$ and $\mathcal{E}=\{E_{\bfn}\}_\bfn:\mathcal{M}_1\to\mathcal{M}_2$ are the restriction and extension operators, then we have
	\begin{equation}\label{GLT_op}
		\mathcal{R}(\mathcal{G}_2)\subset\mathcal{G}_1, \qquad \mathcal{E}(\mathcal{G}_1)\subset\mathcal{G}_2.
	\end{equation}
\end{thm}
\begin{proof}
	First, we prove the thesis for Toeplitz sequences $\{T_{l_i\bfn}(f)\}_\bfn$, where $f\in L^1([-\pi,\pi]^d)$, and for diagonal sequences $\{D_{l_i\bfn}(a_i)\}_\bfn$,where $a_i:Q_i\to\mathbb{C}, \,i=1,2,$ is continuous almost everywhere (i.e. measurable). 
    
    Fix an index $\bfn$ and consider $T_{l_2\bfn}(f)\in\mathbb{C}^{N(l_2\cdot \bfn)\times N(l_2\cdot\bfn)}$. Then, we have 
	\begin{align*}
		(R_{\bfn}(T_{l_2\bfn}(f)))_{\bfi,\bfj}&=\sum_{\bfk,\bfh=\bf1}^{l_2\cdot\bfn}(\Pi_{\bfn})_{\bfi,\bfk}(T_{l_2\bfn}(f))_{\bfk,\bfh}(\Pi_{\bfn}^T)_{\bfh,\bfj}\\
        &=\sum_{\bfk,\bfh=\bf1}^{l_2\cdot\bfn}\delta_{\psi_{\bfn}(\bfi),\bfk}\cdot f_{\bfk-\bfh}\cdot\delta_{\bfh,\psi_{\bfn}(\bfj)}\\
        &=f_{\psi_{\bfn}(\bfi)-\psi_{\bfn}(\bfj)}=f_{\bfi-\bfj}=(T_{l_1\bfn}(f))_{\bfi,\bfj},
	\end{align*}
	since 
	\begin{equation*}
		\psi_{\bfn}(\bfi)-\psi_{\bfn}(\bfj)=\bfi+\bfn\cdot(\bfy_1-\bfy_2)-(\bfj+\bfn\cdot(\bfy_1-\bfy_2))=\bfi-\bfj.
	\end{equation*}
	Hence, the restriction operator maps Toeplitz matrix-sequences to Toeplitz matrix-sequences with the same generating function:
	\begin{equation*}
		\mathcal{R}\left(\{T_{l_2\bfn}(f)\}_{\bfn}\right)=\{T_{l_1\bfn}(f)\}_{\bfn}.
	\end{equation*} 
	Exploiting this property, given $T_{l_1\bfn}(f)\in\mathbb{C}^{N(l_1\cdot \bfn)\times N(l_1\cdot\bfn)}$, we have
	\begin{equation*}
		E_{\bfn}(T_{l_1\bfn}(f))=E_{\bfn}(R_{\bfn}(T_{l_2\bfn}(f)))=D_{l_2\bfn}(\mathds{1}_{Q_1})T_{l_2\bfn}(f)D_{l_2\bfn}(\mathds{1}_{Q_1}),
	\end{equation*}
	so that
	\begin{equation*}
		\mathcal{E}\left(\{T_{l_1\bfn}(f)\}_{\bfn}\right)=\{D_{l_2\bfn}(\mathds{1}_{Q_1})T_{l_2\bfn}(f)D_{l_2\bfn}(\mathds{1}_{Q_1})\}_{\bfn}.
	\end{equation*}
	Thus, the extension of a Toeplitz matrix-sequence is a GLT matrix-sequence.
    
	Similarly, given an almost everywhere continuous function $a_2:Q_2\to\mathbb{C}$ and an index $\bfn$, we have 
	\begin{align*}
		(R_{\bfn}(D_{l_2\bfn}(a_2)))_{\bfi,\bfj}&=\sum_{\bfk,\bfh=\bf1}^{l_2\cdot\bfn}(\Pi_{\bfn})_{\bfi,\bfk}(D_{l_2\bfn}(a_2))_{\bfk,\bfh}(\Pi_{\bfn}^T)_{\bfh,\bfj}\\
        &=\sum_{\bfk,\bfh=\bf1}^{l_2\cdot\bfn}\delta_{\psi_{\bfn}(\bfi),\bfk}\cdot \delta_{\bfk,\bfh}a_2\bigg(\bfy_2+\frac{\bfk}{\bfn}\bigg)\cdot\delta_{\bfh,\psi_{\bfn}(\bfj)}\\
		&=\delta_{\psi_{\bfn}(\bfi),\psi_{\bfn}(\bfj)}a_2\left(\bfy_2+\frac{\psi_{\bfn}(\bfi)}{\bfn}\right)\\
        &=\delta_{\bfi,\bfj}a_2\left(\bfy_1+\frac{\bfi}{\bfn}\right)=(D_{l_1\bfn}(a_2|_{Q_1}))_{\bfi,\bfj},
	\end{align*}
	so that
	\begin{equation*}
		\mathcal{R}\left(\{D_{l_2\bfn}(a_2)\}_{\bfn}\right)=\{D_{l_1\bfn}(a_2|_{Q_1})\}_{\bfn}.
	\end{equation*}
	Conversely, given $a_1:Q_1\to\mathbb{C}$, we set 
	\begin{equation*}
			a_1^E= \begin{cases}
				a_1(s) & \text{if } s \in Q_1, \\
				0 & \text{if } s \in Q_2\setminus Q_1,
			\end{cases}
	\end{equation*}
	It follows that 
	\begin{equation*}
		E_{\bfn}(D_{l_1\bfn}(a_1))=E_{\bfn}(R_{\bfn}(D_{l_2\bfn}(a_1^E)))=D_{l_2\bfn}(\mathds{1}_{Q_1})D_{l_2\bfn}(a_1^E)D_{\bfn}(\mathds{1}_{Q_1})=D_{l_2\bfn}(a_1^E\cdot\mathds{1}_{Q_1})=D_{l_2\bfn}(a_1^E),
	\end{equation*}
	\begin{equation*}
		\mathcal{E}\left(\{D_{l_1\bfn}(a_1)\}_{\bfn}\right)=\{D_{l_2\bfn}(a_1^E)\}_{\bfn}.
	\end{equation*}
	As a consequence, the operators $\mathcal{R}$ and $\mathcal{E}$ map Toeplitz sequences and diagonal sequences to GLT sequences. By a classical characterization (see \cite[Theorem 8.6]{glt-book-1}), the GLT algebra is the a.c.s. closure of the sub-algebra generated by Toeplitz and diagonal sequences. Since both the operators respect a.c.s. convergence by Lemma \ref{acs_cts_operators}, we get
	\begin{equation*}
		\mathcal{R}(\mathcal{G}_2)\subset\mathcal{G}_1, \qquad \mathcal{E}(\mathcal{G}_1)\subset\mathcal{G}_2.
	\end{equation*}
	This ends the proof.
\end{proof}
 \begin{cor}[Compatibility of canonical GLT symbols]\label{cor_restr_ext_standard_GLT_symbol}
        Consider two hypercubes $Q_1\subset Q_2$, and consider $\{A_\bfn\}_\bfn\sim_{\mathrm{GLT}}^{Q_2}\, f$. 
        Then, 
        \begin{equation*}
            \mathcal{R}(\{A_\bfn\}_\bfn)\sim_{\mathrm{GLT}}^{Q_1}\,f|_{Q_1\times[-\pi,\pi]^d}.
        \end{equation*}
        Conversely, if $\{B_\bfn\}_\bfn\sim_{\mathrm{GLT}}^{Q_1}$g, then
        \begin{equation*}
            \mathcal{E}(\{B_\bfn\}_\bfn)\sim_{\mathrm{GLT}}^{Q_2}\,g^E,
        \end{equation*}
        where
        \begin{equation*}
            f^E(\bfx,\bftheta):=\begin{cases}
	g(\bfx,\bftheta) & \text{if } (\bfx,\bftheta)\in Q_1\times[-\pi,\pi]^d, \\
	0 & \text{otherwise.}
	\end{cases}
        \end{equation*}
    \end{cor}
    \begin{proof}
        Following the proof of Theorem \ref{restriction_of_standard_GLT}, the thesis holds for Toeplitz and diagonal matrix-sequences. In addition, taking canonical symbols is compatible with the algebra operations and with passing to $\mathrm{a.c.s.}$ limits. Thus, the thesis follows from the fact that both $\mathcal{G}_1$ and $\mathcal{G}_2$ are the $\mathrm{a.c.s.}$ closure of the sub-algebra generated by Toeplitz and diagonal matrix-sequences (see \cite[Theorem 8.6]{glt-book-1}).
    \end{proof}
\begin{rem}\label{rest_surjective}
    Note that, given two hypercubes $Q_1\subset Q_2$, by Lemma \ref{lemma_1}, we have
    \begin{equation*}
        \mathcal{R}\circ\mathcal{E}=\mathrm{Id}_{\mathcal{M}_1},
    \end{equation*}
    where $\mathrm{Id}_{\mathcal{M}_1}:\mathcal{M}_1\to\mathcal{M}_1$ is the identity operator. In particular, this implies that
    \begin{equation*}
        (\mathcal{R}\circ\mathcal{E})(\mathcal{G}_1)=\mathcal{G}_1.
    \end{equation*}
    As a consequence, we can update Theorem \ref{restriction_of_standard_GLT} obtaining that
    \begin{equation*}
        \mathcal{R}(\mathcal{G}_2)=\mathcal{G}_1,
    \end{equation*}
    i.e., the restriction operator is always surjective. 
    
    Similarly, one can easily prove that the image of the extension operator is the sub-algebra of $\mathcal{G}_2$ of sequences whose canonical symbol $f$ satisfies
    \begin{equation*}
        f(\bfx,\bftheta)=0,\qquad \text{for a.e. }\bfx\in Q_2\setminus Q_1.
    \end{equation*}
    \end{rem}

\section{Reduced GLT on regular bounded domains}\label{sec-reduced-glt}

 In this section, we consider a bounded domain with sufficiently regular boundary and define on it an algebra of reduced GLT, in the spirit of \cite{Barb}, with a slightly different convention for the grid points (and thus for the dimension of the matrix-sequences). This different viewpoint will not affect the spectral information, which will still be carried over to the reduced GLT algebras. On the other hand, our convention is the suitable one when dealing with unbounded domains, as shown in the following sections. 

 Given a divergent sequence of $d$-dimensional multi-indices $\{\bfn\}$, for $\bfx\in\mathbb{R}^d$, we define
 \begin{equation*}
     B_{\bfn}(\bfy):=\left\{\bfx\in\mathbb{R}^d \, \big|\,\max_{i=1,\ldots,d}n_i|x_i-y_i|<1\right\},
 \end{equation*}
and note that $\Ld(B_{\bfn}(\bfy))=2^d(N(\bfn))^{-1}$.

Now, consider an open bounded set $\Omega\subset\mathbb{R}^d$ and, for every $\bfn$, define
\begin{equation}\label{def_grid}    
\Theta_{\bfn,\Omega}:=\left\{\bfp\in\Theta_{\bfn}\,\big|\, B_{\bfn}(\bfp)\subset\Omega\right\}
\end{equation}
and set 
\begin{equation*}
    d_{\bfn}^{\Omega}:=\#\Theta_{\bfn,\Omega}.
\end{equation*}
The points of $\Theta_{\bfn,\Omega}$ are those of $\Theta_{\bfn}$ that are reasonably far from the boundary of the domain $\Omega$, compared to the shrinking factor depending on $\bfn$. In addition, $\Theta_{\bfn,\Omega}$ is a totally ordered set, equipped with the restriction of the lexicographic order $\preceq$.

Similarly as before, we want to cut rows and columns from GLT matrix-sequences associated with hypercubes containing $\Omega$, in such a way that we retrieve the portion of singular values associated with $\Omega$. 

Fix a hypercube $Q_{\bfy,l}\supset\Omega$ and consider the inclusion of ordered sets $\Theta_{\bfn,\Omega}\to\Theta_{\bfn,\bfy,l}$. With this notation, define the family of rectangular matrices $\Pi_{\bfn}(=\Pi_{\bfn,\bfy,l,\Omega})\in\mathbb{C}^{d_{\bfn}^{\Omega}\times N(l\cdot\bfn)}$ by
\begin{equation*}
    (\Pi_{\bfn})_{\bfi,\bfj}=\delta_{\bfi,\bfj},\qquad\bfi\in\Theta_{\bfn,\Omega},\quad \bfj\in\Theta_{\bfn,\bfy,l}.
\end{equation*}
Using the matrices above we can define the restriction and extension maps\\
\begin{equation*}
	R_{\bfn}(=R_{{\bfn, \bfy},l,\Omega}):\mathbb{C}^{N(l\cdot \bfn)\times N(l\cdot\bfn)}\to\mathbb{C}^{d_{\bfn}^{\Omega}\times d_{\bfn}^{\Omega}},
\end{equation*}
\begin{equation*}
	E_{\bfn}(=E_{\bfn,\Omega,\bfy,l}):\mathbb{C}^{d_{\bfn}^{\Omega}\times d_{\bfn}^{\Omega}}\to\mathbb{C}^{N(l\cdot \bfn)\times N(l\cdot\bfn)},
\end{equation*}
as
\begin{equation*}
	R_{\bfn}(A_{\bfn}):=\Pi_{\bfn}A_{\bfn}\Pi_{\bfn}^T,
\end{equation*}
\begin{equation*}
	E_{\bfn}(B_{\bfn}):=\Pi_{\bfn}^T B_{\bfn}\Pi_{\bfn}. 
\end{equation*}
Again, $E_{\bfn}$, $R_{\bfn}$ send normal matrices to normal matrices and Hermitian matrices to Hermitian matrices. 
Now, we restrict to open domains $\Omega$ such that $\mathcal{L}^d(\partial\Omega)=0$, where $\mathcal{L}^d$ denotes the $d$-dimensional Lebesgue measure. In principle, we could define the algebra of reduced GLT for any open bounded domain, but this assumption is necessary to retrieve spectral information, so we include it in the definition.

\begin{lem}\label{restr_ext_reduced_GLT}
    Let $\Omega\subset Q_{\bfy,l}$, with $\Omega$ open and $\mathcal{L}^d(\partial\Omega)=0$. For any multi-index $\bfn$, we have 
	\begin{equation}\label{la_stessa_dell_altro}
    \Pi_{\bfn}\Pi_{\bfn}^T = I_{\bfn}\in\mathbb{C}^{d_{\bfn}^{\Omega}\times d_{\bfn}^{\Omega}},
	\end{equation}
	\begin{equation}\label{piccola_rank_correction}
    \Pi_{\bfn}^T \Pi_{\bfn} = D_{l\bfn}(\mathds{1}_{\Omega})+S_{\bfn,\Omega}\in\mathbb{C}^{N(l\cdot \bfn)\times N(l\cdot\bfn)},
	\end{equation}
	where $I_{\bfn}$ is the identity matrix of size $d_\bfn^{\Omega}$, $D_{l\bfn}(\mathds{1}_{\Omega})$ is the $l\bfn$-th matrix of the diagonal matrix-sequence $\{D_{l\bfn}(\mathds{1}_{\Omega})\}_\bfn\in\mathcal{G}_{{\bfy},l}$, and $S_{\bfn,\Omega}$ is a rank correction satisfying
    \begin{equation*}
        \mathrm{rank}(S_{\bfn,\Omega})=o(N(\bfn)).
    \end{equation*}
	As a consequence, for square matrices $A_{\bfn}$ of size $N(l\bfn)$ and $B_{\bfn}$ of size $d_{\bfn}^{\Omega}$, we have:
	\begin{equation*}
		R_{\bfn}(E_{\bfn}(B_{\bfn}))=B_{\bfn},
	\end{equation*}
	\begin{equation*}
		E_{\bfn}(R_{\bfn}(A_{\bfn}))=D_{l\bfn}(\mathds{1}_{\Omega})A_{\bfn}D_{l\bfn}(\mathds{1}_{\Omega})+\tilde{S}_{\bfn,\Omega},
	\end{equation*}
    where $\tilde{S}_{\bfn,\Omega}$ is again a rank correction satisfying
    \begin{equation*}
        \mathrm{rank}(\tilde{S}_{\bfn,\Omega})=o(N(\bfn)).
    \end{equation*}
\end{lem}

\begin{proof}
    First, the proof of \eqref{la_stessa_dell_altro} is exactly the same as that of equation \eqref{lemma_1_eq_1} in Lemma \ref{lemma_1}.

    Now, we prove \eqref{piccola_rank_correction}. Reasoning in the same way as in the proof of Equation \eqref{lemma_1_eq_2} in Lemma \ref{lemma_1}, we find that $\Pi_{\bfn}^T \Pi_{\bfn}$ is a diagonal matrix, whose diagonal entry $(\bfi,\bfi)$ is non-zero if and only if $\bfi\in\Theta_{\bfn,\Omega}$.\\
    On the other hand, the standard diagonal matrix $D_\bfn(\mathds{1}_{\Omega})$ is a diagonal matrix whose entry $(\bfi,\bfi)$ is non-zero if and only if $\bfi\in\Omega$.
    
     Therefore, they differ for a diagonal rank correction $R_{\bfn,\Omega}$ such that
    \begin{equation*}
        \mathrm{rank}(R_{\bfn,\Omega})=\#\left\{\bfp\in\Theta_\bfn\,\vert\,\bfp\in\Omega,B_\bfn(\bfp)\not\subset\Omega\right\}
    \end{equation*}
    Note that a given point $\bfx\in\mathbb{R}^d$ is at most in $2^d$ different sets of the open cover $\{B_\bfn(p)\,\vert\,\bfp\in\Theta_\bfn\}$.

    In addition, if we define 
    \begin{align*}
        U_\bfn&:=\bigcup\limits_{\{\bfp\in\Theta_\bfn\,\vert\,\bfp\in\overline{\Omega},B_\bfn(\bfp)\not\subset\Omega\}} B_\bfn(\bfp),\\
        V_\bfn&:=\bigcup\limits_{\bfm\geq\bfn}U_\bfm,
    \end{align*}
    then the open sets $\{V_\bfn\}$ are shrinking as $\bfn\to\infty$ and 
    \begin{equation*}
        \partial\Omega=\bigcap_\bfn V_\bfn.
    \end{equation*}
    Therefore, by exterior regularity, it holds
    \begin{equation*}
        \limsup\limits_{\bfn\to\infty}\mathcal{L}^d(U_\bfn)\leq\lim\limits_{\bfn\to\infty}\mathcal{L}^d(V_\bfn)=\mathcal{L}^d(\partial\Omega)=0,
    \end{equation*}
    that is 
    \begin{equation*}
        \exists\lim\limits_{\bfn\to\infty}\mathcal{L}^d(U_\bfn)=0.
    \end{equation*}
    With all this information, we can estimate the rank of $R_{\bfn,\Omega}$ in the following way:
    \begin{align*}
        \mathrm{rank}(R_{\bfn,\Omega})&=\#\left\{\bfp\in\Theta_\bfn\,\vert\,\bfp\in\Omega,B_\bfn(\bfp)\not\subset\Omega\right\}\\
        &\leq\sum\limits_{\{\bfp\in\Theta_\bfn\,\vert\,\bfp\in\overline{\Omega},B_\bfn(\bfp)\not\subset\Omega\}}1\\
        &=\sum\limits_{\{\bfp\in\Theta_\bfn\,\vert\,\bfp\in\overline{\Omega},B_\bfn(\bfp)\not\subset\Omega\}}2^{-d}N(\bfn)\mathcal{L}^d(B_\bfn(\bfp))\\
        &\leq 2^d2^{-d}N(\bfn)\mathcal{L}^d\left(U_\bfn\right)\\
        &=N(\bfn)\mathcal{L}^d\left(U_\bfn\right),
    \end{align*}
    and the last term is $o(N(\bfn))$, since $\mathcal{L}^d(U_\bfn)\to0$, as $\bfn\to\infty$.\\ The rest of the proof is again the same as in Lemma \ref{lemma_1}.
\end{proof}

\begin{rem}
    As a consequence of Lemma \ref{restr_ext_reduced_GLT}, precisely as in Corollary \ref{product_hypercubes}, the extension operator $E_{\bfn,\Omega,\bfy,l}$ is an algebra morphism, whilst the restriction operator $R_{\bfn,\bfy,l,\Omega}$ is an algebra morphism, when restricted to the image of the corresponding extension operator.
\end{rem}

\begin{defn}[Reduced GLT algebra]\label{reduced_GLT}
    Consider an open bounded domain $\Omega$ such that $\mathcal{L}^d(\partial\Omega)=0$ and fix a hypercube $Q_{\bfy,l}$ such that $\Omega\subset Q_{\bfy,l}$. 
    
    Let $\mathcal{M}_{\bfy,l}$ and $\mathcal{M}_{\Omega}$ be the sets of all square matrix-sequences of size $\{l^d\cdot N(\bfn)\}_{\bfn}$ and $\{d_{\bfn}^{\Omega}\}_{\bfn}$, respectively. 
    
    If $\mathcal{R}_{\bfy,l,\Omega}=\{R_\bfn\}_\bfn:\mathcal{M}_{\bfy,l}\to\mathcal{M}_{\Omega}$ is the restriction operator, we define the reduced GLT algebra on $\Omega$, and we denote it by $\mathcal{G}_{\Omega}$, as
    \begin{equation*}
        \mathcal{G}_{\Omega}:=\mathcal{R}_{\bfy,l,\Omega}(\mathcal{G}_{\bfy,l}).
    \end{equation*}
\end{defn}

\begin{rem}\label{well-posedness}
    Note that, thanks to Theorem \ref{restriction_of_standard_GLT}, the definition of reduced GLT algebra over a regular bounded domain $\Omega$ does not depend on the choice of the hypercube $Q_{\bfy,l}$ containing $\Omega$. Indeed, thanks to the compatibility between the two definitions of restriction operators (between two hypercubes and from a hypercube to any bounded open domain), they factor through GLT algebras built upon intermediate hypercubes. 
    
    More precisely, if $\Omega\subset Q_{\bfy_1,l_1}\subset Q_{\bfy_2,l_2}$ and 
    \begin{align*}
        \mathcal{R}_{\bfy_2,l_2,\bfy_1,l_1}&:\mathcal{M}_{\bfy_2,l_2}\to\mathcal{M}_{\bfy_1,l_1},\\
        \mathcal{R}_{\bfy_2,l_2,\Omega}&:\mathcal{M}_{\bfy_2,l_1}\to\mathcal{M}_{\Omega},\\
        \mathcal{R}_{\bfy_1,l_1,\Omega}&:\mathcal{M}_{\bfy_1,l_1}\to\mathcal{M}_{\Omega}
    \end{align*}
    are the restriction operators as defined above, then 
    \begin{equation*}
         \mathcal{R}_{\bfy_1,l_1,\Omega}\circ\mathcal{R}_{\bfy_2,l_2,\bfy_1,l_1}=\mathcal{R}_{\bfy_2,l_2,\Omega}
    \end{equation*}
    and 
    \begin{equation*}
        \mathcal{R}_{\bfy_2,l_2,\bfy_1,l_1}(\mathcal{G}_{\bfy_2,l_2})=\mathcal{G}_{\bfy_1,l_1},
    \end{equation*}
    so that the image of the two operators is the same. This suggests that Definition \ref{reduced_GLT} selects matrix-sequences whose spectral properties are related to the domain $\Omega$, which is the case, as proven in Theorem \ref{spectra_reduced_GLT}. 
\end{rem}
\begin{rem}\label{algebra_morphisms_red}
    Thanks to Lemma \ref{restr_ext_reduced_GLT}, the map $\mathcal{E}_{\Omega,\bfy,l}\circ\mathcal{R}_{\bfy,l,\Omega}:\mathcal{M}_{\bfy,l}\to\mathcal{M}_{\bfy,l}$ restricts to an endomorphism of $\mathcal{G}_{\bfy,l}$. Moreover, since $\mathcal{R}_{\bfy,l,\Omega}\circ\mathcal{E}_{\Omega,\bfy,l}=\mathrm{Id}_{\mathcal{M}_{\Omega}}$, it follows that for every reduced GLT $\{B_\bfn\}_\bfn\in\mathcal{G}_{\Omega}$, there exists $\{A_\bfn\}_\bfn\in\mathcal{G}_{\bfy,l}$, such that
    \begin{equation*}
        \mathcal{R}_{\bfy,l,\Omega}(\{A_\bfn\}_\bfn)=\{B_\bfn\}_\bfn,\qquad \{A_\bfn\}_\bfn\in\mathrm{Im}(\mathcal{E}_{\Omega,\bfy,l}),
    \end{equation*}
    that is any reduced GLT sequence is the restriction of a GLT sequence in the image of the corresponding extension operator. 
    
    As a consequence, since the restriction operators are $*$-algebra morphisms (when restricted to the image of the extension operator), any reduced GLT algebra inherits the algebraic properties of a standard GLT algebra (see \ref{GLT3}-\ref{GLT4}).
\end{rem}

The next result is essentially a rephrasing, using our notation and conventions, of \cite[Lemma 4.4]{Barb}.

\begin{thm}\label{spectra_reduced_GLT}
    Consider an open bounded domain $\Omega$, with $\mathcal{L}^d(\partial\Omega)=0$ and a hypercube $Q_{\bfy,l}$ such that $\Omega\subset Q_{\bfy,l}$. 
    
    For any $\{B_{\bfn}\}_\bfn\in\mathcal{G}_{\Omega}$, consider a sequence $\{A_{\bfn}\}_\bfn$ in its preimage through the restriction map $\mathcal{R}:\mathcal{G}_{\bfy,l}\to\mathcal{G}_{\Omega}$, with canonical GLT symbol $f:Q_{\bfy,l}\times[-\pi,\pi]^d\to\mathbb{C}$.\\
    Then
    \begin{equation*}
        \{B_{\bfn}\}_{\bfn}\sim_{\sigma}\left(g,\Omega\times\left[-\pi,\pi\right]^d\right),
    \end{equation*}
    with $g=f|_{\Omega\times[-\pi,\pi]^d}$.
    
    If, additionally, $\{B_\bfn\}_\bfn$ is a sequence of Hermitian matrices, then it also holds 
    \begin{equation*}
        \{B_{\bfn}\}_{\bfn}\sim_{\lambda}\left(g,\Omega\times\left[-\pi,\pi\right]^d\right).
    \end{equation*}
\end{thm}

\begin{rem}
    Thanks to Remark \ref{well-posedness} and Corollary \ref{cor_restr_ext_standard_GLT_symbol}, the symbol $g$ associated with a reduced GLT sequence $\{B_\bfn\}_\bfn$ in Theorem \ref{spectra_reduced_GLT} is unique, up to equality almost everywhere. Therefore, we write that
    \begin{equation*}
        \{B_{\bfn}\}_{\bfn}\sim_{\mathrm{GLT}}^{\Omega}\,g,
    \end{equation*}
    and we call $g$ the canonical (reduced) GLT symbol of $\{B_\bfn\}_\bfn$.
\end{rem}

\begin{proof}[Proof of Theorem \ref{spectra_reduced_GLT}]
    First of all, consider the identity matrix-sequence $\{I_\bfn\}_\bfn\in\mathcal{M}_{\bfy,l}$, and consider the matrix-sequence $\{\tilde{I}_\bfn\}_\bfn:=(\mathcal{E}_{\Omega,\bfy,l}\circ\mathcal{R}_{\bfy,l,\Omega})(\{I_\bfn\}_\bfn)$. Thanks to Lemma \ref{restr_ext_reduced_GLT} and noticing that all the matrices are diagonal, for every $\bfn$, we have 
    \begin{align*}
        \tilde{I}_\bfn&=(D_{l\bfn}(\mathds{1}_{\Omega})+S_{
        \bfn,\Omega})I_\bfn(D_{l\bfn}(\mathds{1}_{\Omega})+S_{\bfn,\Omega})\\
        &=D_{l\bfn}(\mathds{1}_{\Omega})+2D_{l\bfn}(\mathds{1}_{\Omega})S_{\bfn,\Omega}+S_{\bfn,\Omega}^2\\
        &=D_{l\bfn}(\mathds{1}_{\Omega})+(2D_{l\bfn}(\mathds{1}_{\Omega})+S_{\bfn,\Omega})S_{\bfn,\Omega},
    \end{align*}
    with
    \begin{equation*}
        \mathrm{rank}\left((2D_{l\bfn}(\mathds{1}_{\Omega})+S_{\bfn,\Omega})S_{\bfn,\Omega}\right)\leq\mathrm{rank}(S_{\bfn,\Omega})=o(N(\bfn)).
    \end{equation*}
    Therefore, the matrix-sequence $\{\tilde{I}_\bfn\}_\bfn$ is a GLT over $Q_{\bfy,l}$, with canonical symbol $\mathds{1}_{\Omega\times[-\pi,\pi]^d}:Q_{\bfy,l}\times[-\pi,\pi]^d\to\mathbb{C}$. In particular, let $F\in C_c(\mathbb R)$ be a test function such that $F(0)=0$ and $F(1)=1$, so that
   \begin{align*}
     \lim_{\bfn\to\infty}\frac1{l^dN(\bfn)}\sum_{\bfi\in\Theta_{\bfn,\bfy,l}}F(\sigma_\bfi(\tilde{I}_\bfn))&=\frac{1}{\mathcal{L}^{2d}(Q_{\bfy,l}\times[-\pi,\pi]^d)}\int_{Q_{\bfy,l}\times[-\pi,\pi]^d}F(\mathds{1}_{\Omega\times[-\pi,\pi]^d}(\bfx,\bftheta)){\rm d} (\bfx,\bftheta)\\
     &=\frac1{l^d}\int_{\Omega}{\rm d} \bfx\cdot\frac1{(2\pi)^d}\int_{[-\pi,\pi]^d}\mathrm{d}\bftheta\\
     &=\frac{\Ld(\Omega)}{l^d}.
    \end{align*}
    On the other hand, $\{\tilde{I}_\bfn\}_\bfn$ is a diagonal matrix-sequence with entries $0$ or $1$, which selects the points of $\Theta_{\bfn,\bfy,l}$ that are also in $\Theta_{\bfn,\Omega}$. Thus, for every $\bfn$, it holds
    \begin{equation*}
        \sum_{\bfi\in\Theta_{\bfn,\bfy,l}}F(\sigma_\bfi(\tilde{I}_\bfn))=\mathrm{rank}(\tilde{I}_\bfn)=d_\bfn^{\Omega},
    \end{equation*}
    which, coupled with the previous info, gives
    \begin{equation*}
        \lim\limits_{\bfn\to\infty}\frac{d_\bfn^{\Omega}}{N(\bfn)}=\lim\limits_{\bfn\to\infty}\frac{\mathrm{rank}(\tilde{I}_\bfn)}{N(\bfn)}=\Ld(\Omega).
    \end{equation*}
    By Lemma \ref{acs_cts_operators} and Remark \ref{acs_same_structure}, since the last formula proved that the dimensions $\{N(l\bfn)\}_\bfn$ and $\{d_\bfn^\Omega\}_\bfn$ have the same rate of divergence, the restriction operators are continuous with respect to the a.c.s. convergence.
    Now, we consider a reduced GLT $\{B_{\bfn}\}_\bfn\in\mathcal{G}_{\Omega}$ and pick $\{A_\bfn\}_\bfn\in\mathcal{G}_{\bfy,l}$, such that 
    \begin{equation*}
        \mathcal{R}_{\bfy,l,\Omega}(\{A_\bfn\}_\bfn)=\{B_\bfn\}_\bfn,
    \end{equation*}
    \begin{equation*}
         \{A_{\bfn}\}_{\bfn}\sim_{\mathrm{GLT}}^{Q_{\bfy,l}}\,f(\bfx,\bftheta).
    \end{equation*}
    Define the matrix-sequence $\{\tilde{A}_\bfn\}_\bfn:=(\mathcal{E}_{\Omega,\bfy,l}\circ\mathcal{R}_{\bfy,l,\Omega})(\{A_\bfn\}_\bfn)$. Then, again by Lemma \ref{restr_ext_reduced_GLT}$, \{\tilde{A}_\bfn-D_{l\bfn}(\mathds{1}_\Omega)A_\bfn D_{l\bfn}(\mathds{1}_{\Omega})\}_\bfn\sim_{\sigma}0$, so that 
    \begin{equation*}
        \{\tilde{A}_\bfn\}_\bfn\sim_{\mathrm{GLT}}^{Q_{\bfy,l}}\,\mathds{1}_{\Omega}(\bfx)\cdot f(\bfx,\bftheta).
    \end{equation*}
    
    Moreover, (see \cite[Lemma 4.3]{Barb}) for every $\bfn$, there exists a permutation matrix $P_{\bfn,\bfy,l,\Omega}$, depending only on $\bfn$ and on the sets $Q_{\bfy,l}$ and $\Omega$, such that
    \begin{equation*}
        (P_{\bfn,\bfy,l,\Omega})^T\tilde{A}_\bfn (P_{\bfn,\bfy,l,\Omega})=\left( 
    \begin{array}{c|c} 
      B_\bfn & 0 \\ 
      \hline 
      0 & 0 
    \end{array} 
    \right).
    \end{equation*}
    As a consequence, for every $\bfn$, the singular values of $\tilde{A}_\bfn$ are the $d_\bfn^{\Omega}$ singular values of $B_\bfn$ with $l^dN(\bfn)-d_\bfn^{\Omega}$ additional zero singular values. 

    Using the info about the asymptotic distribution of $\{\tilde{A}_\bfn\}_\bfn$, we can finally compute the distribution of $\{B_\bfn\}_\bfn$. Considering any $F\in C_c(\mathbb{R})$, it follows that
    \begin{align*}
        &\lim\limits_{\bfn\to\infty}\frac1{d_\bfn^{\Omega}}\sum_{\bfi\in\Theta_{\bfn,\Omega}}F(\sigma_{\bfi}(B_\bfn))\\
        &=\lim\limits_{\bfn\to\infty}\frac{l^dN(\bfn)}{d_{\bfn}^{\Omega}}\left(\frac1{l^dN(\bfn)}\sum_{\bfi\in\Theta_{\bfn,\bfy,l}}F(\sigma_{\bfi}(\tilde{A}_\bfn))-\frac{l^dN(\bfn)-d_{\bfn}^{\Omega}}{l^dN(\bfn)}F(0)\right)\\
        &=\frac{l^d}{\Ld(\Omega)}\left(\frac{1}{\mathcal{L}^{2d}(Q_{\bfy,l}\times[-\pi,\pi]^d)}\int_{Q_{\bfy,l}\times[-\pi,\pi]^d}F(\mathds{1}_{\Omega}(\bfx)f(\bfx,\bftheta))\mathrm{d}(\bfx,\bftheta)-\frac{l^d-\Ld(\Omega)}{l^d}F(0)\right)\\
        &=\frac1{\mathcal{L}^{2d}(\Omega\times[-\pi,\pi]^d)}\int_{\Omega\times[-\pi,\pi]^d}F(\mathds{1}_{\Omega}(\bfx)f(\bfx,\bftheta))\mathrm{d}(\bfx,\bftheta),\\
    \end{align*}
    where we used that 
    \begin{align*}
        \frac1{\mathcal{L}^{2d}(\Omega\times[-\pi,\pi]^d)}\int_{(Q_{\bfy,l}\setminus\Omega)\times[-\pi,\pi]^d}F(\mathds{1}_{\Omega}(\bfx)f(\bfx,\bftheta))\mathrm{d}(\bfx,\bftheta)&=\frac{\mathcal{L}^{2d}(Q_{\bfy,l}\setminus\Omega)\times[-\pi,\pi]^d)}{\mathcal{L}^{2d}(\Omega\times[-\pi,\pi]^d)}F(0)\\
        &=\frac{l^d-\Ld(\Omega)}{\Ld(\Omega)}F(0).
    \end{align*}
    Since $\mathds{1}_{\Omega}\equiv 1$ almost everywhere in $\Omega$, the previous computation reads precisely as 
    \begin{equation*}
        \lim\limits_{\bfn\to\infty}\frac1{d_\bfn^{\Omega}}\sum_{\bfi\in\Theta_{\bfn,\Omega}}F(\sigma_{\bfi}(B_\bfn))=\frac1{\mathcal{L}^{2d}(\Omega\times[-\pi,\pi]^d)}\int_{\Omega\times[-\pi,\pi]^d}F(f(\bfx,\bftheta))\mathrm{d}(\bfx,\bftheta),
    \end{equation*}
    for every $F\in C_c(\mathbb{R})$, which is the thesis.\\
The proof follows exactly the same steps for the eigenvalue distribution.
\end{proof}

As the classical GLT algebras $\mathcal{G}_{\bfy,l}$ are isometrically equivalent to the spaces of measurable functions $\mathcal{L}^0(Q_{\bfy,l}\times[-\pi,\pi]^d)$, also the reduced GLT ones satisfy this pleasant property.
Indeed, for any open domain $\Omega$ such that $\Ld(\Omega)<+\infty$ and any $f:\Omega\times[-\pi,\pi]^d\to\mathbb{C}$, define
\begin{equation*}
    p_{m,\Omega}(f):=\inf\limits_{F\subset \Omega\times[-\pi,\pi]^d}\left\{\frac{\mathcal{L}^{2d}(F)}{(2\pi)^d\Ld(\Omega)}+\esssup_{F^c}|f|\right\},
\end{equation*}
where $F$ is measurable and $F^c$ is the complement of $F$ in $\Omega\times[-\pi,\pi]^d$. Then the function
\begin{equation*}
    d_{m,\Omega}(f,g):=p_{m,\Omega}(f-g)
\end{equation*}
is a complete pseudo-metric on $\mathcal{L}^0(\Omega\times[-\pi,\pi]^d)$, which induces the convergence in measure and for which two functions are at distance $0$ if and only if they are equal almost everywhere in $\Omega\times[-\pi,\pi]^d$. The next result shows that the symbol map from a reduced GLT algebra still satisfies the same properties as the one for classical GLT algebras. 

\begin{thm}\label{isometry_red_GLT}
 Let $\Omega$ be an open bounded domain such that $\Ld(\partial\Omega)=0$, and consider the map $\Phi_\Omega:(\sfrac{\mathcal{G}_\Omega}{\sim_d},d_{a.c.s.})\to(L^0(\Omega\times[-\pi,\pi]^d),d_{m,\Omega})$ that associates to any reduced GLT its canonical symbol. Then, $\Phi_\Omega$ is a surjective isometry.   
\end{thm}

\begin{proof}
    
    Consider any hypercube $Q_{\bfy,l}\supset\Omega$ and set $\mathcal{G}_\Omega^E:=\mathcal{E}_{\Omega,\bfy,l}(G_\Omega)\subset\mathcal{G}_{\bfy,l}$. Throughout the proof, we drop the indices for the extension (resp. restriction) operator $\mathcal{E}$ (resp. $\mathcal{R}$), since $\Omega$ and $Q_{\bfy,l}$ are fixed. Given any measurable function $f:\Omega\times[\pi,\pi]^d\to\mathbb{C}$, we can consider its extension by $0$ to $Q_{\bfy,l} \times [-\pi,\pi]^d$, $f^E$. By \ref{GLT0}, there exists $\{A_\bfn\}_\bfn\in\mathcal{G}_{\bfy,l}$, such that $\{A_\bfn\}_\bfn\sim_{\mathrm{GLT}}^{Q_{\bfy,l}}f^E$. Therefore, by Theorem \ref{spectra_reduced_GLT}, $\mathcal{R}(\{A_\bfn\}_\bfn)\sim_{\mathrm{GLT}}^{\Omega}f$. This proves that $\Phi_\Omega$ is surjective.

    Now, we prove that $\Phi_\Omega$ is an isometry, at the level of pseudo-metric spaces (which, in turn, implies that also the factorized map at the level of the quotient metric spaces is an isometry).
    By Remark \ref{algebra_morphisms_red}, the map $\mathcal{R}:\mathcal{G}_\Omega^E\to\mathcal{G}_\Omega$ is still surjective. On the other hand, for any $\{A_\bfn\}_\bfn\in\mathcal{G}_\Omega^E$, there exists $\{B_\bfn\}_\bfn\in\mathcal{G}_\Omega$, such that $\{A_\bfn\}_\bfn=\mathcal{E}(\{B_\bfn\}_\bfn)$. Therefore,
    \begin{equation*}
        (\mathcal{E}\circ\mathcal{R})(\{A_\bfn\}_\bfn)=\mathcal{E}(\mathcal{R}\circ\mathcal{E}(\{B_\bfn\}_\bfn))=\mathcal{E}(\{B_\bfn\}_\bfn)=\{A_\bfn\}_\bfn.
    \end{equation*}
    This proves that $(\mathcal{E}\circ\mathcal{R})|_{\mathcal{G}_\Omega^E}=\mathrm{Id}_{\mathcal{G}_\Omega^E}$.
    As a consequence, the restriction operator $\mathcal{R}$ is also injective on $\mathcal{G}_\Omega$, and thus bijective.
    
    In addition, thanks to the proof of Theorem \ref{spectra_reduced_GLT}, any element of $\mathcal{G}_\Omega^E$ has GLT symbol which is invariant under the multiplication by $\mathds{1}_\Omega(\bfx)$, and thus is equal to 0 almost everywhere outside $\Omega\times[-\pi,\pi]^d$.
    
    Moreover, for any $\bfn$ and any $B\in\mathbb{C}^{d_\bfn^\Omega\times d_\bfn^\Omega}$, the singular values of $E_\bfn(B)$ are the singular values of $B$, and $N(l\bfn)-d_\bfn^{\Omega}$ additional zero singular values. Therefore,
    \begin{align}\label{iso_red_eq_4}
        p(E_\bfn(B))&=\min\limits_{i=1,\ldots,N(l\bfn)+1}\left\{\frac{i-1}{N(l\bfn)}+\sigma_i(E_\bfn(B))\right\}\\
                \nonumber    &=\min\limits_{i=1,\ldots,d_\bfn^{\Omega}+1}\left\{\frac{i-1}{N(l\bfn)}+\sigma_i(B)\right\}\\
                \nonumber    &=\frac{d_\bfn^{\Omega}}{N(l\bfn)}\min\limits_{i=1,\ldots,d_\bfn^{\Omega}+1}\left\{\frac{i-1}{d_\bfn^{\Omega}}+\frac{N(l\bfn)}{d_\bfn^{\Omega}}\sigma_i(B)\right\}\\
                \nonumber    &=\frac{d_\bfn^{\Omega}}{N(l\bfn)}\min\limits_{i=1,\ldots,d_\bfn^{\Omega}+1}\left\{\frac{i-1}{d_\bfn^{\Omega}}+\sigma_i\left(\frac{N(l\bfn)}{d_\bfn^{\Omega}}B\right)\right\}\\
                \nonumber    &=\frac{d_\bfn^{\Omega}}{N(l\bfn)}p\left(\frac{N(l\bfn)}{d_\bfn^{\Omega}}B\right).
    \end{align}
    By the proof of Theorem \ref{spectra_reduced_GLT}, we know that
    \begin{equation*}
        \lim\limits_{\bfn\to\infty}\frac{d_\bfn^{\Omega}}{N(l\bfn)}=\frac{\Ld(\Omega)}{l^d}>0.
    \end{equation*}
    As a consequence, given any GLT $\{A_\bfn\}_\bfn\in\mathcal{G}_\Omega^E$ with symbol $f$ and such that $\{B_\bfn\}_\bfn=\mathcal{R}(\{A_\bfn\}_\bfn)$, we have that 
    \begin{equation*}
        \left\{\frac{d_\bfn^{\Omega}}{N(l\bfn)}A_\bfn\right\}_\bfn\sim^{Q_{\bfy,l}}_{\mathrm{GLT}}\frac{\Ld(\Omega)}{l^d}f.
    \end{equation*}
    
    For every $F,H$ measurable such that $F\sqcup H = Q_{\bfy,l}\times[-\pi,\pi]^d$, we have
    \begin{align*}
        \frac{\mathcal{L}^{2d}(F)}{(2\pi l)^d}+\esssup_{H}\left|\frac{\Ld(\Omega)}{l^d}f\right|&\geq \frac{\Ld(\Omega)}{l^d}\left( \frac{\mathcal{L}^{2d}(F\cap(\Omega\times[-\pi,\pi]^d))}{(2\pi )^d\Ld(\Omega)}+\esssup_{H\cap(\Omega\times[-\pi,\pi]^d)}\left|f\right|\right)\\
        &\geq \frac{\Ld(\Omega)}{l^d}d_{m,\Omega}(f|_{\Omega\times[-\pi,\pi]^d},0).
    \end{align*}
    By taking the $\inf$ on the left hand side, we obtain 
    \begin{equation}\label{iso_red_eq_1}
        d_{m,Q_{\bfy,l}}\left(\frac{\Ld(\Omega)}{l^d}f,0\right)\geq\frac{\Ld(\Omega)}{l^d}d_{m,\Omega}(f|_{\Omega\times[-\pi,\pi]^d},0).
    \end{equation}
    Conversely, since $f=0$ almost everywhere outside $\Omega\times[-\pi,\pi]^d$, for every pair of measurable sets $F,H$ such that $F\sqcup H=\Omega\times[-\pi,\pi]^d$, we have 
    \begin{align*}
        \frac{\mathcal{L}^{2d}(F)}{(2\pi)^d\Ld(\Omega)}+\esssup_{H}\left|f\right|&=\frac{l^d}{\Ld(\Omega)}\left(\frac{\mathcal{L}^{2d}(F)}{(2\pi l)^d}+\esssup_{H\cup ((Q_{\bfy,l}\setminus\Omega)\times[-\pi,\pi]^d)}\left|\frac{\Ld(\Omega)}{l^d}f\right|\right)\\
        &\geq \frac{l^d}{\Ld(\Omega)}d_{m,Q_{\bfy,l}}\left(\frac{\Ld(\Omega)}{l^d}f,0\right).
    \end{align*}
    Taking again the $\inf$ on the right hand side, we obtain the converse of the inequality in \eqref{iso_red_eq_1}, thus proving that 
    \begin{equation}\label{iso_red_eq_2}
        d_{m,Q_{\bfy,l}}\left(\frac{\Ld(\Omega)}{l^d}f,0\right)=\frac{\Ld(\Omega)}{l^d}d_{m,\Omega}(f|_{\Omega\times[-\pi,\pi]^d},0).
    \end{equation}
    By Theorem \ref{equiv_GLT_meas} applied to $\left\{\frac{d_\bfn^{\Omega}}{N(l\bfn)}A_\bfn\right\}_\bfn$ and Equation \eqref{iso_red_eq_2}, it follows that 
    \begin{align}\label{iso_red_eq_3}
        \limsup\limits_{\bfn\to\infty}p\left(\frac{d_\bfn^{\Omega}}{N(l\bfn)}A_\bfn\right)&=d_{a.c.s.}\left(\left\{\frac{d_\bfn^{\Omega}}{N(l\bfn)}A_\bfn\right\}_\bfn,\{0\}_\bfn\right)\\
        \nonumber&=d_{m,Q_{\bfy,l}}\left(\frac{\Ld(\Omega)}{l^d}f,0\right)\\
        \nonumber&=\frac{\Ld(\Omega)}{l^d}d_{m,\Omega}(f|_{\Omega\times[-\pi,\pi]^d},0).
    \end{align}
    Using now Equations \eqref{iso_red_eq_4} and \eqref{iso_red_eq_3}, we have
    \begin{align*}
    d_{a.c.s.}(\{B_\bfn\}_\bfn,\{0\}_\bfn)&=\limsup\limits_{\bfn\to\infty}p(B_\bfn)\\
          &=\limsup\limits_{\bfn\to\infty}p\left(\frac{N(l\bfn)}{d_\bfn^\Omega}\frac{d_\bfn^\Omega}{N(l\bfn)}B_\bfn\right)\\
          &=\limsup\limits_{\bfn\to\infty}\frac{N(l\bfn)}{d_\bfn^\Omega}p\left(E_\bfn\left(\frac{d_\bfn^\Omega}{N(l\bfn)}B_\bfn\right)\right)\\
          &=\limsup\limits_{\bfn\to\infty}\frac{N(l\bfn)}{d_\bfn^\Omega}p\left(\frac{d_\bfn^\Omega}{N(l\bfn)}A_\bfn\right)\\
          &=\frac{l^d}{\Ld(\Omega)}\frac{\Ld(\Omega)}{l^d}d_{m,\Omega}(f|_{\Omega\times[-\pi,\pi]^d},0)=d_{m,\Omega}(f|_{\Omega\times[-\pi,\pi]^d},0).
    \end{align*}
    Noticing that $\{B_\bfn\}_\bfn\sim^\Omega_{\mathrm{GLT}}f|_{\Omega\times[-\pi,\pi]^d}$ and that both $d_{a.c.s.}$ and $d_{m,\Omega}$ are invariant under translation, this concludes the proof. 
\end{proof}

\begin{rem}\label{rem_restr_btw_domains}
    If we consider an open domain $\Omega$, such that $d_\bfn^{\Omega}=\#\Theta_{\bfn,\Omega}<+\infty$, we can naturally define the algebra of all matrix-sequences of size $\{d_\bfn^{\Omega}\}_\bfn$.
    
    In addition, if two domains $\Omega_1\subset\Omega_2$ are such that $d_\bfn^{\Omega_i}=\#\Theta_{\bfn,\Omega_i}<+\infty$ for $i=1,2$ and for every $\bfn$, then we can define restriction and extension operators between $\mathcal{M}_{\Omega_1}$ and $\mathcal{M}_{\Omega_2}$. 
    
    Indeed, $\Omega_1\subset\Omega_2$ implies that $\Theta_{\bfn,\Omega_1}\subset\Theta_{\bfn,\Omega_2}$. Thus, consider the family of matrices $\Pi_{\bfn}(=\Pi_{\bfn,\Omega_2,\Omega_1})\in\mathbb{C}^{d_{\bfn}^{\Omega_1}\times d_{\bfn}^{\Omega_2}}$, defined by
\begin{equation*}
    (\Pi_{\bfn})_{\bfi,\bfj}=\delta_{\bfi,\bfj},\qquad\bfi\in\Theta_{\bfn,\Omega_1},\quad \bfj\in\Theta_{\bfn,\Omega_2}.
\end{equation*}
Similarly as above, the restriction operator $\mathcal{R}_{\Omega_2,\Omega_1}=\{R_{\bfn}\}_\bfn:\mathcal{M}_{\Omega_2}\to\mathcal{M}_{\Omega_1}$ and the extension operator $\mathcal{E}_{\Omega_1,\Omega_2}=\{E_{\bfn}\}_\bfn:\mathcal{M}_{\Omega_1}\to\mathcal{M}_{\Omega_2}$ are defined as 
\begin{equation*}
	R_{\bfn}(=R_{\bfn,\Omega_2,\Omega_1}):\mathbb{C}^{d_{\bfn}^{\Omega_2}\times d_{\bfn}^{\Omega_2}}\to\mathbb{C}^{d_{\bfn}^{\Omega_1}\times d_{\bfn}^{\Omega_1}},
\end{equation*}
\begin{equation*}
	E_{\bfn}(=E_{\bfn,\Omega_1,\Omega_2}):\mathbb{C}^{d_{\bfn}^{\Omega_1}\times d_{\bfn}^{\Omega_1}}\to\mathbb{C}^{d_{\bfn}^{\Omega_2}\times d_{\bfn}^{\Omega_2}},
\end{equation*}
where, explicitly,
\begin{equation*}
	R_{\bfn}(A_{\bfn}):=\Pi_{\bfn}A_{\bfn}\Pi_{\bfn}^T,
\end{equation*}
\begin{equation*}
	E_{\bfn}(B_{\bfn}):=\Pi_{\bfn}^T B_{\bfn}\Pi_{\bfn}. 
\end{equation*}

The operators $\mathcal{R}_{\Omega_2,\Omega_1}$ and $\mathcal{E}_{\Omega_1,\Omega_2}$ share the same structure of restriction and extension operators between a hypercube $Q$ and a domain $\Omega$. For this reason, all operators factor through ``intermediate domains'' as in Remark \ref{well-posedness}.

Indeed, this compatibility observation yields
\begin{align*}
    \mathcal{R}_{\Omega_2,\Omega_1}(\mathcal{G}_{\Omega_2})&\subset\mathcal{G}_{\Omega_1},\\
    \mathcal{E}_{\Omega_1,\Omega_2}(\mathcal{G}_{\Omega_1})&\subset\mathcal{G}_{\Omega_2}.
\end{align*}
Moreover, reasoning similarly as in Remark \ref{rest_surjective}, one can prove that the restriction operator is surjective, whilst the image of the extension operator is the sub-algebra of $\mathcal{G}_{\Omega_2}$ of sequences whose canonical symbol satisfies 
\begin{equation*}
        f(\bfx,\bftheta)=0,\qquad \text{for a.e. }\quad(\bfx,\bftheta)\in (\Omega_2\setminus \Omega_1)\times[-\pi,\pi]^d.
    \end{equation*}
For the same reason as in Corollary \ref{product_hypercubes}, the extension operator $\mathcal{E}_{\Omega_1,\Omega_2}$ is a $*$-algebra morphism, whilst the restriction operator $\mathcal{R}_{\Omega_2,\Omega_1}$ is a $*$-algebra morphism, when restricted to the image of $\mathcal{E}_{\Omega_1,\Omega_2}$.
Finally, by Lemma \ref{acs_cts_operators} and Remark \ref{acs_same_structure}, both the extension and the restriction operators are continuous with respect to the a.c.s. convergence, also in this case, as long as the sequences of dimensions have the same rate of divergence (note that for the class of domains of interest, this will always be the case, as a byproduct of item $(iii)$ of Lemma \ref{dimensions}).
\end{rem}

\section{Generalized a.c.s. convergence}\label{sec-gacs}


In this section we briefly recall the concept of generalized approximating class of sequences (g.a.c.s.) and its main application in the study of spectral distributions (see \cite{gacs} for the original definition and some applications).

\begin{defn}\label{def_gacs}
Let $\{A_{n}\}_n$ be a square matrix-sequence of size $d_n$, such that $ d_n\nearrow\infty$, and let $\left\{\{B_{n,t}\}_{n}\right\}_{t}$ be a sequence of square matrix-sequences of size $d_{n,t}$. Denoting by $ \oplus$ the standard direct sum of matrices, we say that $\left\{\{B_{n,t}\}_{n}\right\}_{t}$ is a generalized approximating class of sequences $(\mathrm{g.a.c.s.})$ for $\{A_{n}\}_n$
if, for every $t$, there exists $n_t$ such that, for $n>n_t$,
\begin{equation*}
	A_n =  U_{n,t} \left( B_{n,t}\oplus0_{n,t}\right)V_{n,t} + S_{n,t} + N_{n,t},
\end{equation*}
where $0_{n,t}$ is the null matrix of size $\left(d_n-d_{n,t}\right)$, $U_{n,t}$ and $V_{n,t}$ are two unitary matrices of size $d_n$ and $S_{n,t}, N_{n,t}$ are matrices of the same size as $A_n$, satisfying:
\begin{equation*}
	\textnormal{rank}\left(S_{n,t}\right) \leq c(t) d_n;
\end{equation*}
\begin{equation*}
	\left\|N_{n,t}\right\|\leq \omega(t);
\end{equation*}
\begin{equation*}
	d_n-d_{n,t}=:m_{n,t} \leq m(t)d_n;
\end{equation*}
and
\begin{equation*}
	\lim_{t \to \infty} c(t) =\lim_{t \to \infty} \omega(t) = \lim_{t \to \infty} m(t) =0.
\end{equation*}
If we have a sequence $\{A_{n}\}_n$	of Hermitian matrices, we also ask $\left\{\{B_{n,t}\}_{n}\right\}_{t}$ to be Hermitian and $V_{n,t}=U_{n,t}^*$ for all $n$ and $t$.

The sequences $\left\{\{U_{n,t}\}_n\right\}_t$ and $\left\{\{V_{n,t}\}_n\right\}_t$ are called structure matrix-sequences.
\end{defn}
\begin{rem}\label{rem_obs_ext_operator}
    In the next section, we will need the following useful observation.
    Consider two domains $\Omega_1\subset\Omega_2$, with $\Ld(\partial\Omega_i)=0$, for $i=1,2$. Fix any $\bfn$ and any $B_\bfn\in\mathcal{C}^{d_\bfn^{\Omega_1}\times d_\bfn^{\Omega_1}}$ and consider the extension operator $E_{\bfn}:\mathbb{C}^{d_{\bfn}^{\Omega_1}\times d_{\bfn}^{\Omega_1}}\to\mathbb{C}^{d_{\bfn}^{\Omega_2}\times d_{\bfn}^{\Omega_2}}$. 
    
    Define $\tilde{\Pi}_\bfn\in\mathbb{C}^{d_{\bfn}^{\Omega_2}\times d_{\bfn}^{\Omega_2}}$ as 
    \begin{equation*}
           \tilde{\Pi}_\bfn:= \left(
    \begin{array}{c} 
      \Pi_\bfn \\
      \hline
      C_\bfn
    \end{array} 
    \right), 
    \end{equation*}
    where $C_\bfn(=C_{\bfn,\Omega_2,\Omega_1})$ is any matrix of size $(d_\bfn^{\Omega_2}-d_\bfn^{\Omega_1})\times d_\bfn^{\Omega_2}$ that completes $\Pi_\bfn$ to a permutation matrix (for any row of $C_\bfn$, set all but one entry to $0$, and the last entry equal to $1$, in such a way that any column of $\tilde{\Pi}_\bfn$ has just one non-zero entry). Then, 
    \begin{equation*}
        E_\bfn(B_\bfn)=\Pi_\bfn^{T}B_\bfn\Pi_\bfn=\tilde{\Pi}_\bfn^{T}\tilde{B}_\bfn\tilde{\Pi}_\bfn,
    \end{equation*}
    with $\tilde{B}_\bfn:=B_\bfn\oplus0_\bfn$ and $0_\bfn(=0_{\bfn,\Omega_2,\Omega_1})$ is a square null matrix of size $d_\bfn^{\Omega_2}-d_\bfn^{\Omega_1}$.
    
    As a consequence, any square matrix $A_\bfn$ of size $d_\bfn^{\Omega}$ can be written as 
    \begin{equation}\label{restr_ext_identity}
        A_\bfn=E_\bfn(R_\bfn(A_\bfn))+S_\bfn=\tilde{\Pi}_\bfn^{T}\left(R_\bfn(A_\bfn)\oplus0_\bfn\right)\tilde{\Pi}_\bfn+S_\bfn,
    \end{equation}
    with $S_\bfn(=S_{\bfn,\Omega_2,\Omega_1})$ (which in the following will play the role of a rank correction) of rank bounded by $2(d_\bfn^{\Omega_2}-d_\bfn^{\Omega_1})$, since it has at most $d_\bfn^{\Omega_2}-d_\bfn^{\Omega_1}$ rows and columns with non-zero entries.
\end{rem}

\begin{thm}[\cite{gacs}]\label{gacs_sing_val}
	Let $\{A_n\}_n$ be a matrix-sequence of size $ d_n$ with $ d_n $ being an increasing sequence of integers. Let $\left\{\{B_{n,t}\}_{n}\right\}_{t}$ be a $\mathrm{g.a.c.s.}$ for $\{A_n\}_n$. If, $\forall t$, $\exists\left(f_t,\Omega_t\right)$ such that \\
	\begin{itemize}
		\item $\{B_{n,t}\}_{n} \sim_{\sigma} (f_t,\Omega_t)$;\\
		\item $\Omega_t \subset \Omega_{t+1}$, $\forall t$;\\
		\item $\Omega := \bigcup\limits_{t>0} \Omega_t$ of finite measure;\\
		\item $\exists f: \Omega\to\mathbb{C}
		$ measurable, such that $f^{E}_t\to f$ in measure, $t\to +\infty$, with
	\end{itemize}
	\begin{equation*}
	f_t^E= \begin{cases}
	f_t(s) & \text{if } s \in \Omega_t, \\
	0 & \text{if } s \in \Omega\setminus \Omega_t,
	\end{cases}
	\end{equation*}
	then $\{A_n\}_n \sim_{\sigma} (f,\Omega)$.

    If, additionally, $\{A_n\}_n$ is a sequence of Hermitian matrices, then it also holds $\{A_n\}_n \sim_{\lambda} (f,\Omega)$.
\end{thm}


\section{The algebra of unbounded GLT}\label{sec-UGLT}

In this section we introduce the main object of our derivations. Using all the tools from the previous sections, we construct an algebra of matrix-sequences canonically associated to functions over an unbounded domain with finite measure, which we call unbounded GLT (uGLT), we study the main algebraic and spectral properties of the uGLT sequences, ultimately proving the equivalence between uGLT sequences and measurable functions over $\Omega$. For the convenience of the reader, we divide this section in three subsections, each devoted to develop some crucial aspects of our work.

\subsection{Geometric ingredients and definition of uGLT}

In this subsection, we construct an algebra of matrix-sequences canonically associated with functions over an unbounded domain with finite measure.

First of all, consider an open domain $\Omega$, such that 
\begin{equation*}
    \mathcal{L}^d(\Omega)<+\infty,\qquad\mathcal{L}^d(\partial\Omega)=0.
\end{equation*}
Note that the definition of the grid associated with a bounded domain in \eqref{def_grid} gives a finite number of points for every $\bfn$, even if $\Omega$ is unbounded.
Indeed, considering $\Omega$ unbounded, set as before
\begin{equation*}    
\Theta_{\bfn,\Omega}:=\left\{\bfp\in\Theta_{\bfn}\,\big|\, B_{\bfn}(\bfp)\subset\Omega\right\}
\end{equation*}
and 
\begin{equation*}
    d_{\bfn}^{\Omega}:=\#\Theta_{\bfn,\Omega}.
\end{equation*}
Then 
\begin{align*}
    d_{\bfn}^{\Omega}&=\#\left\{\bfp\in\Theta_\bfn\,\vert\, B_\bfn(\bfp)\subset\Omega\right\}\\
        &=\sum\limits_{\{\bfp\in\Theta_\bfn\,\vert\, B_\bfn(\bfp)\subset\Omega\}}1\\
        &=\sum\limits_{\{\bfp\in\Theta_\bfn\,\vert\, B_\bfn(\bfp)\subset\Omega\}}2^{-d}N(\bfn)\mathcal{L}^d(B_{\bfn}(\bfp))\\
        &\leq2^d2^{-d}N(n)\mathcal{L}^d(\Omega),
\end{align*}
using that 
\begin{equation}\label{eq_3}
    \bigcup\limits_{\{\bfp\in\Theta_\bfn\,\vert\, B_\bfn(\bfp)\not\subset\Omega\}}B_\bfn(\bfp)\subset\Omega,
\end{equation}
and that any point in $\Omega$ appears at most in $2^d$ sets of the union in the left hand side of \eqref{eq_3}.

In particular, $d_{\bfn}^{\Omega}$ is finite for every $\bfn$ and bounded from above by $\mathcal{L}^d(\Omega)N(\bfn)$.

Applying Remark \ref{rem_restr_btw_domains}, it is possible to construct restriction and extension operators also for the algebra $\mathcal{M}_\Omega$, where $\Ld(\Omega)<\infty$ and $\Ld(\partial\Omega)=0$

\begin{defn}[Regular exhaustion]\label{reg_exhaustion}
    Let $\Omega$ be an open domain, such that $\Ld(\Omega)<\infty$, $\Ld(\partial\Omega)=0$. A regular exhaustion of $\Omega$ is a family of open domains $\{\Omega_t\}_t$ (where $t$ is a discrete or continuous index), such that
    \begin{enumerate}
    \item\label{reg_item_1} $\forall t$, $\Omega_t$ is bounded and such that $\Ld(\partial\Omega_t)=0$;
    \item\label{reg_item_2} $\forall t<t'$, $\Omega_t\subset\Omega_{t'}$;
    \item\label{reg_item_3} $\bigcup_t\Omega_t=\Omega$.
    \end{enumerate}

\end{defn}
\begin{exmp}\label{exmp:reg_exhaustion}
    Any domain $\Omega$ as in Definition \ref{reg_exhaustion} above possesses a regular exhaustion. 
    Indeed, for every $t>0$, define
    \begin{equation*}
        \Omega_t:=\left\{\bfx\in\Omega\,\vert\,\|\bfx\|_\infty<t\right\},
    \end{equation*}
then, by construction, $\Omega_t$ is open and satisfies items \eqref{reg_item_2}-\eqref{reg_item_3} of Definition \ref{reg_exhaustion}. Moreover, for every $t>0$, it holds 
\begin{equation*}
    \partial\Omega\setminus\partial\Omega_t\subset E_t:=\left\{\bfx\in\Omega\,\vert\,\|\bfx\|_\infty=t\right\}.
\end{equation*}
The sets $\{E_t\}_{t>0}$ all have zero $d$-dimensional Lebesgue measure. 
Therefore, $\{\Omega_{t}\}_t$ is a regular exhaustion for $\Omega$. 
\end{exmp}

The next lemma establishes some useful preliminary properties of the dimension of the grid $d_\bfn^{\Omega}$, associated with a domain $\Omega$.
\begin{lem}\label{dimensions}
    Let $\Omega$ be an open domain, such that $\Ld(\Omega)<\infty$, $\Ld(\partial\Omega)=0$. Consider a regular exhaustion $\{\Omega_t\}_t$ of $\Omega$.\\
    \begin{itemize}
        \item[(i)]\label{dim_item_1} There exist $\bfn_0$ and $c>0$, such that $d_\bfn^{\Omega}\geq cN(\bfn)$, for every $\bfn\geq\bfn_0$;\\
        \item[(ii)]\label{dim_item_2} there exists $c'>0$ such that, for every $t$ and every $\bfn\geq\bfn_t$, it holds 
        \begin{equation*}
            \frac{d_\bfn^{\Omega}-d_\bfn^{\Omega_t}}{d_\bfn^{\Omega}}\leq c'\Ld(\Omega\setminus\Omega_t);
        \end{equation*}
        \item[(iii)]\label{dim_item_3} as $\bfn\to\infty$, it holds
        \begin{equation*}
            \frac{d_\bfn^{\Omega}}{N(\bfn)}\to\Ld(\Omega).
        \end{equation*}
    \end{itemize}
\end{lem}
\begin{proof}
    Since $\Omega$ is open, there exist $\bfx\in\Omega$ and $r>0$ such that
    \begin{equation*}
        B_{\infty}(\bfx,r):=\left\{\bfy\in\mathbb{R}^d\,\big\vert\,|y_i-x_i|<r,\,\, i=1,\dots,d\right\}\subset\Omega.
    \end{equation*}
    As $\bfn\to\infty$, there exists $\bfn_0$ such that 
    \begin{equation*}
        \frac{1}{\min_{i}n_i}<\frac r2,
    \end{equation*}
    for every $\bfn\geq\bfn_0$. In particular, $\bfn\geq\bfn_0$ implies that
    \begin{equation*}
        \Theta_\bfn\cap B_{\infty}(\bfx,r/2)\subset\Theta_{\bfn,\Omega}.
    \end{equation*}
    We want to estimate $\#\Theta_\bfn\cap B_{\infty}(\bfx,r/2)$ from below to get a lower bound for $d_\bfn^{\Omega}$. 
    
    Consider an interval $(a,b)\subset\mathbb{R}$ and $n\in\mathbb{N}$. Then, the number of points of the form $i/n$, with $i\in\mathbb{Z}$, contained in $(a,b)$, is at least
    \begin{equation*}
        \bigg\lfloor\frac{n}{b-a}\bigg\rfloor-1
    \end{equation*}
    and, for $n$ large enough, we have
    \begin{equation*}
        \bigg\lfloor\frac{n}{b-a}\bigg\rfloor-1\geq\frac{n}{2(b-a)}.
    \end{equation*}
    Now, possibly increasing $\bfn_0$, and applying the previous reasoning to all $d$ directions, we have 
    \begin{equation*}
        d_\bfn^{\Omega}\geq\#\Theta_\bfn\cap B_{\infty}(\bfx,r/2)\geq\prod\limits_{i=1}^d\frac{n_i}{2r}=\frac1{(2r)^d}N(\bfn),
    \end{equation*}
    for every $\bfn\geq\bfn_0$, which proves item (i). 
    
    For proving item (ii), note that
    \begin{equation*}
        d_\bfn^{\Omega}-d_\bfn^{\Omega_t}=\#I_{\bfn,t},
    \end{equation*}
    where 
    \begin{equation*}
        I_{\bfn,t}:=\{\bfp\in\Theta_\bfn\,\vert\,B_\bfn(\bfp)\subset\Omega,\,B_\bfn(\bfp)\not\subset\Omega_t\}.
    \end{equation*}
    If, for some $t$, $\Omega_t=\Omega$, then item (ii) is trivially true for that $t$ for any $c'$ and every $\bfn$.
    Otherwise, we split the set of indices in two disjoint subsets and estimate the cardinality of each of them separately. Indeed, for every $(\bfn,t)$, define
    \begin{align*}
        J_{\bfn,t}&:=\{\bfp\in\Theta_\bfn\,\vert\,B_\bfn(\bfp)\subset\Omega,\,B_\bfn(\bfp)\not\subset\Omega\setminus\Omega_t,\,B_\bfn(\bfp)\not\subset\Omega_t\}\\
        &=\{\bfp\in\Theta_\bfn\,\vert\,B_\bfn(\bfp)\subset\Omega,\,B_\bfn(\bfp)\cap\partial\Omega_t\neq\emptyset\}.
    \end{align*}
    Moreover, note that $\Omega\setminus\overline{\Omega}_t$ is open and such that
    \begin{align*}
        \Theta_{\bfn,\,\Omega\setminus\overline{\Omega}_t}&=\{\bfp\in\Theta_\bfn\,\vert\,B_\bfn(\bfp)\subset\Omega\setminus\overline{\Omega}_t\}\\
        &=\{\bfp\in\Theta_\bfn\,\vert\,B_\bfn(\bfp)\subset\Omega\setminus\Omega_t\}.
    \end{align*}
    By definition, it holds that 
    \begin{equation*}
        I_{\bfn,t}=\Theta_{\bfn,\,\Omega\setminus\overline{\Omega}_t}\sqcup J_{\bfn,t}.
    \end{equation*}
    We estimate the cardinality of $J_{\bfn,t}$ with a covering argument, already used before. 
    Set 
    \begin{align*}
        U_{\bfn,t}&:=\bigcup\limits_{\bfp\in J_{\bfn,t}}B_\bfn(\bfp),\\
        V_{\bfn,t}&:=\left\{\bfx\in\Omega\,\bigg\vert\,\mathrm{dist}(\bfx,\partial\Omega_t)<\frac{1}{\min_i n_i}\right\},
    \end{align*}
    and, by construction, $U_{\bfn,t}\subset V_{\bfn,t}$ for every $(\bfn,t)$. Moreover, the sets $\{V_{\bfn,t}\}_\bfn$ are decreasing and 
    \begin{equation*}
        \bigcap_{\bfn}V_{\bfn,t}=\partial\Omega_t\cap\Omega.
    \end{equation*}
    Therefore, since $\overline{\Omega}_t$ is bounded, by outer regularity,
    \begin{equation*}
    0\leq\limsup_\bfn\Ld(U_{\bfn,t})\leq\lim_\bfn\Ld(V_{\bfn,t})=\Ld(\partial\Omega_t\cap\Omega)=0.
    \end{equation*}
    In particular, for a fixed $t$, there exists $\bfn_t$ such that, for every $\bfn\geq\bfn_t$,
    \begin{equation*}
        \Ld(U_{\bfn,t})\leq\Ld(\Omega\setminus\Omega_t).
    \end{equation*}
    Putting all this info together and using item (i), we get
    \begin{align*}
        \#J_{\bfn,t}&=\sum\limits_{\bfp\in J_{\bfn,t}}1\\
            &=\sum\limits_{\bfp\in J_{\bfn,t}}2^{-d}N(\bfn)\Ld(B_\bfn(\bfp))\\
            &\leq2^d2^{-d}N(\bfn)\Ld(U_{\bfn,t})\\
            &\leq\frac1cd_\bfn^{\Omega}\Ld(\Omega\setminus\Omega_t),
    \end{align*}
    for every $\bfn\geq\bfn_t$. On the other hand,
    \begin{align*}
        \#\Theta_{\bfn,\,\Omega\setminus\overline{\Omega}_t}=d_\bfn^{\Omega\setminus\overline{\Omega}_t}&\leq N(\bfn)\Ld(\Omega\setminus\overline{\Omega}_t)\\
        &\leq\frac1c d_\bfn^{\Omega}\Ld(\Omega\setminus\Omega_t),
    \end{align*}
    where we used again that $\Ld(\partial\Omega_t)=0$. 
    
    We just proved that, for every $\bfn\geq\bfn_t$, it holds
    \begin{equation*}
        d_\bfn^{\Omega}-d_\bfn^{\Omega_t}=\#I_{\bfn,t}\leq\frac2c d_\bfn^{\Omega}\Ld(\Omega\setminus\Omega_t),
    \end{equation*}
    which is precisely item (ii), with $c'=2/c$ (Note that $c'$ does not depend on $t$, since all the dependence is encoded in the choice of $\bfn_t$). 

    Recalling the proof of Theorem \ref{spectra_reduced_GLT}, we already know that item (iii) holds if $\Omega$ is bounded. Fix any $\varepsilon>0$, and pick $t$ such that $\Ld(\Omega\setminus\Omega_t)\leq\varepsilon.$
    Applying item (ii) and item (iii) to $\Omega_t$, there exists $\bfn_t$ such that, for every $\bfn\geq\bfn_t$, we have
    \begin{align*}
          \frac{d_\bfn^{\Omega}-d_\bfn^{\Omega_t}}{N(\bfn)}&\leq\Ld(\Omega)\frac{d_\bfn^{\Omega}-d_\bfn^{\Omega_t}}{d_\bfn^{\Omega}}\\
          &\leq c'\Ld(\Omega)\Ld(\Omega\setminus\Omega_t)\\
          &\leq c'\Ld(\Omega)\varepsilon,
    \end{align*}
    and also
    \begin{equation*}
        \bigg\vert\frac{d_\bfn^{\Omega_t}}{N(\bfn)}-\Ld(\Omega_t)\bigg\vert\leq\varepsilon.
    \end{equation*}
    As a consequence,
    \begin{equation*}
        \bigg\vert\frac{d_\bfn^{\Omega}}{N(\bfn)}-\Ld(\Omega)\bigg\vert\leq\frac{d_\bfn^{\Omega}-d_\bfn^{\Omega_t}}{N(\bfn)}+\bigg\vert\frac{d_\bfn^{\Omega_t}}{N(\bfn)}-\Ld(\Omega_t)\bigg\vert+\Ld(\Omega\setminus\Omega_t)\leq \big(c'\Ld(\Omega)+2\big)\varepsilon,
    \end{equation*}
    which proves item (iii), by the arbitrariness of $\varepsilon>0$.
\end{proof}

\begin{defn}[Unbounded GLT sequences]\label{def_UGLT}
    Let $\Omega$ be an open domain, such that $\Ld(\Omega)<\infty$, $\Ld(\partial\Omega)=0$.
    Consider the algebra $\mathcal{M}_{\Omega}$ of all matrix-sequences of size $\{d_\bfn^{\Omega}\}_\bfn$. An element $\{A_\bfn\}_\bfn\in\mathcal{M}_\Omega$ is called an unbounded GLT over $\Omega$ if there exists a regular exhaustion $\{\Omega_t\}_t$ of $\Omega$ and a sequence of matrix-sequences $\left\{\{B_{\bfn,t}\}_\bfn\right\}_t$ of size $\left\{\{d_\bfn^{\Omega_t}\}_\bfn\right\}_t$, such that
    \begin{itemize}
        \item[(i)] $\forall t$, $\{B_{\bfn,t}\}_\bfn\in\mathcal{G}_{\Omega_t}$, with GLT symbol $f_t:\Omega_t\times[-\pi,\pi]^d\to\mathbb{C}$;
        \item[(ii)] $\left\{\{B_{\bfn,t}\}_\bfn\right\}_t$ is a $\mathrm{g.a.c.s.}$ for $\{A_\bfn\}_\bfn$, with structure matrix-sequences given by $\left\{\left\{\tilde{\Pi}^{T}_{\bfn,t}\right\}_\bfn\right\}_t$ and $\left\{\left\{\tilde{\Pi}_{\bfn,t}\right\}_\bfn\right\}_t$, as in Remark \ref{rem_obs_ext_operator};
        \item[(iii)] $\forall t<t'$, $f_t=f_{t'}$, almost everywhere in  $\Omega_t\times[-\pi,\pi]^d$.
    \end{itemize}
    The sequence $\{\{B_{\bfn,t}\}_\bfn\}_t$ is regarded as an approximating sequence for $\{A_\bfn\}_\bfn$, associated with the exhaustion $\{\Omega_t\}_t$.
\end{defn}
\begin{rem}[Identity sequence]
    Given an open domain $\Omega$, such that $\Ld(\Omega)<+\infty$, and $\Ld(\partial\Omega)=0$, the identity matrix-sequence $\{I_\bfn\}_\bfn\in\mathcal{M}_\Omega$ is an unbounded GLT.  
\end{rem}
\begin{proof}
    For every open set $D$, throughout the proof, denote by $\{I_\bfn^{D}\}_\bfn$ the identity matrix-sequence in $\mathcal{M}_D$. 
    Consider a regular exhaustion $\{\Omega_t\}_t$ of $\Omega$. By Remark \ref{rem_obs_ext_operator}, for every $t$, we can write
    \begin{equation*}
        \{I_\bfn^{\Omega}\}_\bfn=\mathcal{E}_{\Omega_t,\Omega}(\{I_\bfn^{\Omega_t}\}_\bfn)+\{S_{\bfn,t}\}_\bfn,
    \end{equation*}
    with $\mathrm{rank}(S_{\bfn,t})\leq 2(d_\bfn^{\Omega}-d_\bfn^{\Omega_t})$. 
    
    Note that, for every $t$, we have
    \begin{equation*}
        \{I_\bfn^{\Omega_t}\}_\bfn\sim_{\mathrm{GLT}}^{\Omega_t}1,
    \end{equation*}
    so that, by item (ii) of Lemma \ref{dimensions}, we conclude that $\{\{I_\bfn^{\Omega_t}\}_\bfn\}_t$ is a g.a.c.s. for $\{I_\bfn\}_\bfn$, which concludes the proof.
\end{proof}

\subsection{Uniqueness of the canonical uGLT symbol} This section is devoted to the unveiling of Definition \ref{def_UGLT}. More precisely, we prove that any approximation of $\{A_\bfn\}_\bfn$ in the sense of Definition \ref{def_UGLT} produces the same symbol, which we call the canonical symbol function or the unbounded GLT symbol of the sequence $\{A_\bfn\}_\bfn$.

\begin{rem}\label{rem_unboundedGLT_distribution}
    In view of Definition \ref{def_UGLT}, if $\{A_\bfn\}_\bfn$ is an unbounded GLT over a domain $\Omega$, then a direct application of Theorem \ref{gacs_sing_val} shows that $\{A_n\}_n \sim_{\sigma} (f,\Omega\times[-\pi,\pi]^d)$, where $f=f_t$ almost everywhere on $\Omega_t$, for every $t$ (and similarly for the eigenvalue distribution, if $\{A_\bfn\}_\bfn$ is Hermitian).

    The symbol $f$ seems to depend highly on the choice of the exhaustion and of the approximating sequence. Indeed, this is not the case, as we show in the following pair of results. In particular, we also see that the selected symbol $f$ carries a lot of the asymptotic structure of the sequence $\{A_\bfn\}_\bfn$ and behaves well under algebra operations.
\end{rem}
\begin{thm}\label{RGLT_equiv_UGLT_bounded}
    Let $\Omega$ be an open \textbf{bounded} domain, such that $\Ld(\partial\Omega)=0$. Then, a matrix-sequence $\{A_\bfn\}_\bfn\in\mathcal{M}_\Omega$ is a reduced GLT if and only if it is an unbounded GLT over $\Omega$.

    Moreover, if $\{A_\bfn\}_\bfn\in\mathcal{G}_\Omega$, then any choice of regular exhaustion of $\Omega$ and of $\mathrm{g.a.c.s.}$ for $\{A_\bfn\}_\bfn$ as in Definition \ref{def_UGLT} computes the reduced GLT symbol of $\{A_\bfn\}_\bfn$.

\end{thm}
\begin{proof}
    \noindent{\textbf{Step 1.}} We prove that any reduced GLT $\{A_\bfn\}_\bfn\in\mathcal{G}_\Omega$ is an unbounded GLT over $\Omega$. As a byproduct, we obtain that any choice of regular exhaustion and of definining sequence for $\{A_\bfn\}_\bfn$ computes the reduced GLT symbol.

    Indeed, consider $\{A_\bfn\}_\bfn\in\mathcal{G}_\Omega$, with associated symbol $f:\Omega\times[-\pi,\pi]^d\to\mathbb{C}$, and consider a regular exhaustion $\{\Omega_t\}_t$ of $\Omega$. 

    For every $t$, define
    \begin{equation*}
        \{B_{\bfn,t}\}_\bfn:=\mathcal{R}_{\Omega,\Omega_t}(\{A_\bfn\}_\bfn).
    \end{equation*}
    Then, $\{B_{\bfn,t}\}_\bfn\in\mathcal{G}_{\Omega_t}$, for every $t$, with canonical symbol $f_t:=f|_{\Omega_t\times[-\pi,\pi]^d}$. In particular, for every $t<s$, it trivially holds $f_t=f_s$ almost everywhere in $\Omega_t\times[-\pi,\pi]^d$. Therefore, if we prove that $\left\{\{B_{\bfn,t}\}_\bfn\right\}_t$ is a $\mathrm{g.a.c.s.}$ for $\{A_\bfn\}_\bfn$, then $\{A_\bfn\}_\bfn$ is an unbounded GLT over $\Omega$.
    
    Now, using Equation \eqref{restr_ext_identity} in Remark \ref{rem_obs_ext_operator}, for every $(\bfn,t)$, there exist square matrices $\tilde{\Pi}_{\bfn,t},S_{\bfn,t}$ of size $d_\bfn^{\Omega}$, such that 
    \begin{equation}\label{gacs_red_unb}
        A_\bfn=\tilde{\Pi}_{\bfn,t}^{T}\left(B_{\bfn,t}\oplus0_{\bfn,t}\right)\tilde{\Pi}_{\bfn,t}+S_{\bfn,t},
    \end{equation}
    where, for every $(\bfn,t)$, $\tilde{\Pi}_{\bfn,t}$ is a permutation, $\mathrm{rank}(S_{\bfn,t})\leq2(d_\bfn^{\Omega}-d_\bfn^{\Omega_t})$ and $0_{\bfn,t}$ is the zero matrix of size $d_\bfn^{\Omega}-d_\bfn^{\Omega_t}$.
    Therefore, we can define 
    \begin{equation*}
        m(t):=c'\Ld(\Omega\setminus\Omega_t),\qquad \omega(t):=0,\qquad c(t):=2c'\Ld(\Omega\setminus\Omega_t),
    \end{equation*}
    and item (ii) of Lemma \ref{dimensions} implies that the decomposition in Equation \eqref{gacs_red_unb} provides that $\left\{\{B_{\bfn,t}\}_\bfn\right\}_t$ is a $\mathrm{g.a.c.s.}$ for $\{A_\bfn\}_\bfn$, with the correct structure matrix-sequences, as desired. 

    \noindent{\textbf{Step 2.}} We prove that any unbounded GLT $\{A_\bfn\}_\bfn\in\mathcal{M}_\Omega$ is a reduced GLT over $\Omega$. 
    
    As a matter of fact, consider an unbounded GLT $\{A_\bfn\}_\bfn$ over $\Omega$, and consider its defining exhaustion $\{\Omega_t\}_t$, with corresponding approximating sequence $\left\{\{B_{\bfn,t}\}_\bfn\right\}_t$. Since $\Omega$ is bounded, choose a hypercube $Q_{\bfy,l}$ such that $\Omega\subset Q_{\bfy,l}$, and note that $\{A_\bfn\}_\bfn\in\mathcal{M}_{\Omega}$, so we can apply to it the extension operator $\mathcal{E}_{\Omega,\bfy,l}:\mathcal{M}_\Omega\to\mathcal{M}_{\bfy,l}$.
    Write 
    \begin{equation}\label{identity_1}
        \{A_\bfn\}_\bfn=\{\tilde{\Pi}_{\bfn,t}^{T}\left(B_{\bfn,t}\oplus0_{\bfn,t}\right)\tilde{\Pi}_{\bfn,t}\}_\bfn+\{N_{\bfn,t}\}_\bfn+\{S_{\bfn,t}\}_\bfn,
    \end{equation}
    where $\{N_{\bfn,t}\}_\bfn$ are the norm corrections and $\{S_{\bfn,t}\}_\bfn$ are the rank corrections.
    
    Since the structure matrix-sequences of the $\mathrm{g.a.c.s.}$ are compatible with restrictions and extensions, applying $\mathcal{E}_{\Omega,\bfy,l}$ to Equation \eqref{identity_1}, we get
    \begin{align*}
        \mathcal{E}_{\Omega,\bfy,l}(\{A_\bfn\}_\bfn)&= \mathcal{E}_{\Omega,\bfy,l}(\{\tilde{\Pi}_{\bfn,t}^{T}\left(B_{\bfn,t}\oplus0_{\bfn,t}\right)\tilde{\Pi}_{\bfn,t}\}_\bfn)+\mathcal{E}_{\Omega,\bfy,l}(\{N_{\bfn,t}\}_\bfn)+\mathcal{E}_{\Omega,\bfy,l}(\{S_{\bfn,t}\}_\bfn)\\
        &=\mathcal{E}_{\Omega_t,\bfy,l}(\{B_{\bfn,t}\}_\bfn)+\mathcal{E}_{\Omega,\bfy,l}(\{N_{\bfn,t}\}_\bfn)+\mathcal{E}_{\Omega,\bfy,l}(\{S_{\bfn,t}\}_\bfn).
    \end{align*}
    In addition, the extension operator $\mathcal{E}_{\Omega,\bfy,l}$ does not increase the norm nor the rank of any element of the matrix-sequence and the sequences of dimensions $\{N(l\bfn)\}_\bfn$ and $\{d_\bfn^\Omega\}_\bfn$ have the same rate of divergence (recall the proof of Theorem \ref{spectra_reduced_GLT}). As a consequence, this proves that $\left\{\{\mathcal{E}_{\Omega_t,\bfy,l}(\{B_{\bfn,t}\}_\bfn)\}_\bfn\right\}_t$ is an $\mathrm{a.c.s.}$ for $\{\mathcal{E}_{\Omega,\bfy,l}(\{A_\bfn\}_\bfn)\}_\bfn$.

    As $\left\{\{B_{\bfn,t}\}_\bfn\right\}_t$ are reduced GLT, the approximating sequences $\{\mathcal{E}_{\Omega_t,\bfy,l}(\{B_{\bfn,t}\}_\bfn)\}_t$ are regular GLT over $Q_{\bfy,l}$. Since the GLT algebra $\mathcal{G}_{\bfy,l}$ is closed under the $\mathrm{a.c.s.}$ convergence (see Theorem \ref{equiv_GLT_meas}), it follows that $\mathcal{E}_{\Omega,\bfy,l}(\{A_\bfn\}_\bfn)\in\mathcal{G}_{\bfy,l}$ and 
    \begin{equation*}
        \{A_\bfn\}_\bfn=\left(\mathcal{R}_{\bfy,l,\Omega}\circ\mathcal{E}_{\Omega,\bfy,l}\right)(\{A_\bfn\}_\bfn)
    \end{equation*}
    is a reduced GLT.
\end{proof}
    If we have two open domains $\Omega_1$ and $\Omega_2$, such that $\Ld(\Omega_i)<\infty$ and $\Ld(\partial\Omega_i)=0$, for $i=1,2$, there are two natural maps between $\mathcal{M}_{\Omega_1}\to\mathcal{M}_{\Omega_2}$. More precisely, we can extend the matrix-sequences to $\Omega_1\cup\Omega_2$ and then restricting them to $\Omega_2$, or we can restrict them to $\Omega_1\cap\Omega_2$, and then extend them to $\Omega_2$. Indeed, as we see in the following lemma, these two natural operations commute.
\begin{lem}\label{lem:auxiliary_lem1}
    Consider two open domains $\Omega_1$ and $\Omega_2$, such that $\Ld(\Omega_i)<\infty$, and $\Ld(\partial\Omega_i)=0$, for $i=1,2$. Suppose that $\Omega_1\cap\Omega_2\neq\emptyset$. Then the following diagram commutes:
    \[ \begin{tikzcd}
\mathcal{M}_{\Omega_1} \arrow{r}{\mathcal{E}_{\Omega_1,\Omega_1\cup\Omega_2}}\arrow[swap]{d}{\mathcal{R}_{\Omega_1,\Omega_1\cap\Omega_2}} & \mathcal{M}_{\Omega_1\cup\Omega_2} \arrow{d}{\mathcal{R}_{\Omega_1\cup\Omega_2,\Omega_2}} \\%
\mathcal{M}_{\Omega_1\cap\Omega_2} \arrow{r}{\mathcal{E}_{\Omega_1\cap\Omega_2,\Omega_2}}& \mathcal{M}_{\Omega_2}
\end{tikzcd}
\]
\end{lem}
\begin{proof}
    For every $\bfn$, consider the matrix-sequences defining the restriction and extension operators $\{\Pi_\bfn\}_\bfn$. Throughout this proof, we write explicitly the dependence on the domains. 
    
    With careful computation, the Lemma is just a consequence of the definition of the operators involved. Indeed, fix $\bfn$ and, for every $A_\bfn\in\mathbb{C}^{d_\bfn^{\Omega_1}\times d_\bfn^{\Omega_1}}$, we have 
    \begin{align*}
       R_{\bfn,\Omega_1\cup\Omega_2,\Omega_2}\circ E_{\bfn,\Omega_1,\Omega_1\cup\Omega_2}(A_\bfn)&=\Pi_{\bfn,\Omega_1\cup\Omega_2,\Omega_2}\Pi_{\bfn,\Omega_1\cup\Omega_2,\Omega_1}^T A_{\bfn}\Pi_{\bfn,\Omega_1\cup\Omega_2,\Omega_1}\Pi_{\bfn,\Omega_1\cup\Omega_2,\Omega_2}^T\\
       &=\left(\Pi_{\bfn,\Omega_1\cup\Omega_2,\Omega_1}\Pi_{\bfn,\Omega_1\cup\Omega_2,\Omega_2}^T\right)^TA_\bfn\left(\Pi_{\bfn,\Omega_1\cup\Omega_2,\Omega_1}\Pi_{\bfn,\Omega_1\cup\Omega_2,\Omega_2}^T\right),\\
        E_{\bfn,\Omega_1\cap\Omega_2,\Omega_2}\circ R_{\bfn,\Omega_1,\Omega_1\cap\Omega_2}(A_\bfn)&=\Pi_{\bfn,\Omega_2,\Omega_1\cap\Omega_2}\Pi_{\bfn,\Omega_1,\Omega_1\cap\Omega_2}^T A_{\bfn}\Pi_{\bfn,\Omega_1,\Omega_1\cap\Omega_2}\Pi_{\bfn,\Omega_2,\Omega_1\cap\Omega_2}^T\\
        &=\left(\Pi_{\bfn,\Omega_1,\Omega_1\cap\Omega_2}\Pi_{\bfn,\Omega_2,\Omega_1\cap\Omega_2}^T\right)^TA_\bfn\left(\Pi_{\bfn,\Omega_1,\Omega_1\cap\Omega_2}\Pi_{\bfn,\Omega_2,\Omega_1\cap\Omega_2}^T\right).
    \end{align*}
    Therefore, the thesis reduces to prove that 
    \begin{equation}\label{quella_equation}
    \Pi_{\bfn,\Omega_1\cup\Omega_2,\Omega_1}\Pi_{\bfn,\Omega_1\cup\Omega_2,\Omega_2}^T=\Pi_{\bfn,\Omega_1,\Omega_1\cap\Omega_2}\Pi_{\bfn,\Omega_2,\Omega_1\cap\Omega_2}^T.
    \end{equation}
    For $\bfi\in\Theta_{\bfn,\Omega_1}$ and $\bfj\in\Theta_{\bfn,\Omega_2}$, the right hand side of the above equation is
    \begin{align}\label{large_formula}
        \left(\Pi_{\bfn,\Omega_1\cup\Omega_2,\Omega_1}\Pi_{\bfn,\Omega_1\cup\Omega_2,\Omega_2}^T\right)_{\bfi,\bfj}&=\sum_{\bfk\in\Theta_{\Omega_1\cup\Omega_2}}\delta_{\bfi,\bfk}\delta_{\bfk,\bfj}.
    \end{align}
    Then, note that $\delta_{\bfi,\bfk}\delta_{\bfk,\bfj}$ is non-zero if and only if $\bfi=\bfk=\bfj$. In particular, if $\bfi$ or $\bfj$ are not in $\Theta_{\Omega_1\cap\Omega_2}$, all the terms in the sum are zero. On the other hand, if $\bfi,\bfj\in\Theta_{\Omega_1\cap\Omega_2}$, it is not restrictive to sum only on the indices $\bfk\in\Theta_{\Omega_1\cap\Omega_2}$. Carrying on from \eqref{large_formula}, we get
    \begin{align*}
        \sum_{\bfk\in\Theta_{\Omega_1\cup\Omega_2}}\delta_{\bfi,\bfk}\delta_{\bfk,\bfj}&=\sum_{\bfk\in\Theta_{\Omega_1\cap\Omega_2}}\delta_{\bfi,\bfk}\delta_{\bfk,\bfj}\\
        &=\left(\Pi_{\bfn,\Omega_1,\Omega_1\cap\Omega_2}\Pi_{\bfn,\Omega_2,\Omega_1\cap\Omega_2}^T\right)_{\bfi,\bfj},
    \end{align*}
    which is precisely \eqref{quella_equation}. Since this holds for all $\bfi,\bfj$, and for all $\bfn$, the thesis follows.
\end{proof}

The next result shows that the property of being an unbounded GLT is preserved by the restriction operators, and is compatible with taking the corresponding restriction of the canonical symbols.
\begin{thm}[Permanence of spectral information]\label{perm_spectr_info}
    Let $\Omega$ be an open domain, such that $\Ld(\Omega)<\infty$, $\Ld(\partial\Omega)=0$. Consider an unbounded GLT $\{A_\bfn\}_\bfn\in\mathcal{M}_\Omega$. Then, for every open bounded domain $\Omega_0\subset\Omega$ such that $\Ld(\partial\Omega_0)=0$, it holds $\mathcal{R}_{\Omega,\Omega_0}(\{A_\bfn\}_\bfn)\in\mathcal{G}_{\Omega_0}$.
    
    Additionally, let $f:\Omega\times[-\pi,\pi]^d\to\mathbb{C}$ be any symbol for $\{A_\bfn\}_\bfn$, computed via an approximating sequence defining $\{A_\bfn\}_\bfn$. Then
    \begin{equation*}
        \mathcal{R}_{\Omega,\Omega_0}(\{A_\bfn\}_\bfn)\sim_{GLT}^{\Omega_0}f|_{\Omega_0\times[-\pi,\pi]^d}.
    \end{equation*}
\end{thm}
\begin{proof}
    Fix an unbounded GLT $\{A_\bfn\}_\bfn$ over $\Omega$, and consider a regular exhaustion $\{\Omega_t\}_t$, with corresponding approximating sequence $\left\{\{B_{\bfn,t}\}_\bfn\right\}_t$. In view of Remark \ref{rem_obs_ext_operator}, this explicitly means that there exist rank corrections $\{\{S_{\bfn,t}\}_\bfn\}_t$ and norm corrections $\{\{N_{\bfn,t}\}_\bfn\}_t$, such that, for every $t$, it holds 
    \begin{align*}
        \{A_\bfn\}_\bfn&=\mathcal{E}_{\Omega_t,\Omega}(\{B_{\bfn,t}\}_\bfn)+\{S_{\bfn,t}\}_\bfn+\{N_{\bfn,t}\}_\bfn,
    \end{align*}
    with appropriate estimates on the size difference, on the rank correction and on the norm correction in terms of fixed functions $m(t),c(t),\omega(t)$, respectively, going to $0$ as $t\to\infty$.
    
    Now, we apply the restriction operator $\mathcal{R}_{\Omega,\Omega_0}$ to previous equation, and use Lemma \ref{lem:auxiliary_lem1} together with the property that restrictions and extensions factor through intermediate domains, to get
    \begin{align*}
        \mathcal{R}_{\Omega,\Omega_0}(\{A_\bfn\}_\bfn)&=\mathcal{R}_{\Omega,\Omega_0}\circ\mathcal{E}_{\Omega_t,\Omega}(\{B_{\bfn,t}\}_\bfn)+\mathcal{R}_{\Omega,\Omega_0}(\{S_{\bfn,t}\}_\bfn)+\mathcal{R}_{\Omega,\Omega_0}(\{N_{\bfn,t}\}_\bfn)\\
         &=\mathcal{R}_{\Omega_t\cup\Omega_0,\Omega_0}\circ\mathcal{E}_{\Omega_t,\Omega_t\cup\Omega_0}(\{B_{\bfn,t}\}_\bfn)+\mathcal{R}_{\Omega,\Omega_0}(\{S_{\bfn,t}\}_\bfn)+\mathcal{R}_{\Omega,\Omega_0}(\{N_{\bfn,t}\}_\bfn)\\
        &=\mathcal{E}_{\Omega_t\cap\Omega_0,\Omega_0}(\mathcal{R}_{\Omega_t,\Omega_t\cap\Omega_0}(\{B_{\bfn,t}\}_\bfn))+\mathcal{R}_{\Omega,\Omega_0}(\{S_{\bfn,t}\}_\bfn)+\mathcal{R}_{\Omega,\Omega_0}(\{N_{\bfn,t}\}_\bfn).
    \end{align*}
    Fix $t$ and, since $\left\{\{B_{\bfn,t}\}_\bfn\right\}_t$ is a $\mathrm{g.a.c.s.}$ for $\{A_\bfn\}_\bfn$, for $\bfn>\bfn_t$, we have
    \begin{equation*}
        \mathrm{rank}(\mathcal{R}_{\Omega,\Omega_0}(S_{\bfn,t}))\leq\mathrm{rank}(S_{\bfn,t})\leq c(t)d_\bfn^{\Omega},
    \end{equation*}
    where $c(t)$ is a fixed function going to $0$, as $t\to\infty$. 
    
    In order to prove a g.a.c.s. result for the restricted sequence, we need to estimate the above ranks in terms of $d_\bfn^{\Omega_0}$ and not of $d_\bfn^{\Omega}$.
    
    Indeed, thanks to the preliminary discussion at the beginning of this section and to item (i) of Lemma \ref{dimensions}, for every admissible domain $\tilde{\Omega}$ and every $t$, there exists $C=C(\tilde{\Omega})>1$ and $\bfn_t$, such that
    \begin{equation*}
        C(\tilde{\Omega})^{-1}N(\bfn)<d_\bfn^{\tilde{\Omega}}<C(\tilde{\Omega})N(\bfn),
    \end{equation*}
    for $\bfn>\bfn_t$.
    
    In particular, applying this to the domains $\Omega$ and $\Omega_0$, it yields
    \begin{equation*}
        d_\bfn^{\Omega}<C(\Omega)N(\bfn)<C(\Omega)C(\Omega_0)d_\bfn^{\Omega_0},
    \end{equation*}
    from which it follows
    \begin{equation*}
        \mathrm{rank}(\mathcal{R}_{\Omega,\Omega_0}(D_{\bfn,t}))\leq \tilde{c}(t)d_\bfn^{\Omega_0},
    \end{equation*}
    for every $\bfn>\bfn_t$ and with $\tilde{c}(t):=C(\Omega)C(\Omega_0)c(t)$.
    
    Similarly, for the norm corrections, for every $\bfn>\bfn_t$, it holds
    \begin{equation*}
        \|\mathcal{R}_{\Omega,\Omega_0}(N_{\bfn,t})\|\leq \|N_{\bfn,t})\|\leq \omega(t)d_\bfn^{\Omega}\leq\tilde{\omega}(t)d_\bfn^{\Omega_0},
    \end{equation*}
    with $\tilde{\omega}(t):=C(\Omega)C(\Omega_0)\omega(t)$ and $\omega(t)$ fixed and going to $0$ as $t\to\infty$. 
    
    If $\{B_{\bfn,t}\}_\bfn\sim^{\Omega_t}_{\mathrm{GLT}}f_t$, then $\mathcal{R}_{\Omega_t,\Omega_t\cap\Omega_0}(\{B_{\bfn,t}\}_\bfn)\sim^{\Omega_t\cap\Omega_0}_{\mathrm{GLT}}f_t|_{(\Omega_t\cap\Omega_0)\times[-\pi,\pi]^d}$. In addition, the size difference $d_\bfn^{\Omega_0}-d_\bfn^{\Omega_t\cap\Omega_0}$ can be estimated, for $\bfn>\bfn_t$, using item (ii) of Lemma \ref{dimensions}. More precisely, for every $\bfn>\bfn_t$, it holds
    \begin{equation*}
        d_\bfn^{\Omega_0}-d_\bfn^{\Omega_t\cap\Omega_0}\leq c'(\Omega_0)d_\bfn^{\Omega_0}\Ld(\Omega_0\setminus(\Omega_t\cap\Omega_0)).
    \end{equation*}
    Since $\{\Omega_t\cap\Omega_0\}_t$ is a regular exhaustion of $\Omega_0$ and in view of Remark \ref{rem_obs_ext_operator}, it follows that $\{\mathcal{R}_{\Omega_t,\Omega_t\cap\Omega_0}(\{B_{\bfn,t}\}_\bfn)\}_t$ is a g.a.c.s. for $\mathcal{R}_{\Omega,\Omega_0}(\{A_\bfn\}_\bfn)$, with canonical symbols $\{f_t|_{(\Omega_t\cap\Omega_0)\times[-\pi,\pi]^d}\}_t$. 

    As a consequence, $\mathcal{R}_{\Omega,\Omega_0}(\{A_\bfn\}_\bfn)$ is an unbounded GLT over $\Omega_0$, with symbol $f|_{\Omega_0\times[-\pi,\pi]^d}$, where $f$ is the symbol of $\{A_\bfn\}_\bfn$ computed using the approximating sequence $\{\{B_{\bfn,t}\}_\bfn\}_t$. Thanks to Theorem \ref{RGLT_equiv_UGLT_bounded}, this proves that
    \begin{equation*}
        \mathcal{R}_{\Omega,\Omega_0}(\{A_\bfn\}_\bfn)\sim_{\mathrm{GLT}}^{\Omega_0}f|_{\Omega_0\times[-\pi,\pi]^d},
    \end{equation*}
    which ends the proof.
\end{proof}
\begin{cor}[Independence from the approximation]\label{indep_on_approx}
    Let $\Omega$ be an open domain, such that $\Ld(\Omega)<\infty$, $\Ld(\partial\Omega)=0$. Consider an unbounded GLT $\{A_\bfn\}_\bfn\in\mathcal{M}_\Omega$. Then, any regular exhaustion and any approximating sequence compute the same symbol function $f:\Omega\times[-\pi,\pi]^d\to\mathbb{C}$, up to equality almost everywhere.

    In addition, for every regular exhaustion $\{\Omega_t\}_t$ of $\Omega$, there exists an approximating sequence $\left\{\{B_{\bfn,t}\}_\bfn\right\}_t$ for $\{A_\bfn\}_\bfn$ such that $\{B_{\bfn,t}\}_\bfn\in\mathcal{G}_{\Omega_t}$.
\end{cor}
    Thanks to the previous corollary, for any unbounded GLT there is a canonical choice of symbol function. Therefore, with a slight abuse of notation, if $\{A_\bfn\}_\bfn$ is an unbounded GLT over a domain $\Omega$ with canonical symbol $f:\Omega\times[-\pi,\pi]^d\to\mathbb{C}$, we write 
     \begin{equation*}
        \{A_{\bfn}\}_{\bfn}\sim_{\mathrm{GLT}}^{\Omega}\,f.
    \end{equation*}
    Note that, thanks to Theorem \ref{RGLT_equiv_UGLT_bounded}, this notation is compatible with the one of reduced GLT, if $\Omega$ is a bounded domain.
\begin{proof}[Proof of Corollary \ref{indep_on_approx}]
    We exploit the uniqueness of canonical reduced GLT symbols, up to equality almost everywhere. 
    Consider an unbounded GLT $\{A_\bfn\}_\bfn$ over a domain $\Omega$, with a symbol $f:\Omega\times[-\pi,\pi]^d\to\mathbb{C}$, computed via an approximating sequence. Then, for every open bounded domain $\Omega_0\subset\Omega$, such that $\Ld(\partial\Omega_0)=0$, from Theorem \ref{perm_spectr_info}, we know that $f|_{\Omega_0\times[-\pi,\pi]^d}$ is uniquely determined as the canonical reduced GLT symbol of $\mathcal{R}_{\Omega,\Omega_0}(\{A_\bfn\}_\bfn)$.\\ Since $\Omega$ can be covered by open bounded domains with boundary of zero measure (see, e.g., Example \ref{exmp:reg_exhaustion}), this implies that $f$ is determined, up to equality almost everywhere, by its restrictions on the cover.

    Now, consider a regular exhaustion $\{\Omega_t\}_t$ of $\Omega$. The natural candidate of approximating sequence for $\{A_\bfn\}_\bfn$ associated to $\{\Omega_t\}_t$ is 
    \begin{equation*}
        \{B_{\bfn,t}\}_\bfn:=\mathcal{R}_{\Omega,\Omega_t}(\{A_\bfn\}_\bfn).
    \end{equation*}
    Indeed, by Theorem \ref{perm_spectr_info} and by the previous argument, items $(i)$ and $(iii)$ of Definition \ref{def_UGLT} are automatically satisfied, with
    $\{B_{\bfn,t}\}_\bfn\sim_{\mathrm{GLT}}^{\Omega_t}f|_{\Omega_t\times[-\pi,\pi]^d}$. We are left to prove that $\{\{B_{\bfn,t}\}_\bfn\}_t$ is a g.a.c.s. for $\{A_\bfn\}_\bfn$.
    By Remark \ref{rem_obs_ext_operator}, for every $t$, we can write
    \begin{equation*}
        \{A_\bfn\}_\bfn=\mathcal{E}_{\Omega_t,\Omega}(\{B_{\bfn,t}\}_\bfn)+\{S_{\bfn,t}\}_{\bfn},
    \end{equation*}
    with 
    \begin{equation*}
        \mathrm{rank}(S_{\bfn,t})\leq 2(d_{\bfn}^{\Omega}-d_\bfn^{\Omega_t}),
    \end{equation*}
    for every $\bfn$ and every $t$. In particular, thanks to item (ii) of Lemma \ref{dimensions}, this implies that both the size difference and the rank correction satisfy the hypotheses for the g.a.c.s. convergence, with 
    \begin{equation*}
        m(t)=c(t):=2c'(\Omega)\Ld(\Omega\setminus\Omega_t).
    \end{equation*}
    Since there are no norm corrections, this implies that $\{\{B_{\bfn,t}\}_\bfn\}_t$ is a g.a.c.s. for $\{A_\bfn\}_\bfn$, as desired, proving that any regular exhaustion of $\Omega$ possesses an approximating sequence for $\{A_\bfn\}$, in the sense of Definition \ref{def_UGLT}.
\end{proof}

\subsection{Algebra properties, a.c.s. closure and equivalence with measurable functions}

Finally, we are ready to establish the properties of the set of unbounded GLT over any domain. Indeed, in the next result, we show that it is an algebra of matrix-sequences, and shares all the crucial properties as the usual algebras of (reduced) GLT sequences.

\begin{thm}[The algebra of unbounded GLT]\label{main_theorem}
Let $\Omega$ be an open domain, such that $\Ld(\Omega)<\infty$, $\Ld(\partial\Omega)=0$, and define $\mathcal{G}_{\Omega}$ to be the set of all unbounded GLT over $\Omega$. Consider $\{A_\bfn\}_\bfn,\{A'_\bfn\}_\bfn\in\mathcal{G}_{\Omega}$, with canonical symbols $f,g$, and $\alpha,\beta\in\mathbb{C}$. Then
\begin{itemize}
    \item[(i)]\label{main_item_i} $\{A_\bfn^*\}_\bfn\in\mathcal{G}_\Omega$ and $ \{A^*_{\bfn}\}_{\bfn}\sim_{\mathrm{GLT}}^{\Omega}\,\bar{f};$
    \item[(ii)]\label{main_item_ii} $\{\alpha A_\bfn+\beta A'_\bfn\}_\bfn\in\mathcal{G}_\Omega$ and $\{\alpha A_{\bfn}+\beta A'_\bfn\}_{\bfn}\sim_{\mathrm{GLT}}^{\Omega}\,\alpha f+\beta g;$
    \item[(iii)]\label{main_item_iii} $\{A_\bfn A'_\bfn\}_\bfn\in\mathcal{G}_\Omega$ and $\{A_{\bfn} A'_\bfn\}_{\bfn}\sim_{\mathrm{GLT}}^{\Omega}\,fg$.
\end{itemize}
\end{thm}
\begin{proof}
    Item (i) is an immediate consequence of Definition \ref{def_UGLT}, together with the fact that reduced GLT algebras are closed under taking the complex conjugate, and that the complex conjugate operation preserves ranks and norms.

    We prove item (ii)-(iii) jointly, fully relying on Corollary \ref{indep_on_approx}. Indeed, we consider a regular exhaustion $\{\Omega_t\}_t$ of $\Omega$ and, thanks to the proof of Corollary \ref{indep_on_approx}, we can write
    \begin{align*}
         \{A_\bfn\}_\bfn&=\mathcal{E}_{\Omega_t,\Omega}(\{B_{\bfn,t}\}_\bfn)+\{S_{\bfn,t}\}_{\bfn},\\
          \{A'_\bfn\}_\bfn&=\mathcal{E}_{\Omega_t,\Omega}(\{B'_{\bfn,t}\}_\bfn)+\{S'_{\bfn,t}\}_{\bfn},
    \end{align*}
    where
    \begin{equation*}
        \{B_{\bfn,t}\}_\bfn:=\mathcal{R}_{\Omega,\Omega_t}(\{A_\bfn\}_\bfn),\qquad\{B'_{\bfn,t}\}_\bfn:=\mathcal{R}_{\Omega,\Omega_t}(\{A'_\bfn\}_\bfn),
    \end{equation*}
    and $\{S_{\bfn,t}\}_\bfn$, $\{S'_{\bfn,t}\}_\bfn$ are appropriate rank corrections, satisfying
    \begin{equation*}
        \mathrm{rank}(S_{\bfn,t}),\mathrm{rank}(S'_{\bfn,t})\leq 2(d_{\bfn}^{\Omega}-d_\bfn^{\Omega_t}),\quad\forall \bfn,\,\forall t.
    \end{equation*}
    Then, we recall that the algebras of reduced GLT are closed under linear combinations and products, and that for those algebras the operation of taking canonical symbols respects the algebra structure, yielding 
    \begin{align*}
        \{\alpha B_{\bfn,t}+\beta B'_{\bfn,t}\}&\sim_{\mathrm{GLT}}^{\Omega_t}(\alpha f+\beta g)|_{\Omega_t\times[-\pi,\pi]^d},\\
        \{B_{\bfn,t} B'_{\bfn,t}\}&\sim_{\mathrm{GLT}}^{\Omega_t} (fg)|_{\Omega_t\times[-\pi,\pi]^d}.
    \end{align*}
    Moreover, recalling that the extension operators are $*$-algebra morphisms (see Remarks \ref{well-posedness}-\ref{algebra_morphisms_red}), it follows that
    \begin{equation}\label{main_eq_1}
        \{\alpha A_\bfn+\beta A'_\bfn\}_\bfn=\mathcal{E}_{\Omega_t,\Omega}(\{\alpha B_{\bfn,t}+\beta B'_{\bfn,t}\}_\bfn)+\{S_{\bfn,t}+S'_{\bfn,t}\}_{\bfn},
    \end{equation}
    \begin{equation}\label{main_eq_2}
        \{A_\bfn A'_\bfn\}_\bfn =\mathcal{E}_{\Omega_t,\Omega}(\{B_{\bfn,t} B'_{\bfn,t}\}_\bfn)+\mathcal{E}_{\Omega_t,\Omega}(\{B_{\bfn,t}\}_\bfn)\{S'_{\bfn,t}\}_{\bfn}+\{S_{\bfn,t}\}_{\bfn}\mathcal{E}_{\Omega_t,\Omega}(\{B'_{\bfn,t}\}_\bfn)+\{S_{\bfn,t} S'_{\bfn,t}\}_{\bfn}.
    \end{equation}
    By using item (ii) of Lemma \ref{dimensions}, together with 
    \begin{equation*}
        \mathrm{rank}(S_{\bfn,t}+S'_{\bfn,t})\leq 4(d_{\bfn}^{\Omega}-d_\bfn^{\Omega_t}),
    \end{equation*}
    \begin{equation*}
        \mathrm{rank}(E_{\bfn,\Omega_t,\Omega}(B_{\bfn,t})S'_{\bfn,t}+S_{\bfn,t}E_{\bfn,\Omega_t,\Omega}(B'_{\bfn,t})+S_{\bfn,t} S'_{\bfn,t})\leq 6(d_{\bfn}^{\Omega}-d_\bfn^{\Omega_t}),
    \end{equation*}
    Equations \eqref{main_eq_1}-\eqref{main_eq_2} prove that
    \begin{itemize}
        \item $\{\{\alpha B_{\bfn,t}+\beta B'_{\bfn,t}\}_\bfn\}_t$ is a g.a.c.s. for $\{\alpha A_\bfn + \beta A'_\bfn\}_\bfn$;
        \item $\{\{B_{\bfn,t} B'_{\bfn,t}\}_\bfn\}_t$ is a g.a.c.s. for $\{A_\bfn A'_\bfn\}_\bfn$.
    \end{itemize}
    If we denote $f_t:=f|_{\Omega_t\times[-\pi,\pi]^d}$ and $g_t:=g|_{\Omega_t\times[-\pi,\pi]^d}$, then, since $\{\Omega_t\}_t$ is a regular exhaustion of $\Omega$, we have
    \begin{itemize}
        \item $\alpha f_t^E+\beta g_t^E\to \alpha f+\beta g$, in measure;
        \item $f_t^Eg_t^E\to fg$, in measure,
    \end{itemize}
    concluding the proof of items (ii)-(iii), according to Definition \ref{def_UGLT}.
\end{proof}

\begin{thm}[a.c.s. closure]\label{thm_acs_closure}
    Let $\Omega$ be an open domain, such that $\Ld(\Omega)<+\infty$ and $\Ld(\partial\Omega)=0$. Consider a sequence of unbounded $\mathrm{GLT}$ $\{\{A_{\bfn,s}\}_\bfn\}_s\sim_{\mathrm{GLT}}^{\Omega}f_s$, over $\Omega$. Then, the following are equivalent:
    \begin{itemize}
        \item[(i)]\label{thm_acs_closure_item_i} There exists $\{A_\bfn\}_\bfn\in\mathcal{M}_\Omega$, such that $\{\{A_{\bfn,s}\}_\bfn\}_s\overset{a.c.s.}{\longrightarrow}\{A_\bfn\}_\bfn$; 
        \item[(ii)]\label{thm_acs_closure_item_ii} There exists $f\in\mathcal{L}^0(\Omega\times[-\pi,\pi]^d)$ such that $f_s\to f$ in measure.
    \end{itemize}
    Moreover, in such a case, $\{A_\bfn\}_\bfn\in\mathcal{G}_\Omega$ and $f$ is its unbounded GLT symbol.
\end{thm}
\begin{proof}First of all, consider a regular exhaustion $\{\Omega_t\}_t$ of $\Omega$.

    \noindent{\textbf{Step 1.}} We prove that $(i)\Rightarrow(ii)$. 
    
    Indeed, assume that there exists $\{A_\bfn\}_\bfn\in\mathcal{M}_\Omega$, such that $\{\{A_{\bfn,s}\}_\bfn\}_s\overset{a.c.s.}{\longrightarrow}\{A_\bfn\}_\bfn$.

    By continuity with respect to the a.c.s. convergence of the restriction operators (recall Remark \ref{rem_restr_btw_domains}), for every $t$, it follows that
    \begin{equation*}
        \mathcal{R}_{\Omega,\Omega_t}(\{\{A_{\bfn,s}\}_\bfn\}_s)\overset{a.c.s.}{\longrightarrow}\mathcal{R}_{\Omega,\Omega_t}(\{A_\bfn\}_\bfn),\qquad s\to+\infty.
    \end{equation*}
    On the other hand, by Theorem \ref{perm_spectr_info}, for every $s$ and every $t$, it holds 
    \begin{equation*}
        \mathcal{R}_{\Omega,\Omega_t}(\{A_{\bfn,s}\}_\bfn)\sim_{\mathrm{GLT}}^{\Omega_t}f_s|_{\Omega_t\times[-\pi,\pi]^d}.
    \end{equation*}
    By Theorem \ref{isometry_red_GLT}, any reduced GLT algebra is closed under a.c.s. convergence. Moreover, the a.c.s. convergence of a sequence of reduced GLT is equivalent to the convergence in measure of the corresponding canonical symbols. Therefore, for every $t$, we have $\mathcal{R}_{\Omega,\Omega_t}(\{A_\bfn\}_\bfn)\in\mathcal{G}_{\Omega_t}$ and there exists a function $g_t\in\mathcal{L}^0(\Omega_t\times[-\pi,\pi]^d)$, such that
    \begin{itemize}
        \item $f_s|_{\Omega_t\times[-\pi,\pi]^d}\to g_t$ in measure as $s\to +\infty$;
        \item $\mathcal{R}_{\Omega,\Omega_t}(\{A_\bfn\}_\bfn)\sim_{\mathrm{GLT}}^{\Omega_t}g_t$.
    \end{itemize}
    Since the restriction operators factor through intermediate domains and by the uniqueness of the reduced GLT symbol, for every $t<t'$, we have $g_{t'}|_{\Omega_{t}\times[-\pi,\pi]^d}=g_t$ almost everywhere in $\Omega_t\times[-\pi,\pi]^d$. Thus, we can define a measurable function $f:\Omega\times[-\pi,\pi]^d\to\mathbb{C}$ by declaring it to be equal to $g_t$ on $\Omega_t\times[-\pi,\pi]^d$, for every $t$. 
    Thanks to Remark \ref{rem_obs_ext_operator}, for every $t$, we can write 
    \begin{equation*}
        \{A_{\bfn}\}_\bfn=\mathcal{E}_{\Omega_t,\Omega}(\mathcal{R}_{\Omega,\Omega_t}(\{A_{\bfn}\}_\bfn))+\{S_{\bfn,t}\}_\bfn,
    \end{equation*}
    with $\mathrm{rank}(S_{\bfn,t})\leq 2(d_{\bfn}^{\Omega}-d_\bfn^{\Omega_t})$.
    
    Applying item (ii) of Lemma \ref{dimensions}, it follows that $\{\mathcal{R}_{\Omega,\Omega_t}(\{A_\bfn\}_\bfn)\}_t$ is a g.a.c.s. for $\{A_\bfn\}_\bfn$ made with compatible reduced GLT sequences. By Definition \ref{def_UGLT}, we conclude that $\{A_\bfn\}_\bfn\in\mathcal{G}_\Omega$ with canonical symbol $f$.
    
    It remains to prove that $f_s\to f$ in measure in $\Omega$. Indeed, for every $\delta>0$, there exists $t$ such that $\Ld(\Omega\setminus\Omega_t)\leq\delta$. As a consequence, for every $\varepsilon>0$, we have
    \begin{equation*}
        \mathcal{L}^{2d}(\{|f_s-f|>\varepsilon\})\leq (2\pi)^d\Ld(\Omega\setminus\Omega_t)+\mathcal{L}^{2d}(\{|(f_s-f)|_{\Omega_t}|>\varepsilon\})=(2\pi)^d\delta+\mathcal{L}^{2d}(\{|f_s|_{\Omega_t}-g_t|>\varepsilon\}),
    \end{equation*}
    so that passing to the $\limsup$, we have
    \begin{equation*}
        \limsup\limits_{s\to+\infty}\mathcal{L}^{2d}(\{|f_s-f|>\varepsilon\})\leq (2\pi)^d\delta+\limsup\limits_{s\to+\infty}\mathcal{L}^{2d}(\{|f_s|_{\Omega_t\times[-\pi,\pi]^d}-g_t|>\varepsilon\})=(2\pi)^d\delta.
    \end{equation*}
    By the arbitrariness of $\delta>0$, we conclude that $f_s\to f$ in measure.

    We divide the proof of $(ii)\Rightarrow(i)$ in two steps.\\
    \noindent{\textbf{Step 2.}} We find a candidate a.c.s. limit for $\{\{A_{\bfn,s}\}_\bfn\}_s$ in $\mathcal{M}_\Omega$, and we prove that it is an unbounded GLT over $\Omega$, with canonical symbol $f$.
    
    Indeed, assume that there exists $f\in\mathcal{L}^0(\Omega\times[-\pi,\pi]^d)$ such that $f_s\to f$ in measure, as $s\to+\infty$. Applying again Theorem \ref{isometry_red_GLT} to the sequences $\{\mathcal{R}_{\Omega,\Omega_t}(\{A_{\bfn,s}\}_\bfn)\}_s$ for $t$ fixed, since $f_s|_{\Omega_t}\to f|_{\Omega_t}$ in measure, it follows that, for every $t$ there is a sequence $\{B_{\bfn,t}\}_\bfn\in\mathcal{G}_{\Omega_t}$, such that
    \begin{itemize}
        \item $\mathcal{R}_{\Omega,\Omega_t}(\{\{A_{\bfn,s}\}_\bfn\}_s)\overset{a.c.s.}{\longrightarrow}\{B_{\bfn,t}\}_\bfn,$ as $s\to+\infty;$
        \item $\{B_{\bfn,t}\}_\bfn\sim_{\mathrm{GLT}}^{\Omega_t}f|_{\Omega_t\times[-\pi,\pi]^d}$.
    \end{itemize}
    Now, by a.c.s. continuity of the restriction operators, for every $t<t'$, it follows
    \begin{equation*}
        \mathcal{R}_{\Omega_{t'},\Omega_t}(\{B_{\bfn,t'}\}_\bfn)-\{B_{\bfn,t}\}_\bfn\sim_{\mathrm{GLT}}^{\Omega_t}0.
    \end{equation*}
    Applying the extension operator $\mathcal{E}_{\Omega_t,\Omega}$ to the previous identity and using the a.c.s. continuity, we get
    \begin{equation}\label{acs_eq_1}
        \mathcal{E}_{\Omega_t,\Omega}(\mathcal{R}_{\Omega_{t'},\Omega_t}(\{B_{\bfn,t'}\}_\bfn))-\mathcal{E}_{\Omega_t,\Omega}(\{B_{\bfn,t}\}_\bfn)\sim_{\sigma}0.
    \end{equation}
    As usual, by Remark \ref{rem_obs_ext_operator}, for every $t<t'$, we can write
    \begin{equation}\label{acs_eq_2}
        \{B_{\bfn,t'}\}=\mathcal{E}_{\Omega_{t'},\Omega}(\mathcal{R}_{\Omega_{t'},\Omega_t}(\{B_{\bfn,t'}\}_\bfn))+\{S_{\bfn,t,t'}\}_\bfn,
    \end{equation}
    with $\mathrm{rank}(S_{\bfn,t,t'})\leq 2(d_\bfn^{\Omega_{t'}}-d_\bfn^{\Omega})$.
    
    Since the extension operator $\mathcal{E}_{\Omega_{t'},\Omega}$ is rank preserving, by definition of $d_{a.c.s.}$ together with item (iii) of Lemma \ref{dimensions}, this implies that 
    \begin{align}\label{acs_eq_3}
        d_{a.c.s.}(\mathcal{E}_{\Omega_t,\Omega}(\{S_{\bfn,t,t'}\}_\bfn),\{0\}_\bfn)&=\limsup\limits_{\bfn\to\infty} p(\mathcal{E}_{\Omega_t,\Omega}(\{S_{\bfn,t,t'}\}_\bfn))\\
       \nonumber &\leq\limsup\limits_{\bfn\to\infty}2\frac{d_\bfn^{\Omega_{t'}}-d_\bfn^{\Omega_t}}{d_\bfn^{\Omega}}\\
       \nonumber &=2\frac{\Ld(\Omega_{t'})-\Ld(\Omega_t)}{\Ld(\Omega)}=2\frac{\Ld(\Omega_{t'}\setminus\Omega_t)}{\Ld(\Omega)}.
    \end{align}
    Combining now Equations \eqref{acs_eq_1}-\eqref{acs_eq_3}, for $t<t'$, we obtain
    \begin{align*}
        d_{a.c.s.}(\mathcal{E}_{\Omega_{t'},\Omega}(\{B_{\bfn,t'}\}_\bfn),\mathcal{E}_{\Omega_t,\Omega}(\{B_{\bfn,t}\}_\bfn))&\leq d_{a.c.s.}(\mathcal{E}_{\Omega_{t'},\Omega}(\{B_{\bfn,t'}\}_\bfn),\mathcal{E}_{\Omega_t,\Omega}(\mathcal{R}_{\Omega_{t'},\Omega_t}(\{B_{\bfn,t'}\}_\bfn)))\\
        &+d_{a.c.s.}(\mathcal{E}_{\Omega_t,\Omega}(\mathcal{R}_{\Omega_{t'},\Omega_t}(\{B_{\bfn,t'}\}_\bfn)),\mathcal{E}_{\Omega_t,\Omega}(\{B_{\bfn,t}\}_\bfn))\\
        &=d_{a.c.s.}(\mathcal{E}_{\Omega_t,\Omega}(\{S_{\bfn,t,t'}\}_\bfn),0)+0\\
        &\leq 2\frac{\Ld(\Omega_{t'}\setminus\Omega_t)}{\Ld(\Omega)}.
    \end{align*}
    In addition, since 
    \begin{equation*}
        \lim\limits_{(t,t')\to+\infty}2\frac{\Ld(\Omega_{t'}\setminus\Omega_t)}{\Ld(\Omega)}=0,
    \end{equation*}
   it follows that $\{\mathcal{E}_{\Omega_t,\Omega}(\{B_{\bfn,t}\}_\bfn)\}_t$ is a.c.s.-Cauchy in $\mathcal{M}_\Omega$. By Theorem \ref{completeness_acs_conv}, there exists $\{A_\bfn\}_\bfn\in\mathcal{M}_\Omega$, such that $\mathcal{E}_{\Omega_t,\Omega}(\{B_{\bfn,t}\}_\bfn)\overset{a.c.s.}{\longrightarrow}\{A_\bfn\}_\bfn$. Since this is equivalent to $\{\{B_{\bfn,t}\}_\bfn\}_t$ being a g.a.c.s. for $\{A_\bfn\}_\bfn$, it follows that $\{A_\bfn\}_\bfn\in\mathcal{G}_{\Omega}$, with canonical symbol $f$. 

   \noindent{\textbf{Step 3.}} We prove that $\{\{A_{\bfn,s}\}_\bfn\}_s\overset{a.c.s.}{\longrightarrow}\{A_\bfn\}_\bfn$, as $s\to+\infty$.
   
   Since $f_s\to f$ in measure in $\Omega\times[-\pi,\pi]^d$, clearly, for every $t$, $f_s|_{\Omega_t\times[-\pi,\pi]^d}\to f|_{\Omega_t\times[-\pi,\pi]^d}$ in measure in $\Omega_t\times[-\pi,\pi]^d$. Nevertheless, we need a more refined estimate of the convergence of the restrictions.
   
   Fix $s$ and $\varepsilon>0$ and, by definition of $d_{m,\Omega}$, there are two measurable sets $F_\varepsilon,H_\varepsilon\subset\Omega\times[-\pi,\pi]^d$, such that
   \begin{itemize}
       \item $F_\varepsilon\sqcup H_\varepsilon=\Omega\times[-\pi,\pi]^d$;
       \item $\frac{\mathcal{L}^{2d}(F_\varepsilon)}{(2\pi)^d\Ld(\Omega)}+\esssup\limits_{H_\varepsilon}|f_s-f|\leq d_{m,\Omega}(f_s,f)+\varepsilon$.
   \end{itemize}
   Now, we use their intersections with $\Omega_t\times[-\pi,\pi]^d$ to estimate from above the distance $d_{m,\Omega_t}$ between some suitable rescaling of the restrictions. Indeed,
   \begin{align*}
       d_{m,\Omega}(f_s,f)+\varepsilon&\geq \frac{\mathcal{L}^{2d}(F_\varepsilon)}{(2\pi)^d\Ld(\Omega)}+\esssup\limits_{H_\varepsilon}|f_s-f|\\
       &\geq \frac{\mathcal{L}^{2d}(F_\varepsilon\cap(\Omega_t\times[-\pi,\pi]^d))}{(2\pi)^d\Ld(\Omega)}+\esssup\limits_{H_\varepsilon\cap(\Omega_t\times[-\pi,\pi]^d)}|f_s-f|\\
       &=\frac{\Ld(\Omega_t)}{\Ld(\Omega)}\left(\frac{\mathcal{L}^{2d}(F_\varepsilon\cap(\Omega_t\times[-\pi,\pi]^d))}{(2\pi)^d\Ld(\Omega_t)}+\frac{\Ld(\Omega)}{\Ld(\Omega_t)}\cdot\esssup\limits_{H_\varepsilon\cap(\Omega_t\times[-\pi,\pi]^d)}|f_s-f|\right)\\
       &\geq \frac{\Ld(\Omega_t)}{\Ld(\Omega)}d_{m,\Omega_t}\left(\frac{\Ld(\Omega)}{\Ld(\Omega_t)}f_s|_{\Omega_t\times[-\pi,\pi]^d},\frac{\Ld(\Omega)}{\Ld(\Omega_t)}f|_{\Omega_t\times[-\pi,\pi]^d}\right).
   \end{align*}
   By the arbitrariness of $\varepsilon>0$, and reversing the inequality, we obtain
   \begin{equation}\label{acs_eq_4}
       d_{m,\Omega_t}\left(\frac{\Ld(\Omega)}{\Ld(\Omega_t)}f_s|_{\Omega_t\times[-\pi,\pi]^d},\frac{\Ld(\Omega)}{\Ld(\Omega_t)}f|_{\Omega_t\times[-\pi,\pi]^d}\right)\leq \frac{\Ld(\Omega)}{\Ld(\Omega_t)} d_{m,\Omega}(f_s,f).
   \end{equation}
   On the other hand, similarly as in the proof of Theorem \ref{isometry_red_GLT}, for fixed $\bfn$ and $t$ and for any $B\in\mathbb{C}^{d_\bfn^{\Omega}\times d_\bfn^{\Omega}}$, we have
   \begin{equation}\label{acs_eq_5}
       p(E_{\bfn,\Omega_t,\Omega}(B))=\frac{d_\bfn^{\Omega_t}}{d_\bfn^{\Omega}}p\left(\frac{d_\bfn^{\Omega}}{d_\bfn^{\Omega_t}}B\right).
   \end{equation}
   Using Equation \eqref{acs_eq_5}, item (iii) of Lemma \ref{dimensions} and passing to the $\limsup$, for any $\{B_\bfn\}_\bfn\in\mathcal{M}_{\Omega_t}$, we have
   \begin{equation}\label{acs_eq_6}
       d_{a.c.s.}\left(\mathcal{E}_{\Omega_t,\Omega}(\{B_\bfn\}_\bfn),\{0\}_\bfn\right)=\frac{\Ld(\Omega_t)}{\Ld(\Omega)}d_{a.c.s.}\left(\left\{\frac{d_\bfn^{\Omega}}{d_\bfn^{\Omega_t}}B_\bfn\right\}_\bfn,\{0\}_\bfn\right).
   \end{equation}
   Now, we want to estimate the quantity $d_{a.c.s.}(\{A_{\bfn,s}\}_\bfn,\{A_{\bfn}\}_\bfn)$. In order to do so, we decompose their difference into several pieces and estimate each one of them singularly. First of all, by Remark \ref{rem_obs_ext_operator}, for every $t$ and every $s$, we can write
   \begin{equation*}
       \{A_{\bfn,s}-A_\bfn\}_\bfn =\mathcal{E}_{\Omega_t,\Omega}(\mathcal{R}_{\Omega,\Omega_t}(\{A_{\bfn,s}-A_\bfn\}_\bfn))+\{S_{\bfn,s,t}\}_\bfn,
   \end{equation*}
   with $\mathrm{rank}(S_{\bfn,s,t})\leq 2(d_\bfn^{\Omega}-d_\bfn^{\Omega_t})$. Using the linearity of the operators, we obtain 
   \begin{align}\label{acs_eq_extra}
       \{A_{\bfn,s}-A_\bfn\}_\bfn &=\mathcal{E}_{\Omega_t,\Omega}(\mathcal{R}_{\Omega,\Omega_t}(\{A_{\bfn,s}-A_\bfn\}_\bfn))+\{S_{\bfn,s,t}\}_\bfn\\
       \nonumber&=\left(\mathcal{E}_{\Omega_t,\Omega}(\mathcal{R}_{\Omega,\Omega_t}(\{A_{\bfn,s}\}_\bfn))-\mathcal{E}_{\Omega_t,\Omega}(\{B_{\bfn,t}\}_\bfn)\right)\\
       \nonumber&+\left(\mathcal{E}_{\Omega_t,\Omega}(\{B_{\bfn,t}\}_\bfn)-\mathcal{E}_{\Omega_t,\Omega}(\mathcal{R}_{\Omega,\Omega_t}(\{A_{\bfn}\}_\bfn))\right)+\{S_{\bfn,s,t}\}_\bfn\\
   \end{align}
   Then, by the rank estimate and similarly as above, we have
   \begin{equation}\label{acs_eq_7}
       d_{a.c.s.}(\{S_{\bfn,s,t}\}_\bfn,\{0\}_\bfn)\leq 2\frac{\Ld(\Omega\setminus\Omega_t)}{\Ld(\Omega)}.
   \end{equation}
   Moreover, as a byproduct of the proof that $\{A_\bfn\}_\bfn$ is an unbounded GLT over $\Omega$, and since $\mathcal{R}_{\Omega,\Omega_t}(\{A_\bfn\}_\bfn)-\{B_{\bfn,t}\}_\bfn\sim_{\mathrm{GLT}}^{\Omega_t}0$ for every $t$,
   \begin{equation}\label{acs_eq_8}
       d_{a.c.s.}(\mathcal{E}_{\Omega_t,\Omega}(\{B_{\bfn,t}\}_\bfn),\mathcal{E}_{\Omega_t,\Omega}(\mathcal{R}_{\Omega,\Omega_t}(\{A_\bfn\}_\bfn)))=d_{a.c.s.}(\{B_{\bfn,t}\}_\bfn,\mathcal{R}_{\Omega,\Omega_t}(\{A_\bfn\}_\bfn))=0.
   \end{equation}
   Finally, we consider the following rescaled sequences 
   \begin{itemize}
       \item for every $s$, $\mathcal{R}_{\Omega,\Omega_t}\left(\left\{\frac{d_\bfn^{\Omega}}{d_\bfn^{\Omega_t}}A_{\bfn,s}\right\}_\bfn\right)\sim_{\mathrm{GLT}}^{\Omega_t}\frac{\Ld(\Omega)}{\Ld(\Omega_t)}f_s|_{\Omega_t\times[-\pi,\pi]^d}$;
       \item $\left\{\frac{d_\bfn^{\Omega}}{d_\bfn^{\Omega_t}}B_{\bfn,t}\right\}_\bfn\sim_{\mathrm{GLT}}^{\Omega_t}\frac{\Ld(\Omega)}{\Ld(\Omega_t)}f|_{\Omega_t\times[-\pi,\pi]^d}$.
   \end{itemize}
   Since these are reduced GLT, we can apply to them Theorem \ref{isometry_red_GLT}, so that the a.c.s. distance between any two of them is equal to the distance with respect to $d_{m,\Omega_t}$ of the corresponding symbols. 
   Therefore, using Equations \eqref{acs_eq_4} and \eqref{acs_eq_6}, we obtain 
   \begin{align}\label{acs_eq_9}
       \nonumber d_{a.c.s.}\left(\mathcal{E}_{\Omega_t,\Omega}(\mathcal{R}_{\Omega,\Omega_t}(\{A_{\bfn,s}\}_\bfn)),\mathcal{E}_{\Omega_t,\Omega}(\{B_{\bfn,t}\}_\bfn)\right)&=\frac{\Ld(\Omega_t)}{\Ld(\Omega)}d_{a.c.s.}\left(\mathcal{R}_{\Omega,\Omega_t}\left(\left\{\frac{d_\bfn^{\Omega}}{d_\bfn^{\Omega_t}}A_{\bfn,s}\right\}_\bfn\right),\left\{\frac{d_\bfn^{\Omega}}{d_\bfn^{\Omega_t}}B_{\bfn,t}\right\}_\bfn\right)\\
         &=\frac{\Ld(\Omega_t)}{\Ld(\Omega)}d_{m,\Omega_t}\left(\frac{\Ld(\Omega)}{\Ld(\Omega_t)}f_s|_{\Omega_t\times[-\pi,\pi]^d},\frac{\Ld(\Omega)}{\Ld(\Omega_t)}f|_{\Omega_t\times[-\pi,\pi]^d}\right)\\
        \nonumber &\leq \frac{\Ld(\Omega_t)}{\Ld(\Omega)}\frac{\Ld(\Omega)}{\Ld(\Omega_t)}d_{m,\Omega}(f_s,f)=d_{m,\Omega}(f_s,f).
   \end{align}
   To conclude, we compute the a.c.s. distance between $\{A_{\bfn,s}\}_\bfn$ and $\{A_\bfn\}_\bfn$, by using the expansion given by Equation \eqref{acs_eq_extra}, together with the estimates in Equations \eqref{acs_eq_7}-\eqref{acs_eq_9}. Indeed, for every $t$, we have 
   \begin{align}\label{acs_eq_10}
       d_{a.c.s.}\left(\{A_{\bfn,s}\}_\bfn, \{A_\bfn\}_\bfn\right)&\leq d_{a.c.s.}\left(\mathcal{E}_{\Omega_t,\Omega}(\mathcal{R}_{\Omega,\Omega_t}(\{A_{\bfn,s}\}_\bfn)),\mathcal{E}_{\Omega_t,\Omega}(\{B_{\bfn,t}\}_\bfn)\right)\\
       \nonumber&+d_{a.c.s.}(\mathcal{E}_{\Omega_t,\Omega}(\{B_{\bfn,t}\}_\bfn),\mathcal{E}_{\Omega_t,\Omega}(\mathcal{R}_{\Omega,\Omega_t}(\{A_\bfn\}_\bfn)))\\
       \nonumber&+d_{a.c.s.}(\{S_{\bfn,s,t}\}_\bfn,\{0\}_\bfn)\\
       \nonumber&\leq d_{m,\Omega}(f_s,f)+ 0 + 2\frac{\Ld(\Omega\setminus\Omega_t)}{\Ld(\Omega)}.
   \end{align}
   Since Equation \eqref{acs_eq_10} holds for every $t$, by letting $t\to+\infty$, we obtain
   \begin{equation*}
       d_{a.c.s.}\left(\{A_{\bfn,s}\}_\bfn, \{A_\bfn\}_\bfn\right)\leq d_{m,\Omega}(f_s,f),
   \end{equation*}
   which concludes the proof, since $d_{m,\Omega}(f_s,f)\to 0$ as $s\to+\infty$.
   \end{proof}

\begin{exmp}\label{example_UGLT}
    Let $\Omega$ be an open domain, such that $\Ld(\Omega)<+\infty$ and $\Ld(\partial\Omega)=0$. In addition, consider an open bounded domain $\Omega_0\subset\Omega$, such that $\Ld(\partial\Omega_0)=0$. Then, for every $\{B_\bfn\}_\bfn\in\mathcal{G}_{\Omega_0}$ with canonical symbol $f$ over $\Omega_0 \times [-\pi,\pi]^d$, the sequence $\mathcal{E}_{\Omega_0,\Omega}(\{B_\bfn\}_\bfn)\in\mathcal{M}_\Omega$ is an unbounded GLT, with canonical symbol $f^E$ (i.e. the extension by $0$ of $f$ to the whole $\Omega \times [-\pi,\pi]^d$).
\end{exmp}
\begin{proof}
    Consider a regular exhaustion $\{\Omega_t\}_t$ of $\Omega$. Since restriction and extension operators factor through intermediate domains, we have 
    \begin{align*}
        \mathcal{R}_{\Omega,\Omega_t}(\mathcal{E}_{\Omega_0,\Omega}(\{B_\bfn\}_\bfn))&=(\mathcal{R}_{\Omega_0\cup\Omega_t,\Omega_t}\circ\mathcal{R}_{\Omega,\Omega_0\cup\Omega_t}\circ\mathcal{E}_{\Omega_0\cup\Omega_t,\Omega}\circ\mathcal{E}_{\Omega_0,\Omega_0\cup\Omega_t})(\{B_\bfn\}_\bfn)\\
        &=(\mathcal{R}_{\Omega_0\cup\Omega_t,\Omega_t}\circ\mathrm{Id}_{\mathcal{M}_{\Omega_0\cup\Omega_t}}\circ\mathcal{E}_{\Omega_0,\Omega_0\cup\Omega_t})(\{B_\bfn\}_\bfn)\\
        &=(\mathcal{R}_{\Omega_0\cup\Omega_t,\Omega_t}\circ\mathcal{E}_{\Omega_0,\Omega_0\cup\Omega_t})(\{B_\bfn\}_\bfn).
    \end{align*}
    Therefore, since restriction and extension operators send reduced GLT to reduced GLT (and the last line of previous equation does not involve unbounded domains), we just proved that, for every $t$,
    \begin{equation*}
        \mathcal{R}_{\Omega,\Omega_t}(\mathcal{E}_{\Omega_0,\Omega}(\{B_\bfn\}_\bfn))\sim_{\mathrm{GLT}}^{\Omega_t}f^E|_{\Omega_t\times[-\pi,\pi]^d}.
    \end{equation*}
    Moreover, thanks to Remark \ref{rem_obs_ext_operator}, for every $t$, we can write
    \begin{equation*}
        \mathcal{E}_{\Omega_0,\Omega}(\{B_\bfn\}_\bfn)=\mathcal{E}_{\Omega_t,\Omega}(\mathcal{R}_{\Omega,\Omega_t}(\mathcal{E}_{\Omega_0,\Omega}(\{B_\bfn\}_\bfn)))+\{S_{\bfn,t}\}_\bfn,
    \end{equation*}
    with $\mathrm{rank}(S_{\bfn,t})\leq 2(d_\bfn^{\Omega}-d_\bfn^{\Omega_t})$. By item (ii) of Lemma \ref{dimensions}, this proves that $\{\mathcal{R}_{\Omega,\Omega_t}(\mathcal{E}_{\Omega_0,\Omega}(\{B_\bfn\}_\bfn))\}_t$ is a g.a.c.s. for $\mathcal{E}_{\Omega_0,\Omega}(\{B_\bfn\}_\bfn)$. Thus, since $f^E(\bfx,\bftheta)\mathds{1}_{\Omega_t}(\bfx)\to f^E(\bfx,\bftheta)$ in measure in $\Omega \times [-\pi,\pi]^d$, we conclude that $\mathcal{E}_{\Omega_0,\Omega}(\{B_\bfn\}_\bfn)\sim_{\mathrm{GLT}}^{\Omega}f^E$.
 \end{proof}

\begin{cor}\label{cor_UGLT_surj}
    Let $\Omega$ be an open domain, such that $\Ld(\Omega)<+\infty$ and $\Ld(\partial\Omega)=0$. Then, for every $f\in\mathcal{L}^0(\Omega\times[-\pi,\pi]^d)$, there exists $\{A_\bfn\}_\bfn\in\mathcal{G}_\Omega$, such that
    \begin{equation*}
        \{A_\bfn\}_\bfn\sim_{\mathrm{GLT}}^{\Omega}f.
    \end{equation*}
\end{cor}
\begin{proof}
    Consider $f\in\mathcal{L}^{0}(\Omega\times[-\pi,\pi]^d)$ and consider a regular exhaustion $\{\Omega_t\}_t$ of $\Omega$. Thanks to Theorem \ref{isometry_red_GLT}, for every $t$, there exists $\{B_{\bfn,t}\}_\bfn\in\mathcal{G}_{\Omega_t}$, with canonical symbol $f|_{\Omega_t\times[-\pi,\pi]^d}$. By Example \ref{example_UGLT}, we have that $\mathcal{E}_{\Omega_t,\Omega}(\{B_{\bfn,t}\}_\bfn)\in\mathcal{G}_{\Omega}$, with canonical symbol $f(\bfx,\bftheta)\mathds{1}_{\Omega_t}(\bfx)$. Since $f(\bfx,\bftheta)\mathds{1}_{\Omega_t}(\bfx)\to f(\bfx,\bftheta)$ in measure in $\Omega \times [-\pi,\pi]^d$, applying Theorem \ref{thm_acs_closure}, there exists $\{A_\bfn\}_\bfn\in\mathcal{G}_\Omega$, such that 
    \begin{itemize}
        \item $\{\mathcal{E}_{\Omega_t,\Omega}(\{B_{\bfn,t}\}_\bfn)\}_t\overset{a.c.s.}{\longrightarrow}\{A_\bfn\}_\bfn$;
        \item $\{A_\bfn\}_\bfn\sim_{\mathrm{GLT}}^{\Omega}f$,
    \end{itemize}
    which concludes the proof.
\end{proof}
The next result shows that the unbounded GLT class is closed under taking pseudo-inverses, and completes the set of algebra properties satisfied by the latter.
\begin{thm}\label{uGLT_isometry}
Let $\Omega$ be an open domain, such that $\Ld(\Omega)<+\infty$ and $\Ld(\partial\Omega)=0$. Consider $\{A_\bfn\}_\bfn\sim_{\mathrm{GLT}}^{\Omega}f$, such that $f\neq 0$ almost everywhere. Then, $\{A_\bfn^{\dagger}\}_\bfn\in\mathcal{G}_\Omega$, with canonical symbol $f^{-1}$. 
\end{thm}
\begin{proof}
Since $f\neq0$ almost everywhere, applying Corollary \ref{cor_UGLT_surj}, there exists $\{B_n\}_n\in\mathcal{G}_\Omega$, with canonical symbol $f^{-1}$. Thanks to Theorem \ref{main_theorem}, we have that
\begin{equation*}
    \{B_\bfn A_\bfn-I_\bfn\}_\bfn\sim_{\mathrm{GLT}}^{\Omega}0,
\end{equation*}
which implies that 
\begin{equation*}
    d_{a.c.s.}(\{A_\bfn B_\bfn\},\{I_\bfn\}_\bfn)=0.
\end{equation*}
Now, the matrix-sequence $\{A_\bfn\}_\bfn$ admits an almost everywhere non-vanishing symbol function $f$, so it is sparsely vanishing by Proposition \ref{sv_neq_zero_a_e}. By Remark \ref{rem_sv_su}, it follows that the matrix-sequence $\{A_\bfn^\dagger\}_\bfn$ is sparsely unbounded. As a consequence, we can apply Lemma \ref{preserving_zero_distribution} to obtain that
\begin{equation}\label{pseudoinv_eq_1}
    d_{a.c.s.}(B_\bfn A_\bfn A_\bfn^\dagger,A_\bfn^\dagger)=0.
\end{equation}
For every $\bfn$, by definition of pseudo-inverse, since $A_\bfn^\dagger$ acts as the identity on the range of $A_\bfn$ (and vice-versa), we can write
\begin{equation}\label{pseudoinv_eq_2}
    A_\bfn A_\bfn^\dagger=I_\bfn+S_\bfn,
\end{equation}
with $\mathrm{rank}(S_\bfn)=d_\bfn^{\Omega}-\mathrm{rank}(A_\bfn)$. Since $\{A_\bfn\}_\bfn$ is sparsely vanishing, in particular, we have that
\begin{equation*}
    \lim\limits_{\bfn\to \infty}\frac{\mathrm{rank}(S_\bfn)}{d_\bfn^{\Omega}}\lim\limits_{\bfn\to \infty}\frac{d_\bfn^{\Omega}-\mathrm{rank(A_\bfn)}}{d_\bfn^{\Omega}}=\lim\limits_{\bfn\to \infty}\frac{\{i\in\{1,\ldots,d_\bfn^{\Omega}\}\,\vert\,\sigma_i(A_\bfn)=0\}}{d_\bfn^{\Omega}}=0,
\end{equation*}
which, together with Equation \eqref{pseudoinv_eq_2}, implies that
\begin{equation*}
    d_{a.c.s.}(\{A_\bfn A_\bfn^{\dagger}\}_\bfn,\{I_\bfn\}_\bfn)=0.
\end{equation*}
Since $\{B_\bfn\}_\bfn$ admits a symbol function, it is sparsely unbounded by Proposition \ref{sv_neq_zero_a_e}. Thus, we can apply again Lemma \ref{preserving_zero_distribution}, to get that
\begin{equation}\label{pseudoinv_eq_3}
    d_{a.c.s.}(\{B_\bfn A_\bfn A_\bfn^\dagger\}_\bfn,\{B_\bfn\}_\bfn)=0.
\end{equation}
Using now Equations \eqref{pseudoinv_eq_1} and \eqref{pseudoinv_eq_3}, we conclude that 
\begin{equation*}
    d_{a.c.s.}(\{B_\bfn\}_\bfn,\{A_\bfn^\dagger\}_\bfn)=0,
\end{equation*}
which implies that $\{A_\bfn^{\dagger}\}\sim_{\mathrm{GLT}}^{\Omega}f^{-1}$, since the unbounded GLT algebra is closed with respect to the a.c.s. topology, by Theorem \ref{thm_acs_closure}.
\end{proof}
Finally, we are ready to extend Theorem \ref{isometry_red_GLT} to the entire unbounded GLT class.
\begin{thm}\label{thm_isom_final}
 Let $\Omega$ be an open domain, such that $\Ld(\Omega)<+\infty$ and $\Ld(\partial\Omega)=0$. Consider the map $\Phi_\Omega:(\sfrac{\mathcal{G}_\Omega}{\sim_d},d_{a.c.s.})\to(L^0(\Omega\times[-\pi,\pi]^d),d_{m,\Omega})$ that associates to any reduced GLT its canonical symbol. Then, $\Phi_\Omega$ is a surjective isometry.    
\end{thm}
\begin{proof} The surjectivity of $\Phi_\Omega$ is a consequence of Corollary \ref{cor_UGLT_surj}.

Consider an open bounded domain $\Omega_0\subset\Omega$ and $\{B_\bfn\}_\bfn\in\mathcal{G}_{\Omega_0}$, with symbol $f$ over $\Omega$. By Example \ref{example_UGLT}, we know that 
\begin{equation*}
    \mathcal{E}_{\Omega_0,\Omega}(\{B_\bfn\}_\bfn)\sim_{\mathrm{GLT}}^{\Omega}f^E.
\end{equation*}
Since $f^E=0$ almost everywhere in $(\Omega\setminus\Omega_t)\times[-\pi,\pi]^d$, recalling the proof of Theorem \ref{isometry_red_GLT}, we have
\begin{equation}\label{iso_uglt_eq_1}
    d_{m,\Omega}(f^E,0)=\frac{\mathcal{L}^d(\Omega_0)}{\Ld(\Omega)}d_{m,\Omega_0}\left(\frac{\Ld(\Omega)}{\Ld(\Omega_0)}f,0\right),
\end{equation}
and, similarly to the proof of Theorem \ref{thm_acs_closure}, for the extension of the sequence we have
\begin{equation}\label{iso_uglt_eq_2}
d_{a.c.s.}\left(\mathcal{E}_{\Omega_0,\Omega}(\{B_\bfn\}_\bfn),\{0\}_\bfn\right)=\frac{\Ld(\Omega_0)}{\Ld(\Omega)}d_{a.c.s.}\left(\left\{\frac{d_\bfn^{\Omega}}{d_\bfn^{\Omega_0}}B_\bfn\right\}_\bfn,\{0\}_\bfn\right).
\end{equation}
Since 
\begin{equation*}
    \left\{\frac{d_\bfn^{\Omega}}{d_\bfn^{\Omega_0}}B_\bfn\right\}_\bfn\sim_{\mathrm{GLT}}^{\Omega_0}\frac{\Ld(\Omega)}{\Ld(\Omega_0)}f,
\end{equation*}
by applying Theorem \ref{isometry_red_GLT}, we obtain
\begin{align*}
    d_{m,\Omega}(f^E,0)&=\frac{\mathcal{L}^d(\Omega_0)}{\Ld(\Omega)}d_{m,\Omega_0}\left(\frac{\Ld(\Omega)}{\Ld(\Omega_0)}f,0\right)\\
    &=\frac{\Ld(\Omega_0)}{\Ld(\Omega)}d_{a.c.s.}\left(\left\{\frac{d_\bfn^{\Omega}}{d_\bfn^{\Omega_0}}B_\bfn\right\}_\bfn,\{0\}_\bfn\right)\\
    &=d_{a.c.s.}\left(\mathcal{E}_{\Omega_0,\Omega}(\{B_\bfn\}_\bfn),\{0\}_\bfn\right).
\end{align*}
By invariance under translation of $d_{a.c.s.}$ and of $d_{m,\Omega}$, we just proved that $\Phi_\Omega$ is an isometry, when restricted to the set
\begin{equation*}
    \mathcal{A}_\Omega:=\left\{\mathcal{E}_{\Omega_0,\Omega}(\{B_\bfn\}_\bfn)\,\bigg\vert\,\{B_\bfn\}_\bfn\in\mathcal{G}_{\Omega_0},\,\Omega_0\subset\Omega \text{ bounded },\,\Ld(\partial\Omega_0)=0\right\}.
\end{equation*}
Finally, note that as a byproduct of the proof of Corollary \ref{cor_UGLT_surj}, the set $\mathcal{A}_\Omega$ is dense in $\mathcal{G}_\Omega$, so that $\Phi_\Omega$ is an isometry.
\end{proof}

\section{Applications and numerical results}\label{sec:appl}

In this section we introduce the concepts of unbounded Toeplitz and unbounded diagonal sampling sequences. Using these sequences and the algebra properties of reduced and unbounded GLT, we study the spectral distribution of the matrix sequence arising from the discretization of a PDE over an unbounded domain.

\subsection{Unbounded Toeplitz and diagonal sampling sequences}

Let $\Omega$ be an open unbounded domain such that $\Ld(\Omega) < \infty$ and $\Ld(\partial \Omega) =0$ and let $\{\Omega_t\}_t$ be a regular exhaustion of $\Omega$. Let now $f \in L^{1}([-\pi,\pi]^d)$ and let $\{Q_{{\bfy}_t,r_t}\}_t$ be a sequence of hypercubes such that, for every $t$, $ \Omega_t \subseteq Q_{{\bfy}_t,r_t}$ and $Q_{{\bfy_{t}},r_t} \subseteq Q_{{\bfy}_{t+1},r_{t+1}}$. Following the same ideas as in \cite{Barb}, for every $t$, we can define the sequence of reduced Toeplitz over the domain $\Omega_t$ as the restriction of the Toeplitz sequence $\{T_{r_t\bfn}(f)\}_{\bfn}$ over $Q_{{\bfy}_t,r_t}$ to $\Omega_t$. We denote this sequence by $\{T_{\bfn}^{\Omega_t}(f)\}_{\bfn}$. Note that $\{T_{\bfn}^{\Omega_t}(f)\}_{\bfn}$ is a reduced GLT sequence with the canonical symbol $f$ over $\Omega_t \times [-\pi,\pi]^d$. Finally, we consider the extension $\mathcal{E}_{\Omega_t, \Omega}( \{T_{\bfn}^{\Omega_t}(f)\}_\bfn)$ to $\Omega$. We notice that, by Example \ref{example_UGLT}, $\mathcal{E}_{\Omega_t,\Omega}(\{T_{\bfn}^{\Omega_t} (f)\}_{\bfn})$ is an unbounded GLT with canonical symbol $f_t^{E}(\bfx,\bftheta)$ over $\Omega \times [-\pi,\pi]^d$, where
\begin{equation*}
    f_t^{E}(\bfx,\bftheta) = \begin{cases}
      f(\bftheta) & (\bfx,\bftheta) \in \Omega_t \times [-\pi,\pi]^d,\\
      0 & otherwise.
    \end{cases}
\end{equation*}
By Theorem \ref{thm_acs_closure}, since $f_t^{E}$ converges in measure to $\kappa(\bfx,\bftheta)$ over $\Omega \times [-\pi,\pi]^d$, with $\kappa(\bfx,\bftheta)=f(\bftheta)$, $\{ \mathcal{E}_{\Omega_t,\Omega}(\{T_{\bfn}^{\Omega_t}(f)\}_{\bfn})\}_t$ is converging in the a.c.s.\ topology to a limit sequence. We denote such limit by $\{T_{\bfn}^{\Omega}(f)\}_{\bfn}$, which, again by Theorem \ref{thm_acs_closure}, is an unbounded GLT, referred to as unbounded Toeplitz sequence generated by $f$. Therefore,
\begin{equation*}
    \{T_{\bfn}^{\Omega}(f)\}_{\bfn} \sim_{\sigma} (\kappa, \Omega \times [-\pi,\pi]^d),
\end{equation*}
with $\kappa(\bfx,\bftheta)=f(\bftheta)$ and, if $f$ is real-valued almost everywhere,
\begin{equation*}
    \{T_{\bfn}^{\Omega}(f)\}_{\bfn} \sim_{\lambda} (\kappa, \Omega \times [-\pi,\pi]^d).
\end{equation*}

Let now $a:\Omega \to \mathbb{C}$ be an almost everywhere continuous function and denote by $a_t$ its restriction to $\Omega_t$ for every $t$, that is, $a_t : \Omega_t \to \mathbb{C}$ and $a_t(\bfx)=a(\bfx)$ for every $\bfx \in \Omega_t$.
Following the same procedure as above, we define $\{D_{\bfn}^{\Omega}(a)\}_{\bfn}$ as the a.c.s. limit of the extension to $\Omega$ of the reduced diagonal sampling sequences $\{\{D_{\bfn}^{\Omega_t}(a_t)\}_{\bfn}\}_t$ over an exhaustion $\{\Omega_t\}_t$. Clearly,
\begin{equation*}
    \{D_{\bfn}^{\Omega}(a)\}_{\bfn} \sim_{\lambda,\sigma} (\kappa, \Omega \times [-\pi,\pi]^d),
\end{equation*}
with $\kappa(\bfx,\bftheta)=a(\bfx)$.

\subsection{An application with numerical confirmation}

In this subsection we study a model problem whose discretization leads to an unbounded GLT sequence and infer a symbol for the discretized operator. Some numerical tests are conducted in order to confirm the validity of our derivations.

Consider the problem

\begin{equation}\label{test_pb}
    \begin{cases}
        \text{div}(-a\nabla u)=v & \text{ on } \Omega,\\
        u=0 & \text{ on } \partial \Omega,
    \end{cases}
\end{equation}
where $a(x,y)=a(\bfx)$ is a positive non-degenerate variable coefficient on $\Omega = \{ \bfx \in \mathbb{R}^2 \, : \, x>0, y>0, y < g(x)\}$, with
\begin{equation*}
    g(x)=\begin{cases}
        1 & \text{ if } x<1 \\
        \frac{1}{x^2} & \text{ if } x \geq 1.
    \end{cases}
\end{equation*}
Clearly, $\Omega$ is an open unbounded domain with $\Ld(\Omega) < \infty$ and $\Ld(\partial \Omega) =0$. For the discretization of \eqref{test_pb}, we consider as space-step $h=\frac{1}{n}$ and as nodes the points in $\Theta_{\bfn,\Omega} $, $\bfn=(n,n)$.

On this grid, we apply $P_1$ finite elements and, as $n$ increases, we obtain a sequence $\{A_{\bfn}\}_{\bfn}$ which constitutes the discretized version of the original continuous operator. We are interested in studying the spectral properties of the sequence $\{A_{\bfn}\}_{\bfn}$ using the theory we just developed. In order to do so, we consider the regular exhaustion $\{\Omega_t\}_t$ of $\Omega$ given by
\begin{equation*}
    \Omega_t = \left\{\bfx\in\Omega\,\vert\,\|\bfx\|_\infty<t\right\}.
\end{equation*}
As above, let $a_t$ denote the restriction of $a$ to $\Omega_t$.
For a fixed $t$, we consider the grid $\Theta_{\bfn,\Omega_t}$ related to the domain $\Omega_t$. Over this grid, we apply again $P_1$ finite elements, thus obtaining a new sequence that we denote by $\{B_{\bfn,t}\}_{\bfn}$. 

First, we prove that, for every $t$, $\{B_{\bfn,t}\}_{\bfn}$ is a reduced GLT with canonical symbol $a_t(\bfx)f(\bftheta)$ over $\Omega_t \times [-\pi,\pi]^2$. For this purpose, let $B_t = B_{\infty}({\bf0},t)= \left\{\bfx \,\vert\,\|\bfx\|_\infty<t\right\} $ and consider the following auxiliary problem:
\begin{equation}\label{aux_pb}
    \begin{cases}
        \text{div}(-a_t^E \nabla u_t^E)= v_t^E & \text{ on } B_t \\
        u_t^E=0 & \text{ on } \partial B_t,
    \end{cases}
\end{equation}
where
\begin{equation*}
    a_t^E(\bfx) = \begin{cases}
        a(\bfx) & \bfx \in \Omega_t, \\
        0 & otherwise,
    \end{cases} \quad v_t^E(\bfx) = \begin{cases}
        v(\bfx) & \bfx \in \Omega_t, \\
        0 & otherwise.
    \end{cases}
\end{equation*}
By applying $P_1$ elements to \eqref{aux_pb} on the uniform grid $\Theta_{\bfn,B_t}$, we produce a new matrix-sequence, which we denote by $\{C_{\bfn,t}\}_{\bfn}$. The spectral distribution of $\{C_{\bfn,t}\}_{\bfn}$ is known in the literature (see, e.g., \cite[Section 7]{Barb}), from which we know that $\{C_{\bfn,t}\}_{\bfn} $ is a GLT sequence with canonical symbol $ a_t^E(\bfx)h(\bftheta)$ over $B_t \times [-\pi,\pi]^2$, with $h(\bftheta)=h(\theta_1,\theta_2)=4-2\cos(\theta_2)-2\cos(\theta_2)$. In particular, $\{D_{t\bfn}(a_t^E) T_{t \bfn}(h)-C_{\bfn,t}\}_{\bfn} \sim_{\mathrm{GLT}} 0$.

Notice that, by construction, $\{B_{\bfn,t}\}_{\bfn} = \mathcal{R}_{B_t, \Omega_t}(\{C_{\bfn,t}\}_{\bfn})$, so that, by applying Theorem \ref{spectra_reduced_GLT}, we infer that $\{B_{\bfn,t}\}_{\bfn}$ is a reduced GLT with canonical symbol $a_t(\bfx)h(\bftheta)$ over $\Omega_t \times [-\pi,\pi]^2$. Again, by the algebra properties of reduced GLT, this also implies that $\{D_{\bfn}^{\Omega_t}(a_t) T_{\bfn}^{\Omega_t}(h)-B_{\bfn,t}\}_{\bfn} \sim_{\mathrm{GLT}}^{\Omega_t} 0$. 

Finally, the sequence $\{A_{\bfn}\}_{\bfn}$ naturally arises by considering the extension of $\{B_{\bfn,t}\}_{\bfn}$ to $\Omega$. Indeed, proceeding as in \cite[Section 4.1]{gacs}, we have $\{\{B_{\bfn,t}\}_{\bfn}\}_t$ is a g.a.c.s. for $\{A_{\bfn}\}_{\bfn}$, or, equivalently, $\mathcal{E}_{\Omega_t,\Omega}(\{B_{\bfn,t}\}_{\bfn})$ is an a.c.s. for $\{A_{\bfn}\}_{\bfn}$. Now, since $a_t(\bfx)h(\bftheta)$ converges in measure to $a(\bfx)h(\bftheta)$ over $\Omega \times [-\pi,\pi]^2$, applying Theorem \ref{thm_acs_closure}, we conclude that $\{A_{\bfn}\}_{\bfn}$ is an unbounded GLT sequence with canonical symbol $a(\bfx)h(\bftheta)$ over $\Omega \times [-\pi,\pi]^2$. In particular, note that by the algebra properties of unbounded GLT (Theorem \ref{main_theorem}), we have $\{D_{\bfn}^{\Omega}(a)T_{\bfn}^{\Omega}(h)-A_{\bfn}\}_{\bfn} \sim_{\mathrm{GLT}}^{\Omega} 0$.


Let $\mathrm{f}(x,y,\theta_1,\theta_2)=a(x,y)h(\theta_1,\theta_2)$. 
In Figures \ref{FDcoeff_B_t2}-\ref{FDcoeff_B_t8}, we show the behavior of the sequence $\{B_{\bfn,t}\}_{\bfn}$ for different values of the parameter $t$ and of the space-step $h$. In Figures \ref{FDcoeff_Bbig_t2}-\ref{FDcoeff_Bbig_t8}, we verify the spectral distribution of the extension $\mathcal{E}_{\Omega_t,\Omega}(\{B_{\bfn,t}\}_{\bfn})$ for different values of $t$, showing the dependence on the parameter $t$. In particular, as expected from the theory, we notice that the percentage of zero eigenvalues decreases when increasing the parameter $t$. Finally, in Figure \ref{FDcoeff_A_50} we numerically confirm the crucial result of this section, that is, $\{A_{\bfn}\}_{\bfn} \sim_{\lambda} \left(\mathrm{f},\Omega \times [-\pi,\pi]^2\right)$. Note that we make use of the notation $\mathrm{f \chi_t}$ to indicate the function defined on $\Omega \times [-\pi,\pi]^2$ equal to $\mathrm{f}$ in $\Omega_t \times [-\pi,\pi]^2$ and zero otherwise.

\begin{figure}[H]
\centering
  \includegraphics[width=\textwidth]{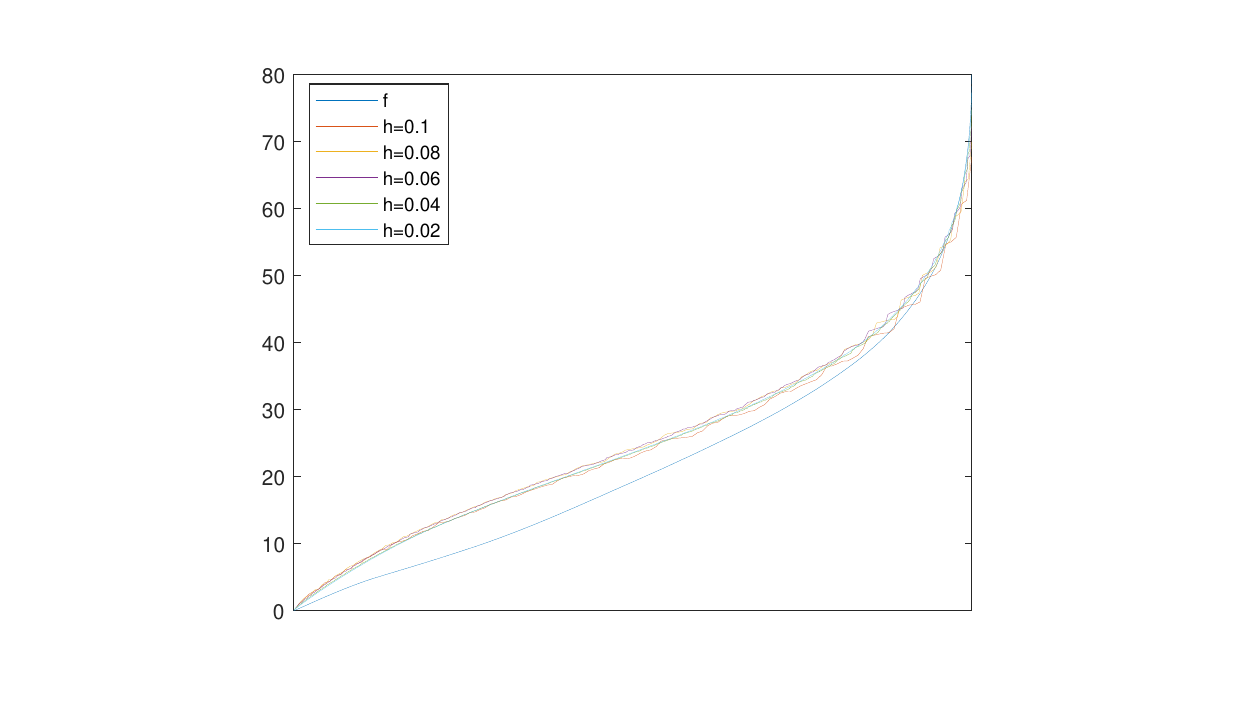} \vskip -0.5cm
  \caption{Eigenvalues distribution of $B_{\bfn,t}$ for different values of $h$  together with the sampling of $\mathrm{f}(x,y,\theta_1,\theta_2)=a(x,y)(4-2\cos \theta_1-2\cos \theta_2)$, $a(x,y)= (10 + x^2 + 2y^2+ \sin^2(x+y)) / ( 1 + x^2 + y^2)$ over $\Omega_t \times [-\pi,\pi]^2$, $t=2$.}
\label{FDcoeff_B_t2}
\end{figure}

\begin{figure}[p]
\centering
  \includegraphics[width=\textwidth]{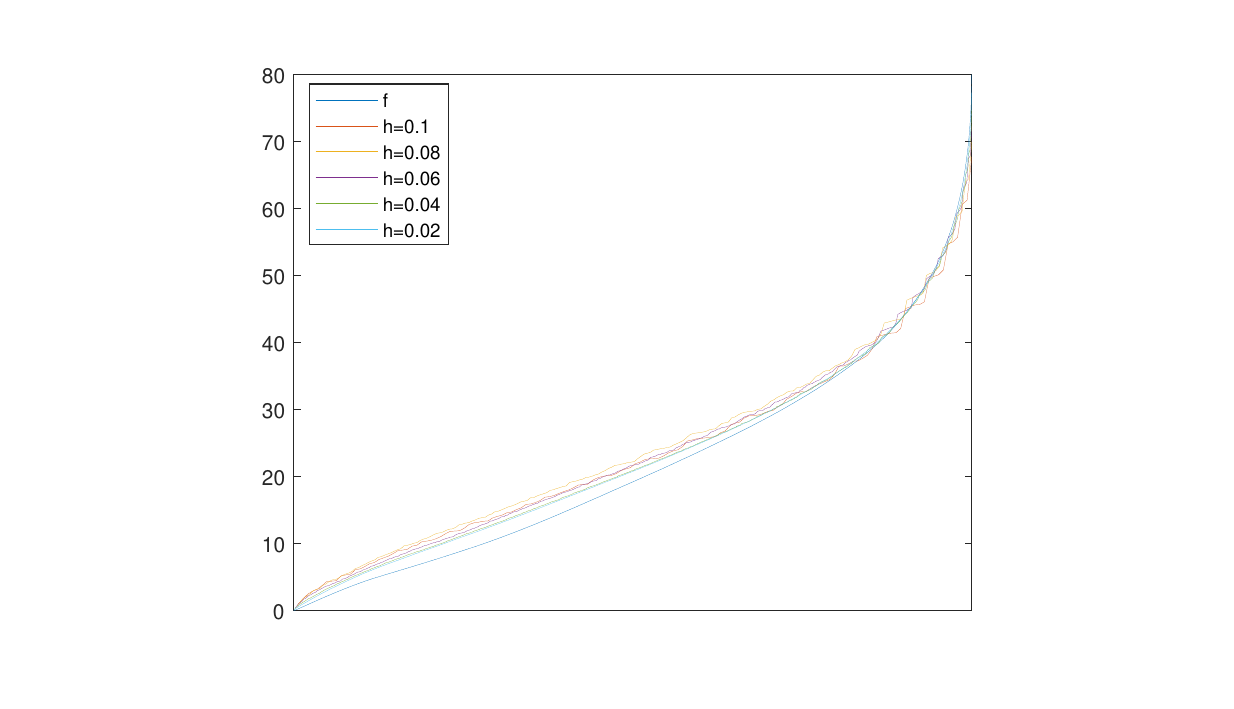} \vskip -0.5cm
  \caption{Eigenvalues distribution of $B_{\bfn,t}$ for different values of $h$  together with the sampling of $\mathrm{f}(x,y,\theta_1,\theta_2)=a(x,y)(4-2\cos \theta_1-2\cos \theta_2)$, $a(x,y)= (10 + x^2 + 2y^2+ \sin^2(x+y)) / ( 1 + x^2 + y^2)$ over $\Omega_t \times [-\pi,\pi]^2$, $t=4$.}
\label{FDcoeff_B_t4}
\end{figure}

\begin{figure}[p]
\centering
  \includegraphics[width=\textwidth]{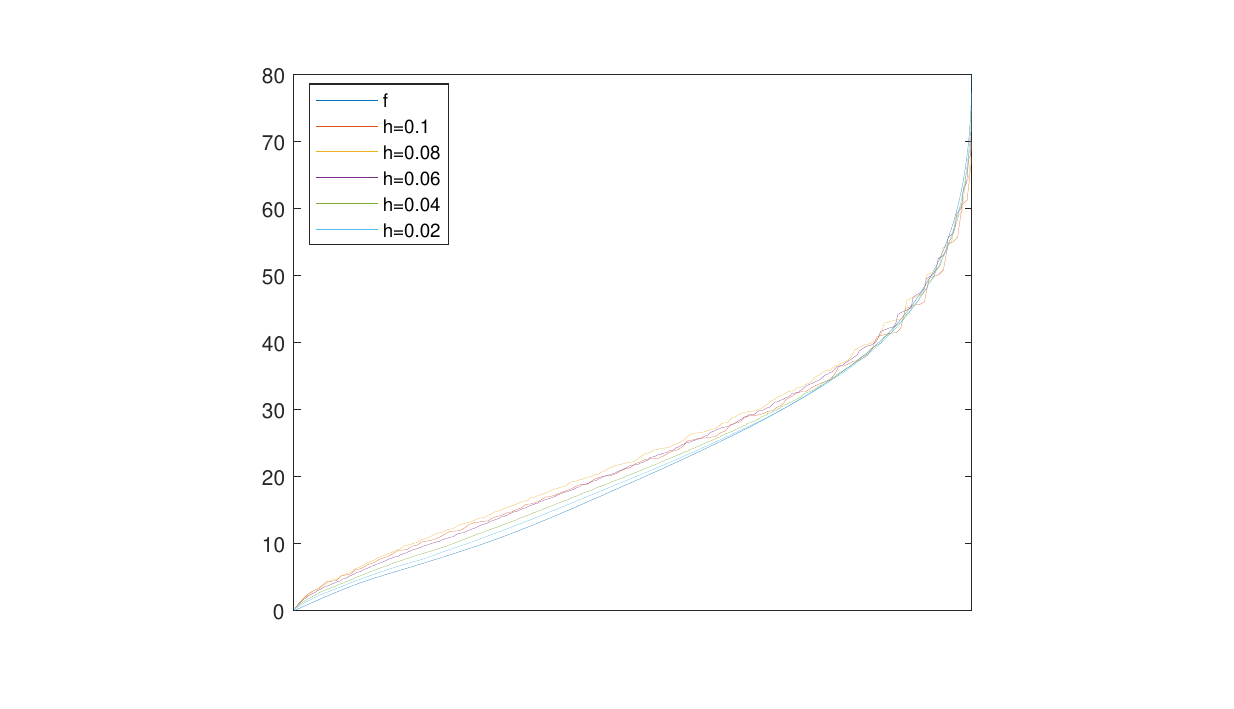} \vskip -0.5cm
  \caption{Eigenvalues distribution of $B_{\bfn,t}$ for different values of $h$  together with the sampling of $\mathrm{f}(x,y,\theta_1,\theta_2)=a(x,y)(4-2\cos \theta_1-2\cos \theta_2)$, $a(x,y)= (10 + x^2 + 2y^2+ \sin^2(x+y)) / ( 1 + x^2 + y^2)$ over $\Omega_t \times [-\pi,\pi]^2$, $t=8$.}
\label{FDcoeff_B_t8}
\end{figure}

\begin{figure}[p]
\centering
  \includegraphics[width=\textwidth]{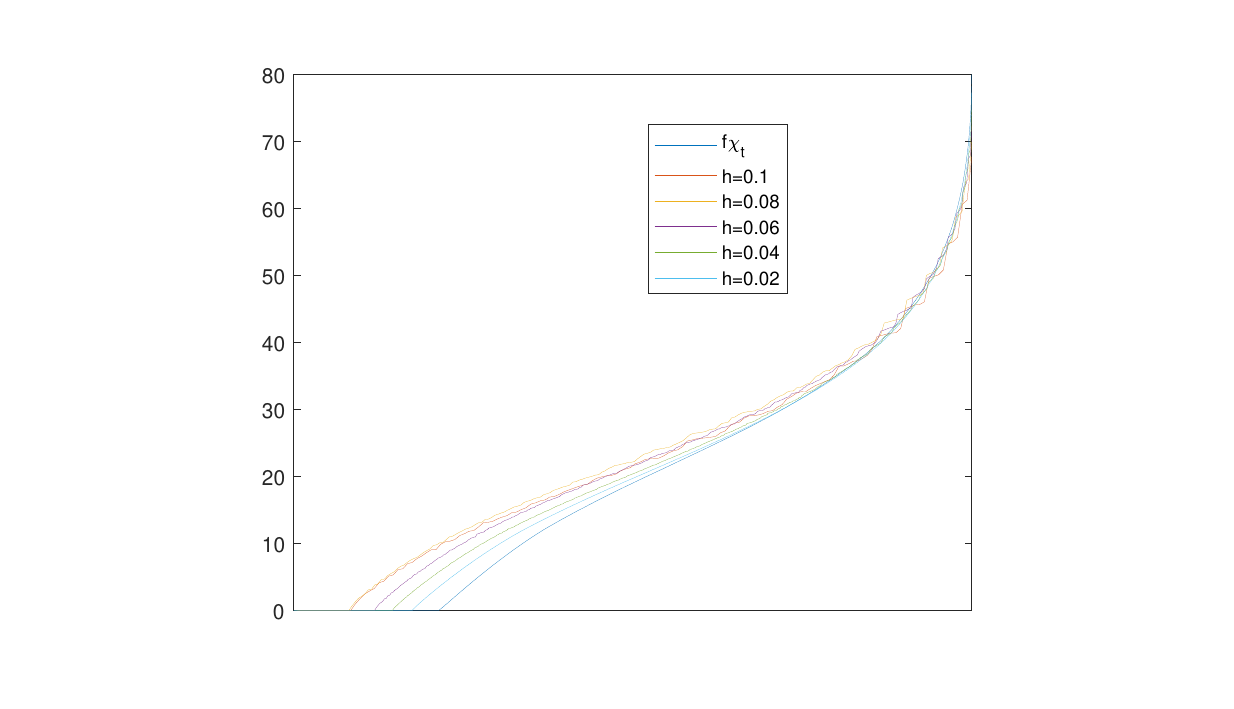} \vskip -0.5cm
  \caption{Eigenvalues distribution of $E_{\bfn,\Omega_t,\Omega}(B_{\bfn,t})$ for different values of $h$  together with the sampling of $\mathrm{f}(x,y,\theta_1,\theta_2)=a(x,y)(4-2\cos \theta_1-2\cos \theta_2)$, $a(x,y)= (10 + x^2 + 2y^2+ \sin^2(x+y)) / ( 1 + x^2 + y^2)$ over $\Omega \times [-\pi,\pi]^2$, $t=2$.}
\label{FDcoeff_Bbig_t2}
\end{figure}

\begin{figure}[p]
\centering
  \includegraphics[width=\textwidth]{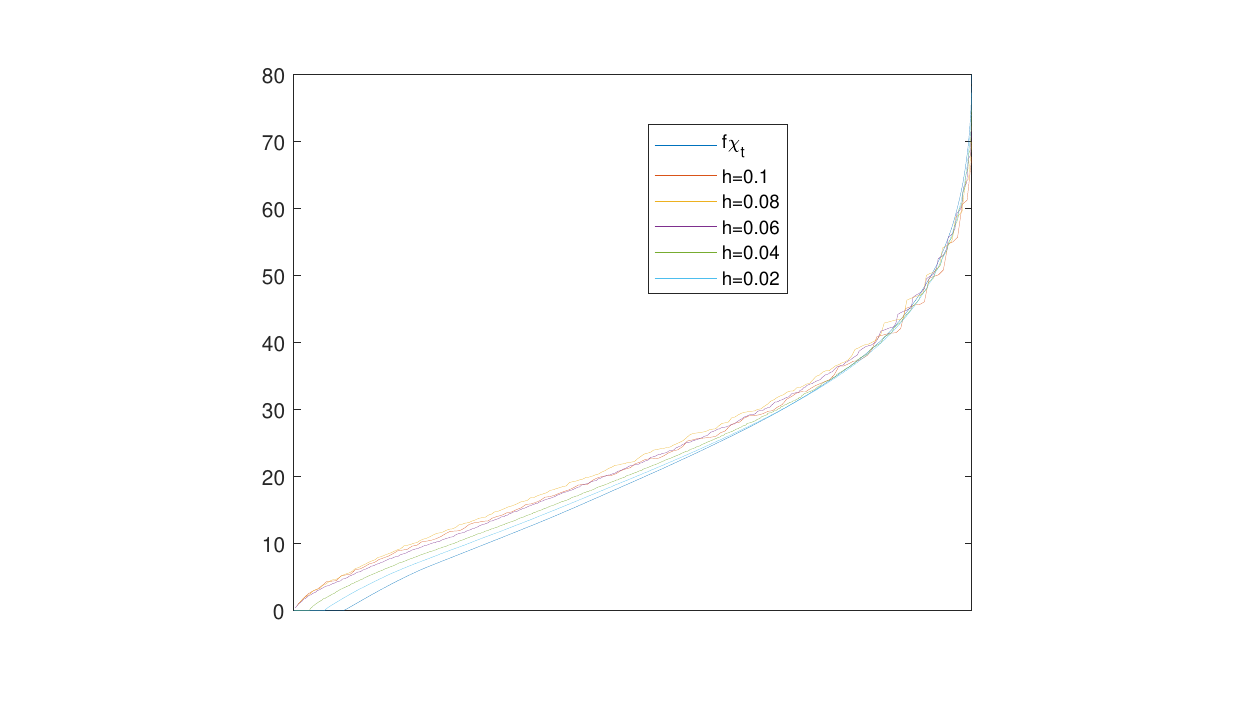} \vskip -0.5cm
  \caption{Eigenvalues distribution of $E_{\bfn,\Omega_t,\Omega}(B_{\bfn,t})$ for different values of $h$  together with the sampling of $\mathrm{f}(x,y,\theta_1,\theta_2)=a(x,y)(4-2\cos \theta_1-2\cos \theta_2)$, $a(x,y)= (10 + x^2 + 2y^2+ \sin^2(x+y)) / ( 1 + x^2 + y^2)$ over $\Omega \times [-\pi,\pi]^2$, $t=4$.}
\label{FDcoeff_Bbig_t4}
\end{figure}

\begin{figure}[p]
\centering
  \includegraphics[width=\textwidth]{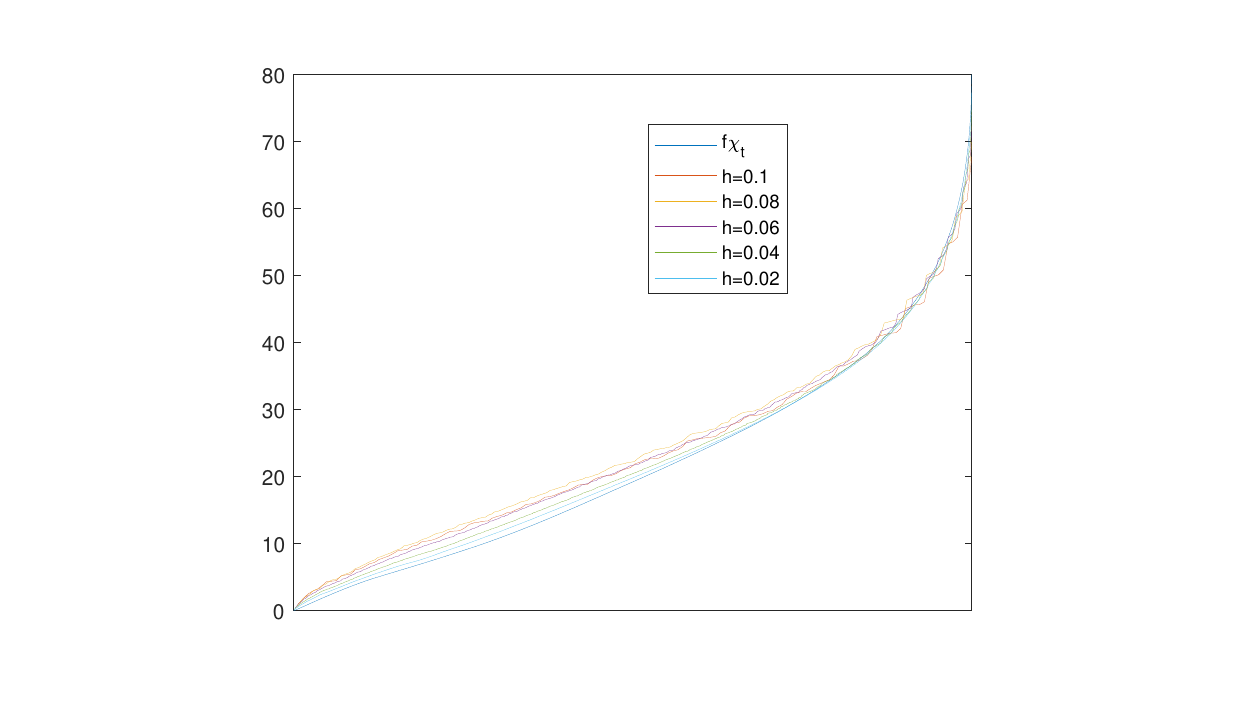} \vskip -0.5cm
  \caption{Eigenvalues distribution of $E_{\bfn,\Omega_t,\Omega}(B_{\bfn,t})$ for different values of $h$  together with the sampling of $\mathrm{f}(x,y,\theta_1,\theta_2)=a(x,y)(4-2\cos \theta_1-2\cos \theta_2)$, $a(x,y)= (10 + x^2 + 2y^2+ \sin^2(x+y)) / ( 1 + x^2 + y^2)$ over $\Omega \times [-\pi,\pi]^2$, $t=8$.}
\label{FDcoeff_Bbig_t8}
\end{figure}

\begin{figure}[p]
\centering
  \includegraphics[width=\textwidth]{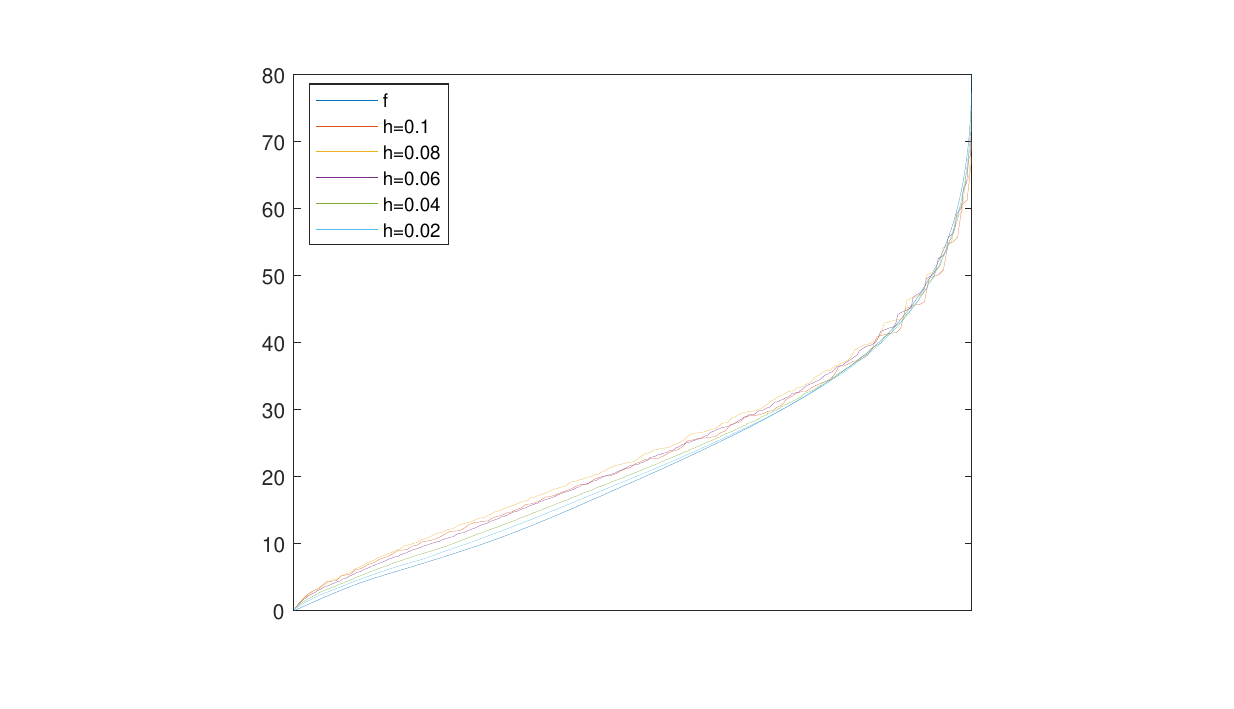} \vskip -0.5cm
  \caption{Eigenvalues distribution of $A_n$ for different values of $h$  together with the sampling of $\mathrm{f}(x,y,\theta_1,\theta_2)=a(x,y)(4-2\cos \theta_1-2\cos \theta_2)$, $a(x,y)= (10 + x^2 + 2y^2+ \sin^2(x+y)) / ( 1 + x^2 + y^2)$.}
\label{FDcoeff_A_50}
\end{figure}


\section{Conclusions and open problems}\label{sec-final}
\noindent

In this paper, we defined a wide class of matrix-sequences related to any unbounded domain $\Omega$ of finite measure satisfying $\Ld(\partial \Omega)=0$, to which we have associated a canonical symbol over $\Omega \times [-\pi,\pi]^d$. This symbol gives to the defined class of matrix-sequences a $*$-algebra structure over $\mathbb{C}$. We obtained this class with a limit procedure over reduced GLT sequences associated with regular exhaustions of bounded domains. A number of questions arise, both from the theoretical viewpoint and for the applications. Here, we briefly analyze some of the most important ones.

\begin{enumerate}
    \item[(a)] Originally, GLT sequences derive from the algebraic closure of the space generated by Toeplitz sequences, diagonal sampling sequences and zero-distributed sequences. The same should also be possible for unbounded GLT sequences. As shown in Section \ref{sec:appl}, it is possible to define sequences of Toeplitz matrices and diagonal sampling matrices over $\Omega$ unbounded. By defining trivially the unbounded zero-distributed sequences, we have all the tools needed for an explicit construction of the uGLT $*$-algebra as the algebraic closure of the space generated by the latter, together with unbounded Toeplitz and unbounded diagonal sampling sequences. As in the classical GLT case (see \cite{applications-axioms,glt-laa}), this should produce precisely the same class of matrix-sequences of the present work. However, in the classical case, the construction of GLT via algebraic closure is way more involved than other formulations which use a.c.s. convergence and, in fact, we preferred to follow this second approach, exploiting the g.a.c.s. machinery. Nevertheless, equivalent constructions would certainly be of interest and could highlight hidden properties.
    \item[(b)] In the present note we focused on the case of unbounded domains $\Omega$ with measure $\Ld(\Omega) < +\infty$. This is a crucial assumption that we think is not possible to remove as it is. In the case of infinite measure, the geometrically natural choice of the grid (and consequently of the sequence of dimensions) would produce linear operators acting on infinite dimensional spaces at each step $\bfn$, so that the discretization procedure would not simplify much the picture of the original problem. On the other hand, the assumption $\Ld(\partial\Omega)=0$ ensures that extending unbounded matrix-sequences to larger domains would produce sequences with comparable dimensions. Indeed, when working in a fixed open domain with finite measure with no need to extend the objects further outside of it, this assumption can be removed with no additional work. However, it may be explored the possibility of taking into account the infinite measure of the domain considering non-uniform grids which are more sparse and definitely vanishing as we approach infinity. This is quite common when dealing with differential equations over unbounded domains: we may expect that the solution has a relevant decay at infinity that allows us to choose the grid of the discretization accordingly (see, for example, \cite{ABKST26,dT14}).
    \item[(c)] The classical GLT theory is a powerful tool for analyzing the asymptotic spectral properties of discretized PDEs, but it generally lacks flexibility. The present work goes in the direction of developing more flexible objects, which can adapt to various situations and different settings. The concept of g.a.c.s., however, still shows some limitations, since it only allows to approximate a matrix-sequence with sequences of smaller size which have the same dimension only in the limit. We think that it is possible to remove this assumption and consider a more general concept of approximation for matrix-sequences.
    \item[(d)] Quite related, and in fact complementary, to the previous item would be a better understanding and a systematic study of the g.a.c.s. pseudo-metric, which would probably allow us to use the uGLT framework to study PDEs with moving domains with relative ease. For example, assume that we have a sequence of domains $\{\Omega_t\}_t$, such that $\mathds{1}_{\Omega_t}\to\mathds{1}_{\Omega}$ in $L^1(\mathbb{R}^d)$, and for each domain we have an uGLT $ \{A_{\bfn,t}\}_\bfn\sim_{\mathrm{GLT}}^{\Omega_t}f_t$ such that $f_t(\bfx,\bftheta)\mathds{1}_{\Omega_t}(\bfx)\to f(\bfx,\bftheta)\mathds{1}_{\Omega}(\bfx)$ in measure. In such a case, one may expect that this sequence of matrix-sequences converges to $\{A_\bfn\}_\bfn\sim_{\mathrm{GLT}}^{\Omega}f$. It would be interesting to specify and study this convergence, also motivated by the applications. Indeed, the authors expect that, after answering the previous questions, it should be possible to perform a precise spectral analysis of various discretizations of PDEs with moving domains. One of the most relevant examples in this direction is represented by the water waves equation (see \cite{L13}).
    \item[(e)] In the present work we studied various algebraic properties of unbounded GLT sequences, such as closure with respect to sum, product, conjugate (Theorem \ref{main_theorem}) and pseudo inverse (Theorem \ref{uGLT_isometry}). It would be interesting to extend this core set, which is still enough to conclude the isometry with measurable functions on $\Omega$ (Theorem \ref{thm_isom_final}), to additional properties satisfied by the GLT class. By way of example, it is known that if we consider a non-Hermitian perturbation $\{Y_{\bfn}\}_{\bfn}$ of a Hermitian GLT sequence $\{X_{\bfn}\}_{\bfn}\sim_{\text{GLT}}\kappa$, then the sequence $\{X_{\bfn}+Y_{\bfn}\}_{\bfn}\sim_{\lambda} \kappa$ as long as $\|Y_{\bfn}\|_{F} = o(\sqrt{N(\bfn)})$ (see \cite{non-herm}). It would certainly be of interest to extend this and other results to our more general setting. 
    \item[(f)] As last point of our discussion, and a more general one than the others since it involves also the original GLT theory and all the following developments, the equivalence between measurable functions and uGLT sequences proved in Theorem \ref{uGLT_isometry} still needs more in-depth analysis. Indeed, among measurable functions, there are some special classes, like $L^p$ and Sobolev and it would be interesting to study the corresponding matrix-sequences in the uGLT class. In the case of classical GLT, the subspaces corresponding to symbols in $L^p$ are studied in \cite{KRS22}, where they select specific representatives from the a.c.s. equivalence class. In the uGLT framework, the authors expect that a similar strategy may be applied, yielding precise quantitative estimates on the growth and on the distribution of the singular values of (privileged representatives of) $L^p$ and Sobolev classes.

\end{enumerate}

{}

\end{document}